\documentclass{article}
\usepackage{amsmath}
\usepackage{amssymb}
\usepackage{amsthm}
\usepackage{graphicx}
\usepackage{cite}
\usepackage[arrow,curve,matrix,arc,2cell]{xy}
\UseAllTwocells
\DeclareFontFamily{U}{rsfs}{} \DeclareFontShape{U}{rsfs}{n}{it}{<->
rsfs10}{} \DeclareSymbolFont{mscr}{U}{rsfs}{n}{it}
\DeclareSymbolFontAlphabet{\scr}{mscr}
\def\mathscr{\scr}
\begin{document}
\def\e#1\e{\begin{equation}#1\end{equation}}
\def\ea#1\ea{\begin{align}#1\end{align}}
\def\eq#1{{\rm(\ref{#1})}}
\theoremstyle{plain}
\newtheorem{thm}{Theorem}[section]
\newtheorem{prop}[thm]{Proposition}
\newtheorem{lem}[thm]{Lemma}
\newtheorem{cor}[thm]{Corollary}
\newtheorem{conj}[thm]{Conjecture}
\newtheorem{quest}[thm]{Question}
\theoremstyle{definition}
\newtheorem{dfn}[thm]{Definition}
\newtheorem{ex}[thm]{Example}
\newtheorem{rem}[thm]{Remark}
\numberwithin{equation}{section}
\def\dim{\mathop{\rm dim}\nolimits}
\def\supp{\mathop{\rm supp}\nolimits}
\def\cosupp{\mathop{\rm cosupp}\nolimits}
\def\id{\mathop{\rm id}\nolimits}
\def\Hess{\mathop{\rm Hess}\nolimits}
\def\Sym{\mathop{\rm Sym}\nolimits}
\def\Crit{\mathop{\rm Crit}}
\def\an{{\rm an}}
\def\stk{{\rm st}}
\def\stm{{{\rm st},\hat\mu}}
\def\nai{{\text{\rm na\"\i}}}
\def\Aut{\mathop{\rm Aut}}
\def\Ext{\mathop{\rm Ext}\nolimits}
\def\coh{\mathop{\rm coh}\nolimits}
\def\Hom{\mathop{\rm Hom}\nolimits}
\def\Perv{\mathop{\rm Perv}\nolimits}
\def\Sch{\mathop{\rm Sch}\nolimits}
\def\Iso{{\mathop{\rm Iso}\nolimits}}
\def\fIso{{\mathop{\mathfrak{Iso}}\nolimits}}
\def\St{\mathop{\bf St}\nolimits}
\def\Art{\mathop{\rm Art}\nolimits}
\def\dSt{\mathop{\bf dSt}\nolimits}
\def\dArt{\mathop{\bf dArt}\nolimits}
\def\dSch{\mathop{\bf dSch}\nolimits}
\def\Sh{{\mathop{\rm Sh}\nolimits}}
\def\MF{\mathop{\rm MF}\nolimits}
\def\red{{\rm red}}
\def\DM{\mathop{\rm DM}\nolimits}
\def\GL{\mathop{\rm GL}\nolimits}
\def\Spec{\mathop{\rm Spec}\nolimits}
\def\bSpec{\mathop{\bs{\rm Spec}}\nolimits}
\def\rank{\mathop{\rm rank}\nolimits}
\def\SO{\mathop{\rm SO}\nolimits}
\def\Pin{\mathop{\rm Pin}\nolimits}
\def\Spin{\mathop{\rm Spin}\nolimits}
\def\qcoh{{\mathop{\rm qcoh}}}
\def\Kalg{{\text{\rm $\K$-alg}}}
\def\Kvect{{\text{\rm $\K$-vect}}}
\def\bs{\boldsymbol}
\def\ge{\geqslant}
\def\le{\leqslant\nobreak}
\def\bA{{\mathbin{\mathbb A}}}
\def\bD{{\mathbin{\mathbb D}}}
\def\bH{{\mathbin{\mathbb H}}}
\def\bL{{\mathbin{\mathbb L}}}
\def\bT{{\mathbin{\mathbb T}}}
\def\bU{{\bs U}}
\def\bV{{\bs V}}
\def\bW{{\bs W}}
\def\bX{{\bs X}}
\def\bY{{\bs Y}}
\def\bZ{{\bs Z}}
\def\cA{{\mathbin{\cal A}}}
\def\cB{{\mathbin{\cal B}}}
\def\cC{{\mathbin{\cal C}}}
\def\cD{{\mathbin{\scr D}}}
\def\cE{{\mathbin{\cal E}}}
\def\cH{{\mathbin{\cal H}}}
\def\cL{{\mathbin{\cal L}}}
\def\cM{{\mathbin{\cal M}}}
\def\oM{{\mathbin{\smash{\,\,\overline{\!\!\mathcal M\!}\,}}}}
\def\bcM{{\mathbin{\bs{\cal M}}}}
\def\O{{\mathbin{\cal O}}}
\def\cP{{\mathbin{\cal P}}}
\def\cQ{{\mathbin{\cal Q}}}
\def\PV{{\mathbin{\cal{PV}}}}
\def\bG{{\mathbin{\mathbb G}}}
\def\C{{\mathbin{\mathbb C}}}
\def\K{{\mathbin{\mathbb K}}}
\def\N{{\mathbin{\mathbb N}}}
\def\Q{{\mathbin{\mathbb Q}}}
\def\Z{{\mathbin{\mathbb Z}}}
\def\al{\alpha}
\def\be{\beta}
\def\ga{\gamma}
\def\de{\delta}
\def\io{\iota}
\def\ep{\epsilon}
\def\la{\lambda}
\def\ka{\kappa}
\def\th{\theta}
\def\ze{\zeta}
\def\up{\upsilon}
\def\vp{\varphi}
\def\si{\sigma}
\def\om{\omega}
\def\De{\Delta}
\def\La{\Lambda}
\def\Th{\Theta}
\def\Om{\Omega}
\def\Ga{\Gamma}
\def\Si{\Sigma}
\def\Tau{{\rm T}}
\def\Up{\Upsilon}
\def\pd{\partial}
\def\ts{\textstyle}
\def\st{\scriptstyle}
\def\sst{\scriptscriptstyle}
\def\sm{\setminus}
\def\bu{\bullet}
\def\op{\oplus}
\def\ot{\otimes}
\def\bigot{\bigotimes}
\def\boxt{\boxtimes}
\def\ov{\overline}
\def\ul{\underline}
\def\bigop{\bigoplus}
\def\iy{\infty}
\def\es{\emptyset}
\def\ab{\allowbreak}
\def\ra{\rightarrow}
\def\rra{\rightrightarrows}
\def\Ra{\Rightarrow}
\def\longra{\longrightarrow}
\def\Longra{\Longrightarrow}
\def\hookra{\hookrightarrow}
\def\bs{\boldsymbol}
\def\t{\times}
\def\ci{\circ}
\def\ti{\tilde}
\def\d{{\rm d}}
\def\od{\odot}
\def\bd{\boxdot}
\def\md#1{\vert #1 \vert}
\def\ha{{\ts\frac{1}{2}}}
\def\cS{{\mathbin{\cal S}}}
\def\cSz{{\mathbin{\cal S}\kern -0.1em}^{\kern .1em 0}}
\def\dd{{\rm d}_{dR}}
\def\cR{{\mathbin{\cal R}}}
\def\cN{{\mathbin{\cal N}}}
\def\Lm{{\mathop{\text{\rm Lis-me}}\nolimits}}
\def\Le{{\mathop{\text{\rm Lis-\'et}}\nolimits}}
\def\Lian{{\mathop{\text{\rm Lis-an}}\nolimits}}
\def\otL{{\kern .1em\mathop{\otimes}\limits^{\sst L}\kern .1em}}
\def\boxtL{{\kern .2em\mathop{\boxtimes}\limits^{\sst L}\kern .2em}}
\def\Lmm{{\mathop{\text{\rm LM-\'em}}\nolimits}}
\def\Mod{\mathop{\rm Mod}\nolimits}
\title{A `Darboux Theorem' for shifted \\ symplectic structures on
derived \\ Artin stacks, with applications}
\author{Oren Ben-Bassat, Christopher Brav, Vittoria Bussi \\
and Dominic Joyce}
\date{}
\maketitle

\begin{abstract} This is the fifth in a series of papers
\cite{Joyc2}, \cite{BBJ}, \cite{BBDJS}, \cite{BJM} on the
`$k$-shifted symplectic derived algebraic geometry' of Pantev,
To\"en, Vaqui\'e and Vezzosi \cite{PTVV}. This paper extends the
results of \cite{BBJ}, \cite{BBDJS}, \cite{BJM} from (derived)
schemes to (derived) Artin stacks. We prove four main results:
\smallskip

\noindent{\bf(a)} If $(\bX,\om_\bX)$ is a $k$-shifted symplectic
derived Artin stack for $k<0$ in the sense of \cite{PTVV}, then near
each $x\in\bX$ we can find a `minimal' smooth atlas
$\bs\vp:\bU\ra\bX$ with $\bU$ an affine derived scheme, such that
$(\bU,\bs\vp^*(\om_\bX))$ may be written explicitly in coordinates
in a standard `Darboux form'.
\smallskip

\noindent{\bf(b)} If $(\bX,\om_\bX)$ is a $-1$-shifted symplectic
derived Artin stack and $X=t_0(\bX)$ the corresponding classical
Artin stack, then $X$ extends naturally to a `d-critical stack'
$(X,s)$ in the sense of \cite{Joyc2}.
\smallskip

\noindent{\bf(c)} If $(X,s)$ is an oriented d-critical stack, we can
define a natural perverse sheaf $\check P^\bu_{X,s}$ on $X$, such that
whenever $T$ is a scheme and $t:T\ra X$ is smooth of relative dimension $n$,
then $T$ is locally modelled on a critical locus
$\Crit(f:U\ra\bA^1)$ for $U$ a smooth scheme, and
$t^*(\check P^\bu_{X,s})[n]$ is locally modelled on the perverse sheaf of
vanishing cycles $\PV_{U,f}^\bu$ of $f$.
\smallskip

\noindent{\bf(d)} If $(X,s)$ is a finite type oriented d-critical
stack, we can define a natural motive $MF_{X,s}$ in a certain ring
of motives $\oM^\stm_X$ on $X,$ such that whenever $T$ is a finite
type scheme and $t:T\ra X$ is smooth of dimension $n$, then $T$ is
locally modelled on a critical locus $\Crit(f:U\ra\bA^1)$ for $U$ a
smooth scheme, and $\bL^{-n/2}\od t^*(MF_{X,s})$ is locally modelled
on the motivic vanishing cycle $MF^{{\rm mot},\phi}_{U,f}$ of $f$ in
$\oM^\stm_T$.
\smallskip

Our results will have applications to categorified and motivic
extensions of Donaldson--Thomas theory of Calabi--Yau 3-folds.
\end{abstract}

\setcounter{tocdepth}{2}
\tableofcontents

\baselineskip 11.5pt plus .5pt

\section{Introduction}
\label{sa1}

This is the fifth in a series of papers \cite{Joyc2}, \cite{BBJ},
\cite{BBDJS}, \cite{BJM} on the subject of the `$k$-shifted
symplectic derived algebraic geometry' of Pantev, To\"en, Vaqui\'e
and Vezzosi \cite{PTVV}, and its applications to generalizations of
Donaldson--Thomas theory of Calabi--Yau 3-folds, and to complex and
algebraic symplectic geometry.

Pantev et al.\ \cite{PTVV} defined notions of $k$-shifted symplectic
derived schemes and stacks $(\bX,\om)$, a new geometric structure on
derived schemes and derived stacks $\bX$ in the sense of To\"en
and Vezzosi \cite{Toen,ToVe}. They proved that any derived
moduli stack $\bs\cM$ of (complexes of) coherent sheaves on a
Calabi--Yau $m$-fold $Y$ carries a $(2-m)$-shifted symplectic
structure.

We are particularly interested in Calabi--Yau 3-folds, in which case
$k=-1$. Pantev et al.\ \cite{PTVV} also proved that the derived
critical locus $\bs\Crit(f:U\ra\bA^1)$ of a regular function $f$ on
a smooth $\K$-scheme $U$ is $-1$-shifted symplectic, and that the
derived intersection $L\cap M$ of two algebraic Lagrangian
submanifolds $L,M$ in an algebraic symplectic manifold $(S,\om)$ is
$-1$-shifted symplectic.

The first paper Joyce \cite{Joyc2} in our series defined and studied
`algebraic d-critical loci' $(X,s)$, a classical $\K$-scheme $X$
with a geometric structure $s$ which records information on how $X$
may Zariski locally be written as a classical critical locus
$\Crit(f:U\ra\bA^1)$ of a regular function $f$ on a smooth
$\K$-scheme $U$. It also discussed `d-critical stacks' $(X,s)$, a
generalization to Artin $\K$-stacks.

The second paper by Bussi, Brav and Joyce \cite{BBJ} proved a
`Darboux Theorem' for the $k$-shifted symplectic derived schemes
$(\bX,\om)$ of \cite{PTVV} when $k<0$, writing $(\bX,\om)$ Zariski
locally in a standard form, and defined a truncation functor from
$-1$-shifted symplectic derived schemes $(\bX,\om)$ to algebraic
d-critical loci $(X,s)$. By \cite{PTVV}, this implies that moduli
schemes $\cM$ of simple (complexes of) coherent sheaves on a
Calabi--Yau 3-fold $Y$ can be made into d-critical loci~$(\cM,s)$.

The third paper by Bussi, Brav, Dupont, Joyce and Szendr\H oi
\cite{BBDJS} proves that if $(X,s)$ is an algebraic d-critical locus
with an `orientation', then one can define a natural perverse sheaf
$P^\bu_{X,s}$, a $\cD$-module $D_{X,s}$, and (over $\K=\C$) a mixed
Hodge module $M_{X,s}$ over $X$, such that if $(X,s)$ is locally
modelled on $\Crit(f:U\ra\bA^1)$ then $P^\bu_{X,s}$ is locally
modelled on the perverse sheaf of vanishing cycles $\PV_{U,f}^\bu$
of $f$, and similarly for $D_{X,s},M_{X,s}$. We hope to apply this
to the categorification of Donaldson--Thomas theory of Calabi--Yau
3-folds, as in Kontsevich and Soibelman~\cite{KoSo2}.

The fourth paper by Bussi, Joyce and Meinhardt \cite{BJM} proves
that if $(X,s)$ is a finite type, oriented algebraic d-critical
locus then one can define a natural motive $MF_{X,s}$ in a ring of
motives $\oM^{\hat\mu}_X$ on $X$, such that if $(X,s)$ is locally
modelled on $\Crit(f:U\ra\bA^1)$ then $MF_{X,s}$ is locally modelled
on the `motivic vanishing cycle' $MF_{U,f}^{\rm mot,\phi}$ of $f$.
We hope to apply this to motivic Donaldson--Thomas invariants of
Calabi--Yau 3-folds, as in Kontsevich and Soibelman~\cite{KoSo1}.

The goal of this paper is to extend the results of \cite{BBJ},
\cite{BBDJS}, \cite{BJM} from $\K$-schemes to Artin $\K$-stacks,
using the notion of d-critical stack from \cite{Joyc2}. The next
four theorems summarize the main results of sections
\ref{sa2}--\ref{sa5} below, respectively:

\begin{thm} Let\/ $\K$ be an algebraically closed field of characteristic zero, $(\bX,\om_\bX)$ a $k$-shifted symplectic derived Artin $\K$-stack as in\/ {\rm\cite{PTVV}} for\/ $k<0,$ and\/ $p\in\bX(\K)$ be a $\K$-point of\/ $\bX$. Then we can construct the following~data:
\begin{itemize}
\setlength{\itemsep}{0pt}
\setlength{\parsep}{0pt}
\item[{\bf(a)}] Affine derived $\K$-schemes $\bU=\bSpec A,$ $\bV=\bSpec B,$ where $A,B$ are commutative differential graded\/ $\K$-algebras (cdgas) in degrees $\le 0,$ of an explicit `standard form' defined in {\rm\S\ref{sa23}.}
\item[{\bf(b)}] A morphism of derived stacks $\bs\vp:\bU=\bSpec A\ra\bX$ which is smooth of the minimal possible relative dimension $n=\dim H^1(\bL_\bX\vert_p).$
\item[{\bf(c)}] An inclusion $\io:B\hookra A$ of\/ $B$ as a dg-subalgebra of\/ $A,$ so that\/ $\bs i=\bSpec\io:\bU\ra\bV$ is a morphism of derived\/ $\K$-schemes. On classical schemes, $i=t_0(\bs i):U=t_0(\bU)\ra V=t_0(\bV)$ is an isomorphism.
\item[{\bf(d)}] A\/ $\K$-point\/ $\ti p\in\Spec H^0(A)$ with\/ $\bs\vp(\ti p)\!=\!p,$ such that the `standard form' cdgas $A,B$ have the minimal possible numbers of generators $\dim H^j(\bL_\bU\vert_{\ti p}),\ab\dim H^j(\bL_\bV\vert_{\bs i(\ti p)})$ in each degree $j=0,-1,\ldots,k,k-1$.
\item[{\bf(e)}] An equivalence of relative (co)tangent complexes $\bL_{\bU/\bV} \simeq \bT_{\bU/\bX}[1-k]$. Hence $\bL_{\bU/\bV}$ is a vector bundle of rank $n$ in degree $k-1$.
\item[{\bf(f)}] A $k$-shifted symplectic structure $\om_B=(\om^0_B,0,\ldots)$ on $\bV=\bSpec B$ which is in `Darboux form' in the sense of\/
{\rm\cite[\S 5]{BBJ}} and\/ {\rm\S\ref{sa24},} with\/ $\bs\vp^*(\om_\bX)\sim \bs i^*(\om_B)$ in $k$-shifted closed\/ $2$-forms on $\bU$.
\end{itemize}

For example, if\/ $k=-2d-1$ for\/ $d=0,1,\ldots$ then the `standard form' and `Darboux form' conditions above mean the following. The degree $0$ part\/ $B^0$ of\/ $B$ is a smooth\/ $\K$-algebra of dimension $m_0,$ and we are given $x^0_1,\ldots,x^0_{m_0}\in B^0$ such that\/ $(x^0_1,\ldots,x^0_{m_0})$ are \'etale coordinates on all of\/ $V(0)=\Spec B^0$. As a graded commutative algebra, $B$ is freely
generated over $B^0$ by variables
\begin{align*}
&x_1^{-i},\ldots,x^{-i}_{m_i} &&\text{in degree $-i$ for
$i=1,\ldots,d,$} \\
&y_1^{i-2d-1},\ldots,y^{i-2d-1}_{m_i} &&\text{in degree $i-2d-1$
for $i=0,1,\ldots,d$.}
\end{align*}
We have $\om^0_B=\sum_{i=0}^d\sum_{j=1}^{m_i}\dd y^{i-2d-1}_j\,\dd x^{-i}_j$ in $(\La^2\Om^1_B)^{-2d-1}$. The differential\/ $\d$ on the cdga $B$ is $\d b=\{H,b\}$ for $b\in B,$ where $\{\,,\,\}:B\t B\ra B$ is the Poisson bracket defined using the inverse of\/ $\om^0_B,$ and\/ $H\in B^{-2d}$ is a Hamiltonian function satisfying the classical master equation $\{H,H\}=0$. Also $B\subset A,$ and\/ $A$ is freely generated as a graded commutative algebra over $B$ by additional variables $w_1^{-2d-2},\ldots,w^{-2d-2}_n$ in degree $-2d-2.$
\label{sa1thm1}
\end{thm}

Theorem \ref{sa1thm1} says that given a $k$-shifted derived Artin
stack $(\bX,\om_\bX)$ for $k<0$, near each $p\in\bX(\K)$ we can find a
smooth atlas $\bs\vp:\bU\ra\bX$ with $\bU=\bSpec A$ an affine
derived scheme, such that $(\bU,\bs\vp^*(\om_\bX))$ is in a standard
`Darboux form'. Although $(\bU,\bs\vp^*(\om_\bX))$ is not
$k$-shifted symplectic, as $\bs\vp^*(\om_\bX)$ is not nondegenerate,
we can build from $(\bU,\bs\vp^*(\om_\bX))$ in a natural way a `Darboux form' $k$-shifted symplectic derived scheme $(\bV,\om_B)$, which is equivalent to $(\bU,\bs\vp^*(\om_\bX))$ except in degree $k-1$.

\begin{thm} Let\/ $(\bX,\om_\bX)$ be a $-1$-shifted symplectic
derived Artin $\K$-stack in the sense of\/ {\rm\cite{PTVV}} over
$\K$ algebraically closed of characteristic zero, and\/ $X=t_0(\bX)$
the corresponding classical Artin $\K$-stack. Then $X$ extends
naturally to a d-critical stack\/ $(X,s)$ in the sense of\/
{\rm\cite{Joyc2}}. If\/ $T$ is a $\K$-scheme and\/ $t:T\ra X$ a
smooth\/ $1$-morphism, this gives a d-critical structure $s(T,t)$ on
$T$ making $(T,s(T,t))$ into an algebraic d-critical locus, in the
sense of\/~{\rm\cite{Joyc2}}.
\label{sa1thm2}
\end{thm}

Theorem \ref{sa1thm2} implies that Artin moduli stacks $\cM$ of
(complexes of) coherent sheaves on a Calabi--Yau 3-fold $Y$ extend
naturally to d-critical stacks~$(\cM,s)$.

\begin{thm} Let\/ $(X,s)$ be an oriented d-critical stack over an algebraically closed field\/ $\K$ with\/ $\mathop{\rm char}\K\ne 2$. Fix a theory of perverse sheaves or $\cD$-modules over $\K$-schemes and Artin $\K$-stacks, for instance Laszlo and Olsson's $l$-adic perverse sheaves {\rm\cite{LaOl1,LaOl2,LaOl3}}. Then there is a natural perverse sheaf or $\cD$-module\/ $\check P^\bu_{X,s}$ on $X$ with Verdier duality and monodromy isomorphisms
\begin{equation*}
\smash{\Si_{X,s}:\check P^\bu_{X,s}\longra \bD_X(\check P^\bu_{X,s}),\qquad
\Tau_{X,s}:\check P^\bu_{X,s}\longra \check P^\bu_{X,s},}
\end{equation*}
such that if\/ $T$ is a $\K$-scheme and\/ $t:T\ra X$ a $1$-morphism
smooth of relative dimension $n,$ then
$t^*(\check P^\bu_{X,s})[n],t^*(\Si_{X,s})[n],t^*(\Tau_{X,s})[n]$ are
isomorphic to the perverse sheaf or $\cD$-module $P_{T,s(T,t)}^\bu$
on the oriented algebraic d-critical locus $(T,s(T,t))$ defined in
{\rm\cite[\S 6]{BBDJS},} and its Verdier duality and monodromy
isomorphisms $\Si_{T,s(T,t)},\Tau_{T,s(T,t)}$. So in particular,
if\/ $(T,s(T,t))$ is locally modelled on a critical locus
$\Crit(f:U\ra\bA^1)$ for $U$ a smooth\/ $\K$-scheme, then
$t^*(\check P^\bu_{X,s})[n]$ is locally modelled on the perverse sheaf or\/
$\cD$-module of vanishing cycles of\/~$f$.
\label{sa1thm3}
\end{thm}

\begin{thm} Let\/ $(X,s)$ be an oriented d-critical stack over $\K$
algebraically closed of characteristic zero, with\/ $X$ of finite
type and locally a global quotient. Then there exists a unique
motive $MF_{X,s}$ in a certain ring $\oM^\stm_X$ of\/
$\hat\mu$-equivariant motives on $X,$ such that if\/ $T$ is a finite type
$\K$-scheme and\/ $t:T\ra X$ is smooth of relative dimension $n,$ so
that\/ $(T,s(T,t))$ is an oriented algebraic d-critical locus over
$\K,$ then
\begin{equation*}
\smash{t^*\bigl(MF_{X,s}\bigr)=\bL^{n/2}\od MF_{T,s(T,t)} \quad\text{in $\oM^\stm_T,$}}
\end{equation*}
where $MF_{T,s(T,t)}\in\oM^\stm_T$ is as in\/ {\rm\cite[\S 5]{BJM}}.
So in particular, if\/ $(T,s(T,t))$ is locally modelled on
$\Crit(f:U\ra\bA^1)$ for\/ $U$ a smooth\/ $\K$-scheme, then
$\bL^{-n/2}\od t^*\bigl(MF_{X,s}\bigr)$ is locally modelled on the
motivic vanishing cycle $MF_{U,f}^{\rm mot,\phi}$ of\/~$f$.
\label{sa1thm4}
\end{thm}

We expect that Theorems \ref{sa1thm3} and \ref{sa1thm4} will have
applications in categorified and motivic extensions of
Donaldson--Thomas theory of Calabi--Yau 3-folds, as in Kontsevich
and Soibelman~\cite{KoSo1,KoSo2}.
\medskip

\noindent{\bf Conventions and notation.} Throughout $\K$ will be an
algebraically closed field with  $\mathop{\rm char}\K=0$, except
that we allow  $\K$ algebraically closed with $\mathop{\rm
char}\K\ne 2$ in \S\ref{sa4}. Classical $\K$-schemes and Artin
$\K$-stacks will be written $W,X,Y,Z,\ldots,$ and derived
$\K$-schemes and derived Artin $\K$-stacks in bold
as~$\bW,\bX,\bY,\bZ,\ldots.$

Basic references for $\K$-schemes are Hartshorne \cite{Hart}, for
Artin $\K$-stacks Laumon and Moret-Bailly \cite{LaMo}, and for
derived $\K$-schemes and derived Artin $\K$-stacks To\"en and
Vezzosi~\cite{Toen,ToVe}.

All (classical) $\K$-schemes and Artin $\K$-stacks $X$ are assumed
locally of finite type, except in \S\ref{sa5} when we assume they
are of finite type. All derived $\K$-schemes and derived $\K$-stacks
$\bX$ are assumed to be locally finitely presented. We write
$\Sch_\K$ for the category of $\K$-schemes, $\Art_\K$ for the
2-category of Artin $\K$-stacks, $\dSch_\K$ for the $\iy$-category
of derived $\K$-schemes, and $\dArt_\K$ for the $\iy$-category of
derived Artin $\K$-stacks, and $t_0:\dSch_\K\ra\Sch_\K$,
$t_0:\dArt_\K\ra\Art_\K$ for the classical truncation functors.
Other notation generally follows the prequels
\cite{BBDJS,BBJ,BJM,Joyc2} to this paper.
\medskip

\noindent{\bf Acknowledgements.} We would like to thank Tom
Bridgeland, Sven Meinhardt, Bal\'azs Szendr\H oi, and Bertrand
To\"en for helpful conversations, and a referee for careful proofreading and useful comments. This research was supported by
EPSRC Programme Grant EP/I033343/1. The first author acknowledges
the support of the European Commission under the Marie Curie
Programme which awarded him an IEF grant. The contents of this
article reflect the views of the authors and not the views of the
European Commission.

\section{Local models for atlases of shifted symplectic derived
stacks}
\label{sa2}\vskip -5pt\relax

Sections \ref{sa21} and \ref{sa22} give background on derived
algebraic geometry \cite{Toen,ToVe} and
Pantev--To\"en--Vaqui\'e--Vezzosi's shifted symplectic structures
\cite{PTVV}, and \S\ref{sa23}--\S\ref{sa24} recall the main
definitions of \cite[\S 4--\S 5]{BBJ}. Then
\S\ref{sa25}--\S\ref{sa27}, the new material in this section,
generalize \S\ref{sa23}--\S\ref{sa24} to derived Artin stacks.\vskip -5pt\relax

\enlargethispage{10pt}

\subsection{Derived algebraic geometry}
\label{sa21}

We work in the context of To\"en and Vezzosi's derived algebraic
geometry \cite{Toen,ToVe}, and Pantev, To\"en, Vaqui\'e and
Vezzosi's theory of $k$-shifted symplectic structures on derived
schemes and stacks \cite{PTVV}. This is a complex subject, and we
give only a brief sketch to fix notation. A longer explanation
suited to our needs can be found in~\cite[\S 2--\S 3]{BBJ}.

Fix an algebraically closed base field $\K$, of characteristic zero.
To\"en and Vezzosi define the $\iy$-category $\dSt_\K$ of {\it
derived\/ $\K$-stacks\/} (or $D^-$-{\it stacks\/})
\cite[Def.~2.2.2.14]{ToVe}, \cite[Def.~4.2]{Toen}. All derived
$\K$-stacks $\bX$ in this paper are assumed to be {\it locally
finitely presented}. There is a {\it spectrum functor\/}
\begin{equation*}
\smash{\bSpec:\{\text{commutative differential graded $\K$-algebras, degrees $\le 0$}\}\longra\dSt_\K.}
\end{equation*}
All cdgas in this paper will be in degrees $\le 0$. A derived $\K$-stack $\bX$ is called an {\it affine derived\/ $\K$-scheme\/} if $\bX$ is equivalent in $\dSt_\K$ to $\bSpec A$ for some cdga $A$ over $\K$. As in \cite[\S 4.2]{Toen}, a derived $\K$-stack $\bX$ is called a {\it derived\/ $\K$-scheme\/} if it may be covered by Zariski open $\bY\subseteq\bX$ with $\bY$ an affine derived $\K$-scheme. Write $\dSch_\K$ for the full $\iy$-subcategory of derived $\K$-schemes in $\dSt_\K$.

We call a derived $\K$-stack $\bX$ a {\it derived Artin\/
$\K$-stack\/} if it is {\it $m$-geometric\/} for some $m$ \cite[Def.~1.3.3.1]{ToVe} and the underlying classical stack is $1$-truncated (that is, just a stack, not a higher stack). Any such $\bX$ admits a smooth surjective morphism $\bs\vp:\bU\ra\bX$, an {\it atlas}, with $\bU$ a derived $\K$-scheme. Write $\dArt_\K$ for the full $\iy$-subcategory of derived Artin $\K$-stacks in $\dSt_\K$. Then $\dSch_\K\!\subset\!\dArt_\K\!\subset\!\dSt_\K$.

Write $\Sch_\K$ for the category of $\K$-schemes $X$, and $\Art_\K$
for the 2-category of Artin $\K$-stacks $X$. By an abuse of notation we regard $\Sch_\K$ as a discrete
2-subcategory of $\Art_\K$, so that $\Sch_\K\subset\Art_\K$. As in
\cite[Prop.~2.1.2.1]{ToVe}, there is an {\it inclusion functor\/}
$i:\Art_\K\ra\dArt_\K$ mapping $\Sch_\K\ra\dSch_\K$, and a {\it
classical truncation functor\/} $t_0:\dArt_\K\ra\Art_\K$ mapping
$\dSch_\K\ra\Sch_\K$.

A derived Artin $\K$-stack $\bX$ has a {\it cotangent complex\/}
$\bL_\bX$ of finite cohomological amplitude $[-m,1]$ and a dual {\it tangent complex\/} $\bT_\bX$ \cite[\S
1.4]{ToVe}, \cite[\S 4.2.4--\S 4.2.5]{Toen} in a stable
$\infty$-category $L_\qcoh(\bX)$ defined in \cite[\S 3.1.7, \S
4.2.4]{Toen}. When $X$ is a classical scheme or stack, then the
homotopy category of $L_\qcoh(\bX)$ is nothing but the triangulated
category $D_{\qcoh}(X)$. These have the usual properties of
(co)tangent complexes. For instance, if $\bs f:\bX\ra\bY$ is a
morphism in $\dArt_\K$ there is a distinguished triangle
\e
\smash{\xymatrix@C=30pt{ \bs f^*(\bL_\bY) \ar[r]^(0.55){\bL_{\bs f}} &
\bL_\bX \ar[r] & \bL_{\bX/\bY} \ar[r] & \bs f^*(\bL_\bY)[1], }}
\label{sa2eq1}
\e
where $\bL_{\bX/\bY}$ is the {\it relative cotangent complex\/} of
$\bs f$. Here $\bs f$ is smooth of relative dimension $n$ if and only if $\bL_{\bX/\bY}$ is locally free of rank $n$, and $\bs f$ is \'etale if and only if $\bL_{\bX/\bY}=0$.

\subsection{Shifted symplectic derived schemes and derived stacks}
\label{sa22}

Let $\bX$ be a derived stack. Pantev, To\"en, Vaqui\'e and Vezzosi
\cite{PTVV} defined $k$-{\it shifted\/ $p$-forms}, $k$-{\it shifted
closed\/ $p$-forms}, and $k$-{\it shifted symplectic structures\/}
on $\bX$, for $k\in\Z$ and $p\ge 0$. One first defines these notions
on derived affine schemes and then defines the general notions by
smooth descent. Since our main theorems are statements about the
local structure of derived stacks endowed with shifted symplectic
forms, it suffices for us to describe the affine case. The basic
idea is this:
\begin{itemize}
\setlength{\itemsep}{0pt}
\setlength{\parsep}{0pt}
\item[(a)] Define the exterior powers $\La^p\bL_\bX$ in $L_\qcoh(\bX)$ for
$p=0,1,\ldots.$ Regard $\La^p\bL_\bX$ as a complex, with
differential $\d$:
\begin{equation*}
\smash{\xymatrix@C=25pt{ \cdots \ar[r]^(0.35)\d & (\La^p\bL_\bX)^{k-1}
\ar[r]^\d & (\La^p\bL_\bX)^k \ar[r]^\d & (\La^p\bL_\bX)^{k+1}
\ar[r]^(0.65)\d & \cdots.}}
\end{equation*}
Then a $k$-{\it shifted\/ $p$-form}, or $p$-{\it form of
degree\/} $k$, is an element $\om^0$ of $(\La^p\bL_\bX)^k$ with
$\d\om^0=0$. Mostly we are interested in the cohomology class
$[\om^0]\in H^k(\La^p\bL_\bX)$.
\item[(b)] There are {\it de Rham differentials\/}
$\dd:\La^p\bL_\bX\ra \La^{p+1}\bL_\bX$ with $\dd\ci\dd=\d\ci\dd+\dd\ci\d=0$. Then a $k$-{\it shifted closed\/ $p$-form}, or {\it closed\/
$p$-form of degree\/} $k$, is a sequence
$\om=(\om^0,\om^1,\om^2,\ldots)$ with $\om^i$ in
$(\La^{p+i}\bL_\bX)^{k-i}$ for $i\ge 0$, satisfying $\d\om^0=0$ and $\dd\om^i+\d\om^{i+1}=0$ for~$i=0,1,\ldots.$

That is, $\om=(\om^0,\om^1,\om^2,\ldots)$ is a $k$-cycle in the
negative cyclic complex
\begin{equation*}
\smash{\bigl(\bigl(\ts\prod_{i=0}^\iy(\La^{p+i}\bL_\bX)^{k-i}\bigr)_{k\in\Z},\d+\dd\bigr).}
\end{equation*}
Mostly we are interested in the cohomology class
$[\om]=[\om^0,\om^1,\ldots]$ in the cohomology of this complex. We will
write $\om\sim\om'$ if $\om,\om'$ are $k$-shifted closed
$p$-forms with the same cohomology class $[\om]=[\om']$. There
is a map $(\om^0,\om^1,\om^2,\ldots)\mapsto \om^0$ from
$k$-shifted closed $p$-forms to $k$-shifted $p$-forms.
\item[(c)] A $k$-{\it shifted symplectic structure\/} on $\bX$ is a
$k$-shifted closed $2$-form $(\om^0,\ldots)$ on $\bX$ whose
induced morphism $\om^0\cdot:\bT_\bX\ra\bL_\bX[k]$ is an
equivalence.
\end{itemize}

If a derived $\K$-scheme $\bX$ has a 0-shifted symplectic structure
then $\bX$ is a smooth $\K$-scheme $X$ with a classical symplectic
structure. Pantev et al.\ \cite{PTVV} construct $k$-shifted
symplectic structures on several classes of derived moduli stacks.
If $Y$ is a Calabi--Yau $m$-fold and $\bs\cM$ a derived moduli stack
of coherent sheaves or perfect complexes on $Y$, then $\bs\cM$ has a
$(2-m)$-shifted symplectic structure. We are particularly interested
in the case $m=3$, so $k=-1$.

\subsection{`Standard form' affine derived schemes}
\label{sa23}

The next definition summarizes \cite[Ex.~2.8, Def.~2.9 \& Def.
2.13]{BBJ}.

\begin{dfn} We will explain how to inductively construct a
sequence of commutative differential graded algebras (cdgas)
$A(0),A(1),\ldots,A(n)=A$ over $\K$ with $A(0)$ a smooth
$\K$-algebra and $A(k)$ having underlying commutative
graded algebra free over $A(0)$ on generators of degrees
$-1,\ldots,-k$. We will call $A$ a {\it standard form\/} cdga. We
will write $\bU(i)=\bSpec A(i)$ for $i=0,\ldots,n$ and
$\bU=\bU(n)=\bSpec A$ for the corresponding affine derived
$\K$-schemes, where $\bU(0)=U(0)$ is a smooth classical $\K$-scheme,
which contains $\Spec H^0(A)$ as a closed $\K$-subscheme.

Begin with a commutative algebra $A(0)$ smooth over $\K$. Choose a
free $A(0)$-module $M^{-1}$ of finite rank together with a map
$\pi^{-1} :M^{-1} \ra A(0)$. Define a cdga $A(1)$ whose underlying
commutative graded algebra is free over $A(0)$ with generators given
by $M^{-1}$ in degree $-1$ and with differential $\d$ determined by
the map $\pi^{-1}: M^{-1} \ra A(0)$. By construction, we have
$H^0(A(1)) = A(0)/I$, where the ideal $I\subseteq A(0)$ is the image
of the map~$\pi^{-1}:M^{-1}\ra A(0)$.

Note that $A(1)$ fits in a homotopy pushout diagram of cdgas
\begin{equation*}
\xymatrix@R=11pt@C=90pt{*+[r]{\Sym_{A(0)}(M^{-1})} \ar[r]_(0.7){0_*}
\ar[d]^{\pi^{-1}_*} & *+[l]{A(0)}
\ar[d] \\ *+[r]{A(0)} \ar[r]^(0.7){f^{-1}} & *+[l]{A(1),\!\!{}} }
\end{equation*}
with morphisms $\pi^{-1}_*,0_*$ induced by $\pi^{-1},0:M^{-1}\ra
A(0)$. Write $f^{-1}:A(0)\ra A(1)$ for the resulting map of
algebras.

Next, choose a free $A(1)$-module $M^{-2}$ of finite rank together
with a map $\pi^{-2}: M^{-2}[1] \ra A(1)$. Define a cdga $A(2)$
whose underlying commutative graded algebra is free over $A(1)$ with
generators given by $M^{-2}$ in degree $-2$ and with differential
$\d$ determined by the map $\pi^{-2}: M^{-2}[1] \ra A(1)$. Write
$f^{-2}$ for the resulting map of algebras~$A(1)\ra A(2)$.

As the underlying commutative graded algebra of $A(1)$ was free over
$A(0)$ on generators of degree $-1$, the underlying commutative
graded algebra of $A(2)$ is free over $A(0)$ on generators of
degrees $-1,-2$. Since $A(2)$ is obtained from $A(1)$ by adding
generators in degree $-2$, we have $H^0(A(1))\cong H^0(A(2)) \cong
A(0)/I$.

Note that $A(2)$ fits in a homotopy pushout diagram of cdgas
\begin{equation*}
\xymatrix@R=11pt@C=90pt{*+[r]{\Sym_{A(1)}(M^{-2}[1])} \ar[r]_(0.7){0_*}
\ar[d]^{\pi^{-2}_*} & *+[l]{A(1)} \ar[d] \\
*+[r]{A(1)} \ar[r]^(0.7){f^{-2}} & *+[l]{A(2),\!\!{}} }
\end{equation*}
with morphisms $\pi^{-2}_*,0_*$ induced by $\pi^{-2},0:M^{-2}[1]\ra
A(1)$.

Continuing in this manner inductively, we define a cdga $A(n)=A$
with $A^0=A(0)$ and $H^0(A)=A(0)/I$, whose underlying commutative
graded algebra is free over $A(0)$ on generators of degrees
$-1,\ldots,-n$. We call any cdga $A$ constructed in this way a {\it
standard form\/} cdga.

If $A$ is of standard form, we will call a cdga $A'$ a {\it
localization\/} of $A$ if $A'=A\ot_{A^0}A^0[f^{-1}]$ for $f\in A^0$,
that is, $A'$ is obtained by inverting $f$ in $A$. Then $A'$ is also
of standard form, with $A^{\prime\, 0}\cong A^0[f^{-1}]$. If $p\in
\Spec H^0(A)$ with $f(p)\ne 0$, we call $A'$ a {\it localization
of\/ $A$ around \/}~$p$.

Let $A$ be a standard form cdga. We call $A$ {\it minimal\/} at $p
\in \Spec H^0(A)$ if for all $k=1,\ldots,n$ the compositions
\begin{equation*}
\smash{H^{-k}\bigl(\bL_{A(k)/A(k-1)}\bigr) \longra
H^{1-k}\bigl(\bL_{A(k-1)}\bigr)\longra
H^{1-k}\bigl(\bL_{A(k-1)/A(k-2)}\bigr)}
\end{equation*}
in the cotangent complexes restricted to $\Spec H^0(A)$ vanish
at~$p$. (For more on this point, see \cite[Prop. 2.12]{BBJ}.)
\label{sa2def1}
\end{dfn}

Here are \cite[Th.s 4.1 \& 4.2]{BBJ}. They say that any derived
scheme $\bX$ is locally modelled on $\bSpec A$ for a (minimal)
standard form cdga $A$, and give us a way to compare two such local
models $\bs f:\bSpec A\hookra\bX$, $\bs g:\bSpec B\hookra\bX$.

\begin{thm} Let\/ $\bX$ be a derived\/ $\K$-scheme, and\/
$x\in\bX$. Then there exist a standard form cdga $A$ over $\K$ which
is minimal at a point\/ $p\in\Spec H^0(A),$ in the sense of
Definition\/ {\rm\ref{sa2def1},} and a morphism $\bs f:\bU=\bSpec
A\ra\bX$ in $\dSch_\K$ which is a Zariski open inclusion with\/~$\bs
f(p)=x$.
\label{sa2thm1}
\end{thm}

\begin{thm} Let\/ $\bX$ be a derived\/ $\K$-scheme, $A,B$ be
standard form cdgas over $\K,$ and\/ $\bs f:\bSpec A\ra\bX,$ $\bs
g:\bSpec B\ra\bX$ be Zariski open inclusions in $\dSch_\K$. Suppose
$p\in\Spec H^0(A)$ and\/ $q\in\Spec H^0(B)$ with\/ $\bs f(p)=\bs
g(q)$ in $\bX$. Then there exist a standard form cdga $C$ over $\K$
which is minimal at\/ $r$ in $\Spec H^0(C)$ and morphisms of cdgas
$\al:A\ra C,$ $\be:B\ra C$ which are Zariski open inclusions, such
that\/ $\bSpec\al:r\mapsto p,$ $\bSpec\be:r\mapsto q,$ and\/ $\bs
f\ci\bSpec\al\simeq\bs g\ci\bSpec\be$ as morphisms $\bSpec C\ra\bX$
in~$\dSch_\K$.

If instead\/ $\bs f,\bs g$ are \'etale rather than Zariski open
inclusions, the same holds with\/ $\al,\be$ \'etale rather than
Zariski open inclusions.
\label{sa2thm2}
\end{thm}

One important advantage of working with derived schemes $\bU=\bSpec
A$ for $A$ a standard form cdga, is that the cotangent complex
$\bL_\bU$ and its exterior powers $\La^p\bL_\bU$ can be written
simply and explicitly in terms of $A$. As in \cite[\S 2, \S
3.3]{BBJ} the differential-graded module of {\it K\"ahler
differentials\/} $\Om^1_A$ is a model for $\bL_\bU$. If $U(0)=\Spec
A^0$ admits global \'etale coordinates $(x_1^0,\ldots,x_{m_0}^0)$,
then $\Om^1_A$ is a finitely-generated free $A$-module, generated by
$\dd x_1^{-i},\ldots,\dd x_{m_i}^{-i}$ in degree $-i$ for
$i=0,\ldots,n$, where $x_1^{-i},\ldots,x_{m_i}^{-i}$ are
$A(i-1)$-bases for the free finite rank $A(i-1)$-modules $M^{-i}$
for $i=1,\ldots,n$, in the notation of Definition~\ref{sa2def1}.

Because of this, on $\bU=\bSpec A$, the $k$-shifted (closed)
$p$-forms from \cite{PTVV} discussed in \S\ref{sa22} can be written
down explicitly in coordinates. Here is \cite[Prop.~5.7]{BBJ}. Part
(a) implies that for a $k$-shifted symplectic form
$\om=(\om^0,\ab\om^1,\ab\om^2,\ab\ldots)$ on a standard form
$\bU=\bSpec A$, up to equivalence we may take
$\om^1=\om^2=\cdots=0$, which simplifies calculations a lot. (Let us
note here that the proof of \cite[Prop.~5.7]{BBJ} uses the
interpretation of shifted symplectic forms as representing classes
in negative cyclic homology.)

\begin{prop}{\bf(a)} Let\/ $\om=(\om^0,\om^1, \om^2,\ldots)$ be a
closed\/ $2$-form of degree $k<0$ on $\bU=\bSpec A,$ for $A$ a
standard form cdga over $\K$. Then there exist\/ $\Phi\in A^{k+1}$
and\/ $\phi\in(\Om^1_A)^k$ such that\/ $\d\Phi=0$ in\/ $A^{k+2}$
and\/ $\dd \Phi+\d\phi=0$ in $(\Om^1_A)^{k+1}$ and\/ $\om\sim (\dd
\phi,0,0,\ldots)$.
\smallskip

\noindent{\bf(b)} In the case $k=-1$ in {\bf(a)} we have $\Phi\in
A^0=A(0),$ so we can consider the restriction\/ $\Phi\vert_{U^\red}$
of\/ $\Phi$ to the reduced\/ $\K$-subscheme $U^\red$ of\/
$U=t_0(\bU)=\Spec H^0(A)$. Then $\Phi\vert_{U^\red}$ is locally
constant on $U^\red,$ and we may choose $(\Phi,\phi)$ in {\bf(a)}
such that\/~$\Phi\vert_{U^\red}=0$.
\smallskip

\noindent{\bf(c)} Suppose $(\Phi,\phi)$ and\/ $(\Phi',\phi')$ are
alternative choices in part\/ {\bf(a)\rm} for fixed\/ $\om,k,\bU,A,$
where if\/ $k=-1$ we suppose $\Phi\vert_{U^\red}=0=
\Phi'\vert_{U^\red}$ as in {\bf(b)}. Then there exist\/ $\Psi\in
A^k$ and\/ $\psi\in(\Om^1_A)^{k-1}$ with\/ $\Phi-\Phi'=\d\Psi$
and\/~$\phi-\phi'=\dd\Psi+\d\psi$.

\label{sa2prop1}
\end{prop}

\subsection{`Darboux form' shifted symplectic derived schemes}
\label{sa24}

The next definition summarizes \cite[Ex.s 5.8--5.10]{BBJ}.

\begin{dfn} Fix $d=0,1,\ldots.$ We will explain how to define a class
of explicit standard form cdgas $(A,\d)=A(n)$ for $n=2d+1$ with a
very simple, explicit $k$-shifted symplectic form
$\om=(\om^0,0,0,\ldots)$ on $\bU=\bSpec A$ for $k=-2d-1$. We will
say that $A,\om$ are in {\it Darboux form}.

First choose a smooth $\K$-algebra $A(0)$ of dimension $m_0$.
Localizing $A(0)$ if necessary, we may assume that there exist
$x^0_1,\ldots,x^0_{m_0}\in A(0)$ such that $\dd x^0_1,\ldots,\dd
x^0_{m_0}$ form a basis of $\Om^1_{A(0)}$ over $A(0)$.
Geometrically, $U(0)=\Spec A(0)$ is a smooth $\K$-scheme of
dimension $m_0$, and $(x^0_1,\ldots,x^0_{m_0}):U(0)\ra\bA^{m_0}$ are
global \'etale coordinates on~$U(0)$.

Next, choose $m_1,\ldots,m_d\in\N=\{0,1,\ldots\}$. Define $A$ as a
commutative graded algebra to be the free algebra over $A(0)$
generated by variables
\e
\begin{aligned}
&x_1^{-i},\ldots,x^{-i}_{m_i} &&\text{in degree $-i$ for
$i=1,\ldots,d$, and} \\
&y_1^{i-2d-1},\ldots,y^{i-2d-1}_{m_i} &&\text{in degree $i-2d-1$ for
$i=0,1,\ldots,d$.}
\end{aligned}
\label{sa2eq2}
\e
So the upper index $i$ in $x^i_j,y^i_j$ always indicates the degree.
We will define the differential $\d$ in the cdga $(A,\d)$ later.

The spaces $(\La^p\Om^1_A)^k$ and the de Rham differential $\dd$
upon them depend only on the commutative graded algebra $A$, not on
the (not yet defined) differential $\d$. Note that $\Om^1_A$ is the
free $A$-module with basis $\dd x^{-i}_j,\dd y^{i-2d-1}_j$ for
$i=0,\ldots,d$ and $j=1,\ldots,m_i$. Define
\e
\om^0=\ts\sum_{i=0}^d\sum_{j=1}^{m_i}\dd y^{i-2d-1}_j\,\dd x^{-i}_j
\qquad \text{in $(\La^2\Om^1_A)^{-2d-1}$.}
\label{sa2eq3}
\e
Then $\dd\om^0=0$ in $(\La^3\Om^1_A)^{-2d-1}$.

Now choose $H$ in $A^{-2d}$, which we will call the
{\it Hamiltonian}, and which we require to satisfy the {\it
classical master equation\/}
\e
\sum_{i=1}^d\sum_{j=1}^{m_i}\frac{\pd H}{\pd x^{-i}_j}\,
\frac{\pd H}{\pd y^{i-2d-1}_j}=0\qquad\text{in $A^{1-2d}$.}
\label{sa2eq4}
\e
The classical master equation can be expressed invariantly as
$\{H,H\}=0$, where $\{,\}$ is a certain shifted Poisson bracket. For
more on this, consult~\cite[\S 5.7]{BBJ}.

Note that \eq{sa2eq4} is trivial when $d=0$, so that $k=-1$, as
$A^1=0$. Define the differential $\d$ on $A$ by $\d=0$ on $A(0)$, and
\e
\d x^{-i}_j =\frac{\pd H}{\pd y^{i-2d-1}_j}, \quad \d
y^{i-2d-1}_j=\frac{\pd H}{\pd x^{-i}_j},\quad\begin{subarray}{l}\ts
i=0,\ldots,d,\\[6pt] \ts j=1,\ldots,m_i.\end{subarray}
\label{sa2eq5}
\e
Then $\d\ci\d=0$, and $(A,\d)$ is a standard form cdga $A=A(n)$ as
in Definition \ref{sa2def1} for $n=2d+1$, defined using free modules
$M^{-i}=\langle x^{-i}_1,\ldots,x^{-i}_{m_i} \rangle_{A(i-1)}$ for
$i=1,\ldots,d$ and $M^{i-2d-1}=\langle y^{i-2d-1}_1,
\ldots,y^{i-2d-1}_{m_i}\rangle_{A(2d-i)}$ for $i=0,\ldots,d$.

Then $\om=(\om^0,0,0,\ldots)$ is a $k$-shifted symplectic structure
on $\bU=\bSpec A$ for $k=-2d-1$. Define $\Phi\in A^{-2d}$ and
$\phi\in(\Om^1_A)^{-2d-1}$ by $\Phi=-\frac{1}{2d+1}\,H$ and
\e
\phi=\ts\frac{1}{2d+1}\sum_{i=0}^d\sum_{j=1}^{m_i}\bigl[(2d+1-i)
y^{i-2d-1}_j\,\dd x^{-i}_j+i\,x^{-i}_j\,\dd y^{i-2d-1}_j\bigr].
\label{sa2eq6}
\e
Then $\d\Phi=0$, $\dd\Phi+\d\phi=0$, and $\om^0=\dd\phi$, as in
Proposition \ref{sa2prop1}(a). We say that $A,\om$ are in {\it
Darboux form\/} for~$k=-2d-1$.
\medskip

In \cite[Ex.s 5.9 \& 5.10]{BBJ} we give similar Darboux forms for
$k=-4d$ and $k=-4d-2$ with $d=0,1,2,\ldots.$ We will not give all
the details. In brief, when $k=-4d$, rather than \eq{sa2eq2}, $A$ is
freely generated over $A(0)$ by the variables
\begin{align*}
&x_1^{-i},\ldots,x^{-i}_{m_i} &&\text{in degree $-i$ for
$i=1,\ldots,2d-1$,} \\
&x_1^{-2d},\ldots,x^{-2d}_{m_{2d}},y_1^{-2d},\ldots,y^{-2d}_{m_{2d}}
&&\text{in degree $-2d$, and} \\
&y_1^{i-4d},\ldots,y^{i-4d}_{m_i} &&\text{in degree $i-4d$ for
$i=0,1,\ldots,2d-1$,}
\end{align*}
and $\om^0\in(\La^2\Om^1_A)^{-4d}$ with $\dd\om^0=0$ is given by
\begin{equation*}
\om^0=\ts\sum_{i=0}^{2d}\sum_{j=1}^{m_i}\dd y^{i-4d}_j\,\dd x^{-i}_j
\qquad \text{in $(\La^2\Om^1_A)^{-4d}$,}
\end{equation*}
and $\d$ on $A$ is defined as in \eq{sa2eq5} using $H\in A^{1-4d}$
satisfying the analogue of \eq{sa2eq4}. We then say that
$A,\bU=\bSpec A,\om$ are in {\it Darboux form\/} for~$k=-4d$.

Similarly, when $k=-4d-2$, $A$ is freely generated over $A(0)$ by
the variables
\begin{align*}
&x_1^{-i},\ldots,x^{-i}_{m_i} &&\text{in degree $-i$ for
$i=1,\ldots,2d$,} \\
&z_1^{-2d-1},\ldots,z^{-2d-1}_{m_{2d+1}}
&&\text{in degree $-2d-1$, and} \\
&y_1^{i-4d-2},\ldots,y^{i-4d-2}_{m_i} &&\text{in degree $i-4d-2$ for
$i=0,1,\ldots,2d$,}
\end{align*}
and $\om^0\in(\La^2\Om^1_A)^{-4d-2}$ with $\dd\om^0=0$ is given by
\begin{equation*}
\om^0=\ts\sum_{i=0}^{2d}\sum_{j=1}^{m_i}\dd y^{i-4d-2}_j\,\dd x^{-i}_j
+\sum_{j=1}^{m_{2d+1}}\dd z^{-2d-1}_j\,\dd z^{-2d-1}_j,
\end{equation*}
and $\d$ is defined as in \eq{sa2eq5} using $H\in A^{-4d-1}$
satisfying
\begin{equation*}
\sum_{i=1}^{2d}\sum_{j=1}^{m_i}\frac{\pd H}{\pd x^{-i}_j}\,
\frac{\pd H}{\pd y^{i-4d-2}_j}+\frac{1}{4}\sum_{j=1}^{m_{2d+1}}
\biggl(\frac{\pd H}{\pd z^{-2d-1}_j}\biggr)^2=0\qquad\text{in
$A^{-4d}$.}
\end{equation*}
We then say that $A,\om$ are in {\it strong Darboux form\/} for
$k=-4d-2$. There is also a {\it weak Darboux form\/}
\cite[Ex.~5.12]{BBJ} in this case, which we will not discuss.
\label{sa2def2}
\end{dfn}

Here is \cite[Th.~5.18]{BBJ}, the main result of \cite{BBJ}. We
consider it to be a shifted symplectic analogue of Darboux' Theorem,
as it shows that we can choose `coordinate systems' on a $k$-shifted
symplectic derived scheme $(\bX,\om)$ in which $\om$ assumes a
standard form.

\begin{thm} Let\/ $\bX$ be a derived\/ $\K$-scheme with\/
$k$-shifted symplectic form $\ti\om$ for $k<0,$ and\/ $x\in\bX$.
Then there exists a standard form cdga $A$ over $\K$ which is
minimal at\/ $p\in\Spec H^0(A),$ a $k$-shifted symplectic form\/
$\om$ on $\bSpec A,$ and a morphism $\bs f:\bU=\bSpec A\ra\bX$
with\/ $\bs f(p)=x$ and\/ $\bs f^*(\ti\om)\sim\om,$ such that:
\begin{itemize}
\setlength{\itemsep}{0pt}
\setlength{\parsep}{0pt}
\item[{\bf(i)}] If\/ $k$ is odd or divisible by $4,$ then $\bs f$ is
a Zariski open inclusion, and\/ $A,\om$ are in Darboux form, as
in Definition\/~{\rm\ref{sa2def2}}.
\item[{\bf(ii)}] If\/ $k\equiv 2\mod 4,$ then $\bs f$ is
\'etale, and\/ $A,\om$ are in strong Darboux form, as in
Definition\/~{\rm\ref{sa2def2}}.
\end{itemize}
\label{sa2thm3}
\end{thm}

Bouaziz and Grojnowski \cite{BoGr} also independently prove a
similar theorem.

\subsection{`Standard form' atlases for derived stacks}
\label{sa25}

We first generalize Definition \ref{sa2def1} and Theorems
\ref{sa2thm1}--\ref{sa2thm2} to derived Artin stacks:

\begin{dfn} Let $\bX$ be a derived Artin $\K$-stack, and $p$ a point
of $\bX$. By this we mean a morphism $p:\Spec \K\ra\bX$; we may also
call $p$ a $\K$-point of $\bX$. A {\it standard form open
neighbourhood\/} $(A,\bs\vp,\ti p)$ of $p$, in the smooth topology,
means a standard form cdga $A$ over $\K$ in the sense of Definition
\ref{sa2def1}, so that $\bU=\bSpec A$ is an affine derived
$\K$-scheme, and a morphism $\bs\vp:\bU\ra\bX$ which is smooth of
some relative dimension $n\ge 0$, and a $\K$-point $\ti p$ in $\bU$
with $p=\bs\vp(\ti p)$, that is, there is an equivalence of
morphisms $p\simeq\bs\vp\ci\ti p:\Spec \K\ra\bX$. If we do not
specify $p,\ti p$, we just call $(A,\bs\vp)$ a {\it standard form
open neighbourhood\/} in~$\bX$.

For such $\bX,p,(A,\bs\vp,\ti p),n$, as for \eq{sa2eq1} we have the
standard fibre sequence
\e
\smash{\xymatrix@C=30pt{ \bs\vp^*(\bL_{\bX}) \ar[r]^(0.55){\bL_{\bs\vp}} & \bL_\bU \ar[r] & \bL_{\bU/\bX} \ar[r] & \bs\vp^*(\bL_{\bX})[1], }}
\label{sa2eq7}
\e
where $\bL_{\bU/\bX}$ is locally free of rank $n$. Restricting
\eq{sa2eq7} to $\ti p$ and taking cohomology, we have the following:
\begin{itemize}
\setlength{\itemsep}{0pt}
\setlength{\parsep}{0pt}
\item[(a)] There are isomorphisms
$H^i\bigl(\bL_\bX\vert_p\bigr) \cong H^i\bigl(\bL_\bU\vert_{\ti
p}\bigr)$ for $i<0$.
\item[(b)] Since $\bU$ is not stacky,
$H^1\bigl(\bL_\bU\vert_{\ti p}\bigr)=0$ and so there is an exact
sequence of $\K$-vector spaces
\begin{equation*}
\smash{\xymatrix@C=15pt{ 0 \ar[r] & H^0\bigl(\bL_\bX\vert_p\bigr) \ar[r] & H^0\bigl(\bL_\bU\vert_{\ti p}\bigr) \ar[r] & H^0\bigl(\bL_{\bU/\bX}
\vert_{\ti p}\bigr) \ar[r] & H^1\bigl(\bL_\bX\vert_p\bigr) \ar[r] & 0, }}
\end{equation*}
where $H^0\bigl(\bL_{\bU/\bX} \vert_{\ti p}\bigr)\cong \K^n$.
Therefore $n\ge\dim H^1\bigl(\bL_\bX\vert_p\bigr)$.

Note that $H^1\bigl(\bL_\bX\vert_p\bigr)\cong\fIso_\bX(p)^*$,
where $\fIso_\bX(p)$ is the Lie algebra of the isotropy group
$\Iso_\bX(p)$ of $\bX$ at $p$, which is an algebraic $\K$-group.

In particular, the minimal possible relative dimension
$n=\rank(\bL_{\bU/\bX})$ of a neighbourhood $\bs\vp:\bU\ra\bX$
of $p$ is $n=\dim H^1\bigl(\bL_\bX\vert_p\bigr)$.
\item[(c)] If $\bs\vp$ is smooth of minimal relative dimension
$n=\dim H^1\bigl(\bL_\bX\vert_p\bigr)$, then
\e
\smash{H^0\bigl(\bL_\bX\vert_p\bigr)\cong H^0\bigl(\bL_\bU \vert_{\ti
p}\bigr)\quad\text{and}\quad H^0\bigl(\bL_{\bU/\bX} \vert_{\ti
p}\bigr) \cong H^1\bigl(\bL_\bX\vert_p\bigr).}
\label{sa2eq8}
\e
\end{itemize}

We call a standard form open neighbourhood $(A,\bs\vp,\ti p)$ {\it
minimal at\/} $p$ if  $A$ is minimal at $\ti p$ in the sense of
Definition \ref{sa2def1} and $n=\dim H^1\bigl(\bL_\bX\vert_p\bigr)$.
Then parts (a),(c) imply that $A(0)$ is smooth of dimension
$m_0=\dim H^0\bigl(\bL_\bX\vert_p\bigr)$, and $A$ has $m_i=\dim
H^{-i}\bigl(\bL_\bX\vert_p\bigr)$ generators in degree $-i$ for
$i=1,2,\ldots.$
\label{sa2def3}
\end{dfn}

\begin{thm} Let\/ $\bX$ be a derived Artin $\K$-stack, and\/ $p$ a
point of\/ $\bX$. Then there exists a minimal standard form open
neighbourhood\/ $(A,\bs\vp,\ti p)$ of\/ $p,$ in the sense of
Definition\/~{\rm\ref{sa2def3}}.
\label{sa2thm4}
\end{thm}

\begin{proof} Since $\bX$ has a smooth atlas, for any $p\in\bX$
there exists an affine neighbourhood $\bs{\hat\vp}:\bs{\hat
U}\ra\bX$ of $p$, where $\bs{\hat U}$ is an affine derived
$\K$-scheme, $\hat p\in\bs{\hat U}$ with $\bs{\hat\vp}(\hat p)=p$,
and $\bs{\hat\vp}$ is smooth of some relative dimension $\hat n$,
with $\hat n\ge\dim H^1\bigl(\bL_\bX\vert_p\bigr)$ by Definition
\ref{sa2def3}(b). Let $r=\hat n-\dim H^1\bigl(\bL_\bX\vert_p\bigr)$,
so that $r$ is the dimension of the kernel of
$H^0\bigl(\bL_{\bs{\hat U}/\bX}\vert_{\hat p}\bigr)\ra
H^1\bigl(\bL_{\bX}\vert_p\bigr)\ra 0$. We shall use this kernel to
cut down $\bs{\hat\vp}:\bs{\hat U}\ra \bX$ to the minimal
dimension~$n=\dim H^1\bigl(\bL_\bX\vert_p\bigr)$.

Localizing $\bs{\hat U}$ around $\hat p$, by Theorem \ref{sa2thm1}
we may take $\bs{\hat U}=\bSpec\hat A$, where $\hat A$ is a standard
form cdga minimal at $\hat p\in\bs{\hat U}$. Then the natural map
$H^0\bigl(\bL_{\smash{\hat U(0)}}\vert_{\hat p}\bigr)\ra
H^0\bigl(\bL_{\smash{\bs{\hat U}}}\vert_{\hat p}\bigr)$ is an
isomorphism. Since $H^0\bigl(\bL_{\smash{\bs{\hat U}}}\vert_{\hat
p}\bigr)\ra H^0\bigl(\bL_{\smash{\bs{\hat U}/\bX}}\vert_{\hat
p}\bigr) \ra H^1\bigl(\bL_\bX\vert_p\bigr)$ is exact, we may choose
(after localization) functions $x_1,\ldots,x_r$ on $\hat U(0)$
vanishing at $\hat p$ so that $\dd x_{1},\ldots,\dd x_{r}$ at $\hat
p$ map to a basis of the kernel of $H^0\bigl(\bL_{\smash{\bs{\hat
U}/\bX}}\vert_{\hat p}\bigr)\ra H^1\bigl(\bL_\bX\vert_p\bigr)$ under
the composition $H^0\bigl(\bL_{\smash{\hat U(0)}}\vert_{\hat
p}\bigr)\ra H^0\bigl(\bL_{\smash{\bs{\hat U}}}\vert_{\hat p}\bigr)
\ra H^0\bigl(\bL_{\smash{\bs{\hat U}/\bX}}\vert_{\hat p}\bigr)$.

The functions $x_1,\ldots,x_r$ define a map $g:\hat U(0)\ra\bA^r$
and hence a map $\bs f:\bs{\hat U}\ra\bA^{r}$ with $\bs f(\hat
p)=0$. We let $\bU$ denote the (homotopy) fibre $\bs f^{-1}(0)$, so
that we have the following diagram in which the square is a
pullback:
\begin{equation*}
\xymatrix@C=60pt@R=11pt{\bU \ar[d]^{\bs{\ti f}}
\ar[r]_{\bs{\ti\jmath}} \ar@<3pt>@/^.5pc/[rr]^{\bs\vp} & \bs{\hat U}
\ar[d]^{\bs f} \ar[r]_{\bs{\hat\vp}} & \bX \\
{*} \ar[r]^{j=0} & \bA^r. }
\end{equation*}

Let $\ti p$ be the preimage of $\hat p$ in $\bU$. We will show that
after localizing $\bU$ around $\ti p$, the composition
$\bs\vp=\bs{\hat\vp}\ci\bs{\ti\jmath}:\bU\ra\bX$ is smooth of
relative dimension $n=\hat n-r=\dim H^1\bigl(\bL_\bX \vert_p\bigr)$.
Consider the fibre sequence $\bL_{\bU/\bs{\hat U}}[-1] \ra
\bs{\ti\jmath}^*(\bL_{\bs{\hat U}/\bX})\ra\bL_{\bU/\bX}$. We claim
$\bL_{\smash{\bU/\bs{\hat U}}}[-1]$ is free of rank $r$ and that the
map $\bL_{\smash{\bU/\bs{\hat U}}}[-1] \ra\bs{\ti\jmath}^*
(\bL_{\bs{\hat U}/\bX})$ is injective at $\ti p$ and hence, by
Nakayama's Lemma, in a neighbourhood of $\ti p$. Localizing $\bU$
around $\ti p$, it will follow immediately that $\bL_{\bU/\bX}$ is
locally free of rank $n=\hat n-r$. Thus $\bs\vp:\bU\ra\bX$ is the
desired neighbourhood of $p$ of minimal relative dimension.

To sustain the claim, note that since the cotangent complex of
$*=\Spec \K$ is zero, we have an equivalence $\bL_{j}[-1] \simeq
j^*(\Om^{1}_{\bA^r})$. Thus $\bL_j[-1]$ is free of rank $r$ and
hence so is $\bs{\ti f}{}^*(\bL_j)[-1] \simeq
\bL_{\smash{\bU/\bs{\hat U}}}[-1]$. Furthermore, the map in question
$\bL_{\smash{\bU/\bs{\hat U}}}[-1] \ra
\bs{\ti\jmath}{}^*(\bL_{\smash{\bs{\hat U}/\bX}})$ factors as
$\bL_{\smash{\bU/\bs{\hat U}}}[-1]\simeq\bs{\ti f}{}^*\ci
j^*(\Om^1_{\smash{\bA^r}})\simeq \bs{\ti\jmath}{}^*\ci \bs
f^*(\Om^1_{\smash{\bA^r}})\ra \bs{\ti\jmath}{}^*(\bL_{\bs{\hat U}})
\ra \bs{\ti\jmath}{}^*(\bL_{\bs{\hat U}/\bX})$. But $\bs f$ was
constructed precisely so that the composition $\bs
f^*(\Om^{1}_{\smash{\bA^r}}) \ra\bL_{\bs{\hat U}}\ra\bL_{\bs{\hat
U}/\bX}$ should be injective at $\hat p$. Thus, we may choose an
affine neighbourhood $\bs\vp:\bU\ra\bX$, $\ti p$ of $p$ which is
smooth of the minimal relative dimension $n=\dim H^1\bigl(\bL_\bX
\vert_p\bigr)$. Applying Theorem \ref{sa2thm1} to $\bU$ at $\ti p$,
we may take $\bU=\bSpec A$, where $A$ is a standard form cdga
minimal at~$\ti p\in\bU$.
\end{proof}

\begin{thm} Let\/ $\bX$ be a derived Artin $\K$-stack and\/
$(A,\bs\vp),(B,\bs\psi)$ standard form open neighbourhoods in\/
$\bX,$ and write $\bU=\bSpec A,$ $\bV=\bSpec B$. Then for
each\/ $p\in\bU\t_\bX\bV$ there exist a standard form cdga $C$ over
$\K$ minimal at\/ $q\in\bW=\bSpec C,$ an \'etale morphism $\bs
i:\bW\ra\bU\t_\bX\bV$ with\/ $\bs i(q)=p,$ and cdga morphisms
$\al:A\ra C,$ $\be:B\ra C$ with\/ $\bs\pi_\bU\ci\bs
i\simeq\bSpec\al:\bW\ra\bU$ and\/ $\bs\pi_\bV\ci\bs
i\simeq\bSpec\be:\bW\ra\bV$.
\label{sa2thm5}
\end{thm}

\begin{proof} Since $\bU,\bV$ are derived $\K$-schemes, $\bU\t_\bX\bV$ is a derived algebraic $\K$-space, and \'etale locally equivalent to a derived $\K$-scheme. Thus given $p\in\bU\t_\bX\bV$ we may choose an affine derived $\K$-scheme $\bs{\hat W}$, a point $\hat q\in\bs{\hat W}$ and an \'etale map $\bs{\hat\imath}:\bs{\hat W}\ra\bU\t_\bX\bV$ with~$\bs{\hat\imath}(\hat q)=p$.

Write $\hat W=t_0(\bs{\hat W})$, $U(0)=\Spec A^0$ and $V(0)=\Spec B^0$ for the classical schemes. The compositions $\hat W\hookra\bs{\hat W}\,\smash{\buildrel\bs{\hat\imath}\over\longra}\,\bU\t_\bX\bV\,{\buildrel\bs\pi_\bU\over\longra}\,\bU\hookra U(0)$ and $\hat W\hookra\bs{\hat
W}\,\smash{\buildrel\bs{\hat\imath}\over\longra}\,\bU\t_\bX\bV\,{\buildrel\bs\pi_\bV\over\longra}\,\bV\hookra V(0)$ give maps
$\hat W\ra U(0)$, $\hat W\ra V(0)$. Choose a map $\hat W\ra\bA^N$
such that the product map $\hat W\ra U(0)\t V(0)\t\bA^N$ is a
locally closed embedding. Localizing $\bs{\hat W},\hat W$ at $\hat q$ if
necessary, we can choose a locally closed $\K$-subscheme $W(0)$ of
$U(0)\t V(0)\t\bA^N$ containing the image of $\hat W$ as a closed
$\K$-subscheme, such that $W(0)$ is smooth of dimension $\dim
T_{\smash{\hat q}}\hat W$. For instance, $W(0)$ can be obtained as an intersection of an appropriate regular sequence of hypersurfaces.

Following the proof of Theorem \ref{sa2thm1} in \cite[\S 4.1]{BBJ},
we can construct a standard form cdga $C$ minimal at $q\in\bW=\bSpec
C$ with $\Spec C^0=W(0)$ and an equivalence $\bs j:\bW\ra\bs{\hat
W}$ with $\bs j(q)=\hat q$. Setting $\bs i=\bs{\hat\imath}\ci\bs j$, we now have morphisms of derived schemes $\bs\pi_\bU\ci\bs i:\bW\ra\bU,$ $\bs\pi_\bV\ci\bs i:\bW\ra\bV$ whose classical truncations $\pi_U\ci i:W\ra U$, $\pi_V\ci i:W\ra V$ extend to morphisms of the ambient smooth
schemes $W(0)=\Spec C^0\ra U(0)=\Spec A^0$, $W(0)=\Spec C^0\ra
V(0)=\Spec B^0$. As $A,B$ are freely generated in negative degrees,
it follows that we may write $\bs\pi_\bU\ci\bs i\simeq\bSpec\al$ and
$\bs\pi_\bV\ci\bs i\simeq\bSpec\be$ for morphisms of cdgas $\al:A\ra
C$, $\be:B\ra C$. This completes the proof.
\end{proof}

\subsection[`Darboux form' atlases for shifted symplectic derived
stacks]{`Darboux form' atlases for shifted symplectic stacks}
\label{sa26}

Here is the main result of this section, a stack analogue of Theorem
\ref{sa2thm3}. Note that (a)(i)--(v) are modelled closely on the
first part of Definition \ref{sa2def2}, and equations
\eq{sa2eq9}--\eq{sa2eq13} are analogues of or identical
to~\eq{sa2eq2}--\eq{sa2eq6}.

\begin{thm}{\bf(a)} Let\/ $(\bX,\om_\bX)$ be a $k$-shifted
symplectic derived Artin $\K$-stack, where $k=-2d-1$ for
$d=0,1,2,\ldots,$ and\/ $p\in\bX$. Then we can construct a minimal
standard form open neighbourhood\/ $(A,\bs\vp:\bU\ra\bX,\ti p)$ of\/ $p$ in
the sense of Definition\/ {\rm\ref{sa2def3},} and a $k$-shifted
closed\/ $2$-form $\om=(\om^0,0,\ldots)$ on $\bU=\bSpec A$ for
$\om^0\in(\La^2\Om^1_A)^k,$ such that $\bs\vp^*(\om_\bX)\sim\om$ in
$k$-shifted closed\/ $2$-forms on $\bU=\bSpec A$. Furthermore,
$A,\om$ are in a standard `Darboux form', a modified version of
Definition {\rm\ref{sa2def2},} as follows:
\begin{itemize}
\setlength{\itemsep}{0pt}
\setlength{\parsep}{0pt}
\item[{\bf(i)}] The degree $0$ part\/ $A^0$ of $A$ is a
smooth\/ $\K$-algebra of dimension $m_0,$ and we are given
$x^0_1,\ldots,x^0_{m_0}\in A^0$ such that\/ $\dd
x^0_1,\ldots,\dd x^0_{m_0}$ form a basis of\/ $\Om^1_{A^0}$
over\/~$A^0$.
\item[{\bf(ii)}] As a graded commutative algebra, $A$ is freely
generated over $A^0$ by variables
\e
\begin{aligned}
&x_1^{-i},\ldots,x^{-i}_{m_i} &&\text{in degree $-i$ for
$i=1,\ldots,d,$} \\
&y_1^{i-2d-1},\ldots,y^{i-2d-1}_{m_i} &&\text{in degree $i-2d-1$
for $i=0,1,\ldots,d,$} \\
&w_1^{-2d-2},\ldots,w^{-2d-2}_n &&\text{in degree $-2d-2,$}
\end{aligned}
\label{sa2eq9}
\e
for $m_0,\ldots,m_d\ge 0$ with\/ $m_0$ as in {\bf(i)\rm,} and\/
$n=\dim H^1\bigl(\bL_\bX\vert_p\bigr)$ the relative dimension
of\/ $\bs\vp$. The upper index\/ $i$ in\/ $w_j^i,x^i_j,y^i_j$ is
the degree. Then
\e
\smash{\om^0=\ts\sum_{i=0}^d\sum_{j=1}^{m_i}\dd y^{i-2d-1}_j\,\dd x^{-i}_j \qquad \text{in $(\La^2\Om^1_A)^{-2d-1}$.}}
\label{sa2eq10}
\e
\item[{\bf(iii)}] We are given \/ $H$ in
$A^{-2d},$ called the \begin{bfseries}Hamiltonian\end{bfseries},
which satisfies the \begin{bfseries}classical master
equation\end{bfseries}\vskip -13pt
\e
\sum_{i=1}^d\sum_{j=1}^{m_i}\frac{\pd H}{\pd x^{-i}_j}\,
\frac{\pd H}{\pd y^{i-2d-1}_j}=0\qquad\text{in $A^{1-2d}$.}
\label{sa2eq11}
\e
The differential\/ $\d$ on $A$ satisfies $\d=0$ on $A^0,$ and\vskip -10pt
\e
\d x^{-i}_j =\frac{\pd H}{\pd y^{i-2d-1}_j}, \quad \d
y^{i-2d-1}_j=\frac{\pd H}{\pd
x^{-i}_j},\quad\begin{subarray}{l}\ts
i=0,\ldots,d,\\[6pt] \ts j=1,\ldots,m_i.\end{subarray}
\label{sa2eq12}
\e\vskip -6pt

Note that\/ \eq{sa2eq12} \begin{bfseries}does not
specify\end{bfseries} $\d w_j^{-2d-2}$ for $j=1,\ldots,n,$ and
so \begin{bfseries}does not completely determine\end{bfseries}
$\d$ on\/~$A$.
\item[{\bf(iv)}] Define $\Phi\in A^{-2d}$ and
$\phi\in(\Om^1_A)^{-2d-1}$ by $\Phi=-\frac{1}{2d+1}\,H$ and\vskip -20pt
\e
\phi=\frac{1}{2d+1}\sum_{i=0}^d\sum_{j=1}^{m_i}\bigl[(2d+1-i)
y^{i-2d-1}_j\,\dd x^{-i}_j+i\,x^{-i}_j\,\dd y^{i-2d-1}_j\bigr].
\label{sa2eq13}
\e\vskip -9pt
\noindent Then $\d\Phi=0,$ $\dd\Phi+\d\phi=0,$ and\/ $\om^0=\dd\phi$.
\item[{\bf(v)}] Minimality of\/ $(A,\bs\vp,\ti p)$ means
that\/ $\d w_j^{-2d-2}\vert_{\ti p}=0$ for $j=1,\ldots,n$ and\vskip -15pt
\begin{equation*}
\d x^{-i}_j\big\vert_{\ti p}=\frac{\pd H}{\pd
y^{i-2d-1}_j}\bigg\vert_{\ti p}=0=\d y^{i-2d-1}_j\big\vert_{\ti
p}=\frac{\pd H}{\pd x^{-i}_j}\bigg\vert_{\ti
p},\;\>\begin{subarray}{l}\ts
i=0,\ldots,d,\\[6pt] \ts j=1,\ldots,m_i.\end{subarray}
\end{equation*}
\end{itemize}

\noindent{\bf(b)} In part\/ {\bf(a)\rm,} let\/ $B$ be the
graded subalgebra of\/ $A$ generated by $A^0$ and the variables
$x^i_j,y^i_j$ in {\bf(ii)} for all\/ $i,j,$ with inclusion
$\io:B\hookra A$. Then $B$ is closed under\/ $\d,$ and so is a dg-subalgebra of\/ $A$. For degree reasons $H,\Phi$ above cannot depend on the $w_j^{-2d-2},$ so $H,\Phi\in B$. Also the data $\om,\om^0,\phi$ in $\Om^1_A,\La^2\Om^1_A$ above are the images under $\io$ of\/ $\om_B,\om^0_B,\phi_B$ in $\Om^1_B,\La^2\Om^1_B$. Then $\om_B$ is a $k$-shifted symplectic structure on $\bV=\bSpec B,$ and\/ $B,\om_B$ is in Darboux form as in Definition\/ {\rm\ref{sa2def2},} and\/ $B$ is minimal at\/ $\ti p$ as in Definition\/~{\rm\ref{sa2def1}}.

Geometrically, we have a diagram of morphisms in $\dArt_\K\!:$
\begin{equation*}
\smash{\xymatrix@C=75pt@R=12pt{ \bV=\bSpec B & \bU=\bSpec A \ar[r]^{\bs\vp} \ar[l]_{\bs i=\bSpec\io} & \bX, }}
\end{equation*}
where $(\bX,\om_\bX),$ $(\bV,\om_B)$ are $k$-shifted symplectic,
with\/ $\bs\vp^*(\om_\bX)\sim \bs i^*(\om_B)$ in $k$-shifted
closed\/ $2$-forms on $\bU$. We can think of\/ $\bs\vp:\bU\ra\bX$ as
a `submersion', and\/ $\bs i:\bU\hookra\bV$ as an embedding of\/
$\bU$ as a derived subscheme of\/ $\bV$. On classical schemes,
$i=t_0(\bs i):U=t_0(\bU)\ra V=t_0(\bV)$ is an isomorphism. There is
a natural equivalence of relative (co)tangent complexes
\e
\smash{\bL_{\bU/\bV} \simeq \bT_{\bU/\bX}[1-k].}
\label{sa2eq14}
\e

\noindent{\bf(c)} The obvious analogues of\/ {\bf(a)\rm,\bf(b)} also
hold if\/ $(\bX,\om_\bX)$ is a $k$-shifted symplectic derived Artin
$\K$-stack for $k<0$ with\/ $k\equiv 0\mod 4$ or $k\equiv 2\mod 4$.
In each case, the algebra $A$ is the corresponding algebra from
Definition {\rm\ref{sa2def2},} modified by adding generators
$w_1^{k-1},\ldots,w^{k-1}_n$ in degree~$k-1$.
\label{sa2thm6}
\end{thm}

\begin{proof} For (a), let $(\bX,\om_\bX)$ be a $k$-shifted
symplectic derived Artin $\K$-stack with $k=-2d-1$ for $d\ge 0$, and
$p\in\bX$. By Theorem \ref{sa2thm4} we may choose a minimal standard
form open neighbourhood $(A,\bs\vp,\ti p)$ of $p$, which we may
localize further during the proof. Then by Definition \ref{sa2def3},
$\bs\vp$ is smooth of relative dimension $n=\dim H^1\bigl(\bL_\bX
\vert_p\bigr)$, and $A(0)$ is smooth of dimension $m_0=\dim
H^0\bigl(\bL_\bX\vert_p\bigr)$, and $A$ has $m_i=\dim
H^{-i}\bigl(\bL_\bX\vert_p\bigr)$ generators in degree $-i$ for
$i=1,2,\ldots.$

Since $(\bX,\om_\bX)$ is $k$-shifted symplectic for $k=-2d-1$ we
have $H^{-i}\bigl(\bL_\bX\vert_p\bigr)\cong
H^{k+i}\bigl(\bL_\bX\vert_p\bigr){}^*$, so $\dim
H^{-i}\bigl(\bL_\bX\vert_p\bigr)=\dim
H^{k+i}\bigl(\bL_\bX\vert_p\bigr)$. Thus, $A$ is freely generated
over $A^0$ by $m_i$ generators in degree $-i$ for $i=1,\ldots,d,$
and $m_i$ generators in degree $i-2d-1$ for $i=0,1,\ldots,d,$ and
$n$ generators in degree $-2d-2$, which is the same number of
variables as in \eq{sa2eq9}.

The pullback $\bs\vp^*(\om_\bX)$ is a $k$-shifted closed 2-form on
$\bU=\bSpec A$, so Proposition \ref{sa2prop1} gives
$\om^0\in(\Om^2_A)^k$ with $\d\om^0=\dd\om^0=0$ and
$\bs\vp^*(\om_\bX)\sim(\om^0,0,0,\ldots)$. Consider the morphism
$\om^0\cdot:\bT_A\ra\Om^1_A[k]$ given by contraction with $\om^0$, and
its restriction to $\ti p$ on cohomology, which gives morphisms
\e
\smash{H^i\bigl(\om^0\cdot\vert_{\ti p}\bigr):H^{i}\bigl(\bT_A\vert_{\ti
p}\bigr)\cong H^{-i}\bigl(\Om^1_A\vert_{\ti p}\bigr)^* \longra
H^{k+i}\bigl(\Om^1_A\vert_{\ti p}\bigr).}
\label{sa2eq15}
\e
On cohomology $\om^0\cdot$ factorizes as $\bT_A \ra
\bs\vp^*(\bT_{\bX}) \ra \bs\vp^*(\bL_{\bX})[k] \ra \Om^1_A[k]$. Here
$\bs\vp^*(\bT_\bX)\ra \bs\vp^*(\bL_\bX)[k]$ is the pullback of
$\om_\bX\cdot:\bT_\bX\ra\bL_{\bX}[k]$, which is an equivalence as
$\om_\bX$ is nondegenerate. Also
$\bs\vp^*(\bL_{\bX})[k]\ra\Om^1_A[k]$ is $\bL_{\bs\vp}[k]$ as in
\eq{sa2eq7}, and so as in Definition \ref{sa2def3}, on cohomology
$H^i$ at $\ti p$ is an isomorphism for $i\le -k$, and zero for
$i=1-k$. The map $\bT_A \ra \bs\vp^*(\bT_{\bX})$ is the dual of
$\bL_{\bs\vp}$, and so on cohomology $H^i$ at $\ti p$ is an
isomorphism for $i\ge 0$, and zero for $i=-1$. Combining these,
\eq{sa2eq15} is an isomorphism for $0\le i\le -k$ and zero
otherwise.

We can now prove (a)(i)--(iv) by following the proof of the $k$ odd
case of Theorem \ref{sa2thm3} in \cite[\S 5.6]{BBJ}. Localizing $A$
at $\ti p$ if necessary, this chooses \'etale coordinates
$x^0_1,\ldots,x^0_{m_0}$ on $U^0=\Spec A^0$, and generators
$x_1^{-i},\ldots,x^{-i}_{m_i}$ in degree $-i$ for $i=1,\ldots,d$ and
$y_1^{i-2d-1},\ldots,y^{i-2d-1}_{m_i}$ in degree $i-2d-1$ for
$i=0,1,\ldots,d$ for $A$, such that $\om^0$ is given by
\eq{sa2eq10}, and also constructs $H,\Phi,\phi$ satisfying
\eq{sa2eq11}--\eq{sa2eq13}. The proof in \cite[\S 5.6]{BBJ} does not
choose the generators $w_1^{-2d-2},\ldots,w^{-2d-2}_n$ for $A$ in
degree $-2d-2$, but as these are not required to satisfy any
conditions, they can be chosen arbitrarily. Note that
$\om^0,H,\Phi,\phi$ do not involve $w_1^{-2d-2},\ldots,w^{-2d-2}_n$
for degree reasons. Part (a)(v) follows from Definition
\ref{sa2def3} and \eq{sa2eq12}. This completes~(a).

The first parts of (b) are immediate, comparing (a) with Definition
\ref{sa2def2}. To construct the equivalence \eq{sa2eq14}, consider
the following diagram, in which the rows are the standard fibre
sequences and the vertical arrow is induced by an inverse of
$\bs\vp^*(\om_{\bX})$:
\e
\begin{gathered}
\xymatrix@C=50pt@R=11pt{\bL_{\bU/\bX}[-1] \ar[r] &
\bs\vp^*(\bL_{\bX}) \ar[r]
\ar[d]^\simeq & \bL_{\bU} \\
\bT_{\bU}[-k] \ar[r] & \bs\vp^*(\bT_{\bX})[-k] \ar[r] &
\bT_{\bU/\bX}[1-k].}
\end{gathered}
\label{sa2eq16}
\e
Since $\bL_{\bU/\bX}$ and $\bT_{\bU/\bX}$ can be assumed to be free,
we have
\begin{align*}
\Ext^{-1}(\bL_{\bU/\bX}[-1],\bT_{\bU/\bX}[1-k])&\cong
\Ext^{1-k}(\bL_{\bU/\bX},\bT_{\bU/\bX})=0,\\
\Hom(\bL_{\bU/\bX}[-1],\bT_{\bU/\bX}[1-k])&\cong
\Ext^{-k+2}(\bL_{\bU/\bX},\bT_{\bU/\bX})=0.
\end{align*}
Applying ${\mathbb R}{\cal H}om(\bL_{\bU/\bX}[-1],-)$ to the bottom
row of \eq{sa2eq16} and taking cohomology, we find that
$\Hom(\bL_{\bU/\bX}[-1],\bT_{\bU}[-k]) \cong
\Hom(\bL_{\bU/\bX}[-1],\bs\vp^*(\bT_{\bX})[-k])$. Thus \eq{sa2eq16}
can be filled in to a commutative diagram
\e
\begin{gathered}
\xymatrix@C=50pt@R=11pt{\bL_{\bU/\bX}[-1] \ar[r] \ar@{.>}[d] &
\bs\vp^*(\bL_{\bX}) \ar[r] \ar[d]^\simeq & \bL_{\bU} \ar@{.>}[d] \\
\bT_{\bU}[-k] \ar[r] & \bs\vp^*(\bT_{\bX})[-k] \ar[r] &
\bT_{\bU/\bX}[1-k],}
\end{gathered}
\label{sa2eq17}
\e
and such a filling is unique up to homotopy.

Restricting \eq{sa2eq17} to $\ti p$ and taking cohomology gives a
commutative diagram:
\e
\begin{gathered}
\xymatrix@C=120pt@R=11pt{ *+[r]{H^{k-1}\bigl(\bL_\bX\vert_p\bigr)}
\ar[r]_\cong \ar[d]^\cong & *+[l]{H^{k-1}\bigl(\bL_\bU\vert_{\ti p}\bigr)}
\ar[d] \\ *+[r]{H^{k-1}\bigl(\bT_\bX\vert_p[-k]\bigr)} \ar[r] &
*+[l]{H^{k-1}\bigl(\bT_{\bU/\bX}\vert_{\ti p}[1-k]\bigr).}}
\end{gathered}
\label{sa2eq18}
\e
Since the morphism
\begin{equation*}
\smash{H^{-1}\bigl(\bT_\bX\vert_p\bigr)\cong H^{k-1}\bigl(\bT_\bX\vert_p[1-k]\bigr)
\ra H^{k-1}\bigl(\bT_{\bU/\bX}\vert_{\ti p}[-k]\bigr) \cong
H^0\bigl(\bT_{\bU/\bX}\vert_{\ti p}\bigr)}
\end{equation*}
is dual to $H^0\bigl(\bL_{\bU/\bX}\vert_{\ti p}\bigr)\ra
H^1\bigl(\bL_\bX\vert_p\bigr)$, which is an isomorphism by
\eq{sa2eq8}, we see from \eq{sa2eq18} that
$H^{k-1}\bigl(\bL_\bU\vert_{\ti p}\bigr)\ra
H^{k-1}\bigl(\bT_{\bU/\bX}\vert_{\ti p}[1-k]\bigr)$ is also an
isomorphism.

Next, consider the fibre sequence $\io^*(\bL_{\bV}) \ra \bL_{\bU}
\ra \bL_{\bU/\bV}$. Note that $\bL_{\bU/\bV}[k-1]$ is free of rank
$\dim H^{k-1}\bigl(\bL_\bX\vert_p\bigr)=\dim
H^1\bigl(\bL_\bX\vert_p\bigr)=n$ and that the natural map
$H^{k-1}\bigl(\bL_\bU\vert_{\ti p}\bigr)\ra
H^{k-1}\bigl(\bL_{\bU/\bV}\vert_{\ti p}\bigr)$ is an isomorphism by
the minimality of the inductive construction of $\bU=\bSpec A$ in
Definition~\ref{sa2def1}.

Since $\io^*(\bL_\bV)$ has amplitude in $[k,0]$ and $\bT_{\bU/\bX}$
is locally free, the composition $\io^*(\bL_\bV)\ra\bL_\bU\ra
\bT_{\bU/\bX}[1-k]$ is homotopic to zero, and we can therefore
choose a factorization $\bL_\bU\ra\bL_{\bU/\bV}\ra
\bT_{\bU/\bX}[1-k]$ of $\bL_\bU\ra\bT_{\bU/\bX}[1-k]$. Restricting
this factorization to $\ti p$ and taking cohomology, we see that the
induced map $H^{k-1}\bigl(\bL_{\bU/\bV}\vert_{\ti p}\bigr)\ra
H^{k-1}(\bT_{\bU/\bX}\vert_{\ti p}[1-k]\bigr)$ is an isomorphism. By
Nakayama's Lemma, the map $\bL_{\bU/\bV}\ra\bT_{\bU/\bX}[1-k]$ is an
equivalence in a neighbourhood of $\ti p$. So localizing $\bU,\bV$
if necessary, equation \eq{sa2eq14} holds, proving part~(b).

For (c), we follow the same method, using the `Darboux form' in
\cite[Ex.~5.9]{BBJ} for $k\equiv 0\mod 4$, and the `strong Darboux
form' in \cite[Ex.~5.10]{BBJ} for $k\equiv 2\mod 4$. As in the proof
of \cite[Th.~5.18(iii)]{BBJ}, in the case $k\equiv 2\mod 4$, as well
as modifying $A$ by localizing at $\ti p$ (i.e. restricting to a
Zariski open neighbourhood of $\ti p$ in $\bU=\bSpec A$), we 
also need to modify $A$ by adjoining square roots of some nonzero
functions in $A^0$ (i.e. taking a finite \'etale cover of
$\bU=\bSpec A$). As the result is still a minimal standard form open
neighbourhood $(A,\bs\vp,\ti p)$ of $p$, this does not affect the
statement of the theorem.
\end{proof}

In the case $k=-1$, as in \cite[Ex.~5.15]{BBJ} the classical
$\K$-schemes $U\cong V$ in Theorem \ref{sa2thm6}(a),(b) are
isomorphic to $\Crit\bigl(H:U(0)\ra\bA^1\bigr)$. Also $\bs\vp:\bs T\ra\bX$ smooth implies $\vp=t_0(\bs\vp):T=t_0(\bs T)\ra X=t_0(\bX)$ is smooth. So changing notation from $U(0),H,\ti p$ to $U,f,u$, using
$H^i\bigl(\bL_X\vert_p\bigr)\cong H^i\bigl(\bL_\bX\vert_p\bigr)$ for
$X=t_0(\bX)$ and $i=0,1$, and applying Proposition \ref{sa2prop1}(b)
to get $f\vert_{T^\red}=0$, we deduce:

\begin{cor} Let\/ $(\bX,\om_\bX)$ be a $-1$-shifted
symplectic derived Artin $\K$-stack, and\/ $X=t_0(\bX)$ the
corresponding classical Artin $\K$-stack. Then for each\/ $p\in X$
there exist a smooth\/ $\K$-scheme $U$ with dimension $\dim
H^0\bigl(\bL_X\vert_p\bigr),$ a point\/ $t\in U,$ a regular function
$f:U\ra\bA^1$ with\/ $\dd f\vert_t=0,$ so that\/
$T:=\Crit(f)\subseteq U$ is a closed\/ $\K$-subscheme with\/ $t\in
T,$ and a morphism $\vp:T\ra X$ which is smooth of relative
dimension $\dim H^1\bigl(\bL_X\vert_p\bigr),$ with\/ $\vp(t)=p$. We
may take\/~$f\vert_{T^\red}=0$.

Here the derived critical locus $\bs\Crit(f:U\ra\bA^1),$ as a
$-1$-shifted symplectic derived scheme, agrees with $(\bV,\om_B)$ in
Theorem\/ {\rm\ref{sa2thm6},} and\/ $\vp:T\ra X$ corresponds to
$t_0(\bs\vp)\ci t_0(\bs i)^{-1}$ in Theorem\/~{\rm\ref{sa2thm6}}.
\label{sa2cor1}
\end{cor}

Thus, the underlying classical stack $X$ of a $-1$-shifted symplectic derived
stack $(\bX,\om_\bX)$ admits an atlas consisting of critical loci of
regular functions on smooth schemes.

Now let $Y$ be a Calabi--Yau 3-fold over $\K$, and $\cM$ a classical
moduli stack of coherent sheaves $F$ on $Y$, or complexes $F^\bu$ in
$D^b\coh(Y)$ with $\Ext^{<0}(F^\bu,F^\bu)=0$. Then
$\cM=t_0(\bs\cM)$, for $\bs\cM$ the corresponding derived moduli
stack. The (open) condition $\Ext^{<0}(F^\bu,F^\bu)=0$ is needed to
make $\bs\cM$ $1$-truncated (that is, a derived Artin stack, in our terminology), and so make $\cM=t_0(\bs\cM)$ an ordinary, and not higher, stack. Pantev et al.\ \cite[\S 2.1]{PTVV}
prove $\bs\cM$ has a $-1$-shifted symplectic structure
$\om_{\bs\cM}$. Applying Corollary \ref{sa2cor1} and using
$H^i\bigl(\bL_{\bs\cM}\vert_{[F]}\bigr)\cong \Ext^{1-i}(F,F)^*$
yields a new result on classical 3-Calabi--Yau moduli stacks,
the statement of which involves no derived geometry:

\begin{cor} Suppose $Y$ is a Calabi--Yau\/ $3$-fold over\/
$\K,$ and\/ $\cM$ a classical moduli $\K$-stack of coherent sheaves
$F,$ or more generally of complexes $F^\bu$ in $D^b\coh(Y)$ with
$\Ext^{<0}(F^\bu,F^\bu)=0$. Then for each\/ $[F]\in\cM,$ there exist
a smooth\/ $\K$-scheme $U$ with\/ $\dim U=\dim\Ext^1(F,F),$ a
point\/ $u\in U,$ a regular function $f:U\ra\bA^1$ with\/ $\dd
f\vert_u=0,$ and a morphism $\vp:\Crit(f)\ra\cM$ which is smooth of
relative dimension $\dim\Hom(F,F),$ with\/~$\vp(u)=[F]$.
\label{sa2cor2}
\end{cor}

This is an analogue of \cite[Cor.~5.19]{BBJ}. When $\K=\C$, a
related result for coherent sheaves only, with $U$ a complex
manifold and $f$ a holomorphic function, was proved by Joyce and
Song \cite[Th.~5.5]{JoSo} using gauge theory and transcendental
complex methods.

\subsection{Comparing `Darboux form' atlases on overlaps}
\label{sa27}

Let $(\bX,\om_\bX)$ be a $k$-shifted symplectic derived Artin
$\K$-stack for $k<0$. Theorem \ref{sa2thm6} gives a minimal standard
form open neighbourhood $(A,\bs\vp,\ti p)$ of each $p$ in $\bX$ with
$\bs\vp^*(\om_\bX)\sim\om$, where the $k$-shifted closed 2-form
$\om=(\om^0,0,\ldots)$ on $\bU=\bSpec A$ is in a standard `Darboux
form', and $\Phi\in A^{k+1}$, $\phi\in (\Om^1_A)^k$ with $\d\Phi=0$,
$\dd\Phi+\d\phi=0$, $\dd\phi=\om^0$, satisfying
$\Phi\vert_{U^\red}=0$ if $k=-1$, as in Proposition
\ref{sa2prop1}(a),(b). We think of $A,\bs\vp,\om,\Phi,\phi$ as like
coordinates on $\bX$ near $p$ in the smooth topology, which write
$\bX,\om_\bX$ in a nice way.

It is often important in geometric problems to compare different
choices of coordinates on the overlap of their domains. So suppose
$A,\bU,\bs\vp,\om,\Phi,\phi$ and $A',\bU',\bs\vp',\om', \Phi',\phi'$
are two choices as above, and $q\in\bU\t_{\bs\vp,\bX,\bs\vp'}\bU'$.
We would like to compare the presentations
$A,\bU,\bs\vp,\om,\Phi,\phi$ and $A',\bU',\bs\vp',\om',\Phi',\phi'$
for $\bX$ near $q$. Here is a method for doing this, following
\cite[\S 5.8]{BBJ} in the scheme case:
\begin{itemize}
\setlength{\itemsep}{0pt}
\setlength{\parsep}{0pt}
\item[(i)] Apply Theorem \ref{sa2thm5} to $(A,\bs\vp),
(A',\bs\vp'),q$. This gives a standard form cdga $B$ minimal at
$r\in\bV=\bSpec B$, an etale map $\bs i:\bV\hookra\bU\t_\bX\bU'$ with $\bs i(r)=q$, and morphisms of cdgas $\al:A\ra B,$ $\al':A'\ra B$ with $\bs\pi_\bU\ci\bs i\simeq\bSpec\al:\bV\ra\bU$ and $\bs\pi_{\bU'}\ci\bs
i\simeq\bSpec\al':\bV\ra\bU'$.
\item[(ii)] The pullbacks $\al_*(\om)=(\al_*(\om^0),0,
\ldots)$, $\al'_*(\om')=(\al'_*(\om^{\prime 0}),0,\ldots)$ are
$k$-shifted closed 2-forms on $\bV=\bSpec B$, which are
equivalent as
\begin{align*}
\al_*(\om)&\sim(\bSpec\al)^*\ci\bs\vp^*(\om_\bX)\sim
\bs i^*\ci\bs\pi_\bU^*\ci\bs\vp^*(\om_\bX)\\
&\sim\bs i^*\ci\bs\pi_{\bU'}^*\ci\bs\vp^{\prime*}(\om_\bX)\sim
(\bSpec\al')^*\ci\bs\vp^{\prime *} (\om_\bX)\sim\al'_*(\om').
\end{align*}
Since $B$ is minimal at $r$, $\al_*(\om),\al'_*(\om')$ satisfy
nondegeneracy properties near $r$. Also $\d\al(\Phi)=0$,
$\dd\al(\Phi)+\d\al_*(\phi)=0$, $\d\al'(\Phi')=0$,
$\dd\al_*(\phi)=\al_*(\om^0)$,
$\dd\al'(\Phi')+\d\al_*'(\phi')=0$,
$\dd\al_*(\phi)=\al_*(\om^0)$, and if $k=-1$ then
$\al(\Phi)\vert_{V^\red}=0=\al'(\Phi')\vert_{V^\red}$. Therefore
Proposition \ref{sa2prop1}(c) applies, yielding $\Psi\in B^k$
and $\psi\in(\Om^1_B)^{k-1}$ with
\begin{align*}
\al(\Phi)-\al'(\Phi')&=\d\Psi &&\text{in $B^{k+1}$, and}\\
\al_*(\phi)-\al_*'(\phi')&=\dd\Psi+\d\psi &&\text{in
$(\Om^1_B)^k.$}
\end{align*}
The data $B,\bV,\bs i,\al,\al',r,\Psi,\psi$ compare the
Darboux presentations $A,\ab\bU,\ab\bs\vp,\ab\om,\ab\Phi,
\ab\phi$ and $A',\bU',\bs\vp',\om',\Phi',\phi'$ for $\bX$
near~$q$.
\end{itemize}

Using this method in the case $k=-1$ yields the following comparison
result for the critical atlases of Corollary \ref{sa2cor1}. We have
replaced $\ab t_0(\bU),U(0),\ab t_0(\bU'),\ab U'(0),\ab t_0(\bV),\ab
V(0),\ab\Spec\al^0,\ab\Spec\al^{\prime 0}$ above by $T,U,T',U',\ab
R,\ab V,\ab\th,\th'$. The conclusion $f\ci\th-f'\ci\th'\in
I_{R,V}^2$ is proved as in~\cite[Ex.~5.35]{BBJ}.

\begin{prop} Let\/ $(\bX,\om_\bX)$ be a $-1$-shifted symplectic
derived Artin $\K$-stack, and\/ $X=t_0(\bX)$ the corresponding
classical Artin $\K$-stack. Suppose $U,f:U\ra\bA^1,$
$\vp:T=\Crit(f)\ra X$ and\/ $U',f':U'\ra\bA^1,$
$\vp':T'=\Crit(f')\ra X$ are two choices of the data constructed in
Corollary\/ {\rm\ref{sa2cor1}} for points $p,p'\in\bX,$ with\/
$f\vert_{T^\red}=0=f'\vert_{T^{\prime\red}}$. Let\/ $q\in
T\t_{\vp,X,\vp'}T'$. Then there exist a smooth\/ $\K$-scheme $V,$ a
closed $\K$-subscheme $R\subseteq V,$ a point $r\in R,$ and
morphisms $\th:V\ra U,$ $\th':V\ra U'$ with\/ $\th(R)\subseteq T,$
$\th'(R)\subseteq T'$ such that the following diagram $2$-commutes
(homotopy commutes) in $\Art_\K\!:$
\begin{equation*}
\xymatrix@!0@C=80pt@R=24pt{ V \ar[rr]_{\th'} \ar[dd]^\th && U'
\ar[r]_(0.7){f'} & \bA^1 \\
& \,\,R \ar@{_(->}[ul]^{\rm inc} \ar[rr]_{\th'\vert_R}
\ar[dd]^{\th\vert_R} \ddrrtwocell_{}\omit^{}\omit{^{\eta\,\,}} && \,\,T'
\ar[dd]_{\vp'} \ar@{_(->}[ul]^(0.7){\rm inc} \\ U \ar[d]^f \\
\bA^1 & \,\,T \ar@{_(->}[ul]^{\rm inc} \ar[rr]^\vp && {X,\!\!} }
\end{equation*}
and the induced morphism $R\ra T\t_XT'$ is \'etale and maps $r\mapsto q$. Furthermore $f\ci\th-f'\ci\th'\in I_{R,V}^2,$ where $I_{R,V}\subseteq\O_V$ is the ideal of functions vanishing on\/~$R\subseteq V$.
\label{sa2prop2}
\end{prop}

\section{A truncation functor to d-critical stacks}
\label{sa3}

Section \ref{sa31} summarizes the theory of algebraic d-critical
loci (on classical $\K$-schemes) from \cite{Joyc2}, and the
truncation functor from $-1$-shifted symplectic derived $\K$-schemes
to algebraic d-critical loci from \cite[\S 6]{BBJ}. Section
\ref{sa32} explains the generalization of d-critical loci to Artin
stacks from \cite{Joyc2}, called d-critical stacks. Our main result
Theorem \ref{sa3thm6}, extending the truncation functor of \cite[\S
6]{BBJ} to (derived) Artin stacks, is stated in \S\ref{sa33} and
proved in~\S\ref{sa34}.

\subsection{Algebraic d-critical loci, the $\K$-scheme case}
\label{sa31}

We now review the main ideas and results in the last author's theory
\cite{Joyc2} of ({\it algebraic\/}) {\it d-critical loci}. Readers
are referred to \cite{Joyc2} for more details. Throughout $\K$ is an
algebraically closed field with $\mathop{\rm char}\K\ne 2$, though
we will take $\mathop{\rm char}\K=0$ in Theorem \ref{sa3thm6} and
its corollaries.

Let $X$ be a $\K$-scheme. Then \cite[Th.~2.1 \& Prop.~2.3]{Joyc2}
define a natural sheaf of $\K$-algebras $\cS_X$ on $X$ in either the
Zariski or \'etale topologies (we will use the \'etale version for
the extension to Artin stacks), with the following properties:
\begin{itemize}
\setlength{\itemsep}{0pt}
\setlength{\parsep}{0pt}
\item[(a)] Suppose $R\subseteq X$ is Zariski open, $U$ is a
smooth $\K$-scheme, and $i:R\hookra U$ a closed embedding.
Define an ideal $I_{R,U}\subseteq i^{-1}(\O_U)$ by the exact
sequence
\begin{equation*}
\smash{\xymatrix@C=30pt{ 0 \ar[r] & I_{R,U} \ar[r] &
i^{-1}(\O_U) \ar[r]^{i^\sharp} & \O_X\vert_R \ar[r] & 0, }}
\end{equation*}
where $\O_X,\O_U$ are the sheaves of regular functions on $X,U$.
Then there is an exact sequence on $R$, where
$\d:f+I_{R,U}^2\mapsto \d f+I_{R,U}\cdot i^{-1}(T^*U)$
\begin{equation*}
\xymatrix@C=20pt{ 0 \ar[r] & \cS_X\vert_R
\ar[rr]^(0.4){\io_{R,U}} &&
\displaystyle\frac{i^{-1}(\O_U)}{I_{R,U}^2} \ar[rr]^(0.4)\d &&
\displaystyle\frac{i^{-1}(T^*U)}{I_{R,U}\cdot i^{-1}(T^*U)}\,. }
\end{equation*}
\item[(b)] Let $R\subseteq S\subseteq X$ be Zariski open, $U,V$
be smooth $\K$-schemes, $i:R\hookra U,$ $j:S\hookra V$ closed
embeddings, and $\Phi:U\ra V$ a morphism with $\Phi\ci
i=j\vert_R:R\ra V$. Then the following diagram of sheaves on $R$
commutes:
\e
\begin{gathered}
\xymatrix@C=12pt@R=13pt{ 0 \ar[r] & \cS_X\vert_R \ar[d]^\id
\ar[rrr]^(0.4){\io_{S,V}\vert_R} &&&
\displaystyle\frac{j^{-1}(\O_V)}{I_{S,V}^2}\Big\vert_R
\ar@<-2ex>[d]^{i^{-1}(\Phi^\sharp)} \ar[rr]^(0.4)\d &&
\displaystyle\frac{j^{-1}(T^*V)}{I_{S,V}\cdot
j^{-1}(T^*V)}\Big\vert_R \ar@<-2ex>[d]^{i^{-1}(\d\Phi)} \\
0 \ar[r] & \cS_X\vert_R \ar[rrr]^(0.4){\io_{R,U}} &&&
\displaystyle\frac{i^{-1}(\O_U)}{I_{R,U}^2} \ar[rr]^(0.4)\d &&
\displaystyle\frac{i^{-1}(T^*U)}{I_{R,U}\cdot i^{-1}(T^*U)}\,.
}\!\!\!\!\!\!\!{}
\end{gathered}
\label{sa3eq1}
\e
\item[(c)] There is a natural decomposition
$\cS_X=\cSz_X\op\K_X,$ where\/ $\K_X$ is the constant sheaf on
$X$ with fibre $\K,$ and\/ $\cSz_X\subset\cS_X$ is the kernel of
the composition
\begin{equation*}
\smash{\xymatrix@C=40pt{ \cS_X \ar[r] & \O_X
\ar[r]^(0.47){i_X^\sharp} & \O_{X^\red}, }}
\end{equation*}
with $i_X:X^\red\hookra X$ the reduced $\K$-subscheme of~$X$.
\item[(d)] Let $\phi:X\ra Y$ be a morphism of $\K$-schemes. Then
there is a unique morphism $\phi^\star:\phi^{-1}(\cS_Y)\ra
\cS_X$ of sheaves of $\K$-algebras on $X,$ which maps
$\phi^{-1}(\cSz_Y)\ra \cSz_X,$ such that if $R\subseteq X,$
$S\subseteq Y$ are Zariski open with $\phi(R)\subseteq S,$ $U,V$
are smooth schemes, $i:R\hookra U,$ $j:S\hookra V$ are closed
embeddings, and $\Phi:U\ra V$ is a morphism with $\Phi\ci
i=j\ci\phi\vert_R:R\ra V,$ then as for \eq{sa3eq1} the following
diagram of sheaves on $R$ commutes:
\end{itemize}
\e
\begin{gathered}
\xymatrix@R=20pt@C=9pt{ 0 \ar[r] & \phi^{-1}(\cS_Y)\vert_R
\ar[rrr]_(0.45){\phi^{-1}(\io_{S,V})\vert_R}
\ar[d]^{\phi^\star\vert_R} &&& {\frac{\ts\phi^{-1}\ci
j^{-1}(\O_V)\vert_R}{\ts\phi^{-1}(I_{S,V}^2)\vert_R}}
\ar[d]^{i^{-1}(\Phi^\sharp)} \ar[rr]_(0.42){\phi^{-1}(\d)} &&
\frac{\ts \phi^{-1}(j^{-1}(T^*V))\vert_R}{\ts
\phi^{-1}(I_{S,V}\cdot j^{-1}(T^*V))\vert_R}
\ar[d]_{i^{-1}(\d\Phi)} \\
0 \ar[r] & \cS_X\vert_R \ar[rrr]^(0.45){\io_{R,U}} &&&
{\frac{\ts i^{-1}(\O_U)}{\ts I_{R,U}^2}} \ar[rr]^(0.42)\d &&
\frac{\ts i^{-1}(T^*U)}{\ts I_{R,U}\cdot i^{-1}(T^*U)}\,.
}\!\!\!\!\!\!\!\!\!\!\!\!\!\!\!{}
\end{gathered}
\label{sa3eq2}
\e
\begin{itemize}
\setlength{\itemsep}{0pt}
\setlength{\parsep}{0pt}
\item[(e)] If $X\,{\buildrel\phi\over\longra}\,Y\,
{\buildrel\psi\over\longra}\,Z$ are smooth morphisms of
$\K$-schemes, then
\begin{equation*}
\smash{(\psi\ci\phi)^\star=\phi^\star\ci\phi^{-1}(\psi^\star):
(\psi\ci\phi)^{-1}(\cS_Z)=\phi^{-1}\ci\psi^{-1}(\cS_Z)\longra
\cS_X.}
\end{equation*}
If\/ $\phi:X\ra Y$ is $\id_X:X\ra X$ then
$\id_X^\star=\id_{\cS_X}:\id_X^{-1}(\cS_X)=\cS_X\ra\cS_X$.
\end{itemize}

Following \cite[Def.~2.5]{Joyc2} we define algebraic d-critical
loci:

\begin{dfn} An ({\it algebraic\/}) {\it d-critical locus\/} over a
field $\K$ is a pair $(X,s)$, where $X$ is a $\K$-scheme and $s\in
H^0(\cSz_X)$, such that for each $x\in X$, there exists a Zariski
open neighbourhood $R$ of $x$ in $X$, a smooth $\K$-scheme $U$, a
regular function $f:U\ra\bA^1=\K$, and a closed embedding
$i:R\hookra U$, such that $i(R)=\Crit(f)$ as $\K$-subschemes of $U$,
and $\io_{R,U}(s\vert_R)=i^{-1}(f)+I_{R,U}^2$. We call the quadruple
$(R,U,f,i)$ a {\it critical chart\/} on~$(X,s)$.

Let $(X,s)$ be an algebraic d-critical locus, and $(R,U,f,i)$ a
critical chart on $(X,s)$. Let $U'\subseteq U$ be Zariski open, and
set $R'=i^{-1}(U')\subseteq R$, $i'=i\vert_{R'}:R'\hookra U'$, and
$f'=f\vert_{U'}$. Then $(R',U',f',i')$ is a critical chart on
$(X,s)$, and we call it a {\it subchart\/} of $(R,U,f,i)$. As a
shorthand we write~$(R',U',f',i')\subseteq (R,U,f,i)$.

Let $(R,U,f,i),(S,V,g,j)$ be critical charts on $(X,s)$, with
$R\subseteq S\subseteq X$. An {\it embedding\/} of $(R,U,f,i)$ in
$(S,V,g,j)$ is a locally closed embedding $\Phi:U\hookra V$ such
that $\Phi\ci i=j\vert_R$ and $f=g\ci\Phi$. As a shorthand we write
$\Phi: (R,U,f,i)\hookra(S,V,g,j)$. If $\Phi:(R,U,f,i)\hookra
(S,V,g,j)$ and $\Psi:(S,V,g,j)\hookra(T,W,h,k)$ are embeddings, then
$\Psi\ci\Phi:(R,U,i,e)\hookra(T,W,h,k)$ is also an embedding.

A {\it morphism\/} $\phi:(X,s)\ra (Y,t)$ of d-critical loci
$(X,s),(Y,t)$ is a $\K$-scheme morphism $\phi:X\ra Y$ with
$\phi^\star(t)=s$. This makes d-critical loci into a category.

\label{sa3def1}
\end{dfn}

There is also a complex analytic version, but we will
not discuss it. Here are~\cite[Prop.s 2.8, 2.30, Th.s 2.20, 2.28,
Def.~2.31,  Rem 2.32 \& Cor.~2.33]{Joyc2}:

\begin{prop} Let\/ $\phi:X\ra Y$ be a smooth morphism of\/
$\K$-schemes. Suppose $t\in H^0(\cSz_Y),$ and set\/
$s:=\phi^\star(t)\in H^0(\cSz_X)$. If\/ $(Y,t)$ is a d-critical
locus, then\/ $(X,s)$ is a d-critical locus, and\/
$\phi:(X,s)\ra(Y,t)$ is a morphism of d-critical loci. Conversely,
if also $\phi:X\ra Y$ is surjective, then $(X,s)$ a d-critical locus
implies $(Y,t)$ is a d-critical locus.
\label{sa3prop1}
\end{prop}

\begin{thm} Suppose\/ $(X,s)$ is an algebraic d-critical locus, and\/
$(R,U,f,i),\ab(S,V,g,j)$ are critical charts on $(X,s)$. Then for each\/
$x\in R\cap S\subseteq X$ there exist subcharts
$(R',U',f',i')\subseteq(R,U,f,i),$ $(S',V',g',j')\subseteq
(S,V,g,j)$ with\/ $x\in R'\cap S'\subseteq X,$ a critical chart\/
$(T,W,h,k)$ on $(X,s),$ and embeddings $\Phi:(R',U',f',i')\hookra
(T,W,h,k),$ $\Psi:(S',V',g',j')\hookra(T,W,h,k)$.
\label{sa3thm2}
\end{thm}

\begin{thm} Let\/ $(X,s)$ be an algebraic d-critical locus, and\/
$X^\red\subseteq X$ the associated reduced\/ $\K$-subscheme. Then
there exists a line bundle $K_{X,s}$ on $X^\red$ which we call the
\begin{bfseries}canonical bundle\end{bfseries} of\/ $(X,s),$ which
is natural up to canonical isomorphism, and is characterized by the
following properties:
\begin{itemize}
\setlength{\itemsep}{0pt}
\setlength{\parsep}{0pt}
\item[{\bf(a)}] For each $x\in X^\red,$ there is a canonical
isomorphism
\e
\smash{\ka_x:K_{X,s}\vert_x\,{\buildrel\cong\over\longra}\,
\bigl(\La^{\rm top}T_x^*X\bigr){}^{\ot^2},}
\label{sa3eq3}
\e
where $T_xX$ is the Zariski tangent space of\/ $X$ at\/~$x$.
\item[{\bf(b)}] If\/ $(R,U,f,i)$ is a critical chart on
$(X,s),$ there is a natural isomorphism
\e
\smash{\io_{R,U,f,i}:K_{X,s}\vert_{R^\red}\longra
i^*\bigl(K_U^{\ot^2}\bigr)\vert_{R^\red},}
\label{sa3eq4}
\e
where $K_U=\La^{\dim U}T^*U$ is the canonical bundle of\/ $U$ in
the usual sense.

\item[{\bf(c)}] In the situation of\/ {\bf(b)\rm,} let\/ $x\in R$.
Then we have an exact sequence
\e
\begin{gathered}
\smash{{}\!\!\!\!\xymatrix@C=22pt@R=15pt{ 0 \ar[r] & T_xX
\ar[r]^(0.4){\d i\vert_x} & T_{i(x)}U
\ar[rr]^(0.53){\Hess_{i(x)}f} && T_{i(x)}^*U \ar[r]^{\d
i\vert_x^*}  & T_x^*X  \ar[r] & 0, }}
\end{gathered}
\label{sa3eq5}
\e
and the following diagram commutes:
\begin{equation*}
\xymatrix@C=150pt@R=11pt{ *+[r]{K_{X,s}\vert_x}
\ar[dr]_{\io_{R,U,f,i}\vert_x} \ar[r]_(0.55){\ka_x} &
*+[l]{\bigl(\La^{\rm top}T_x^*X\bigr){}^{\ot^2}}
\ar[d]_(0.45){\al_{x,R,U,f,i}} \\
& *+[l]{K_U\vert_{i(x)}^{\ot^2},\!\!\!} }
\end{equation*}
where $\al_{x,R,U,f,i}$ is induced by taking top exterior powers
in\/~\eq{sa3eq5}.
\end{itemize}
\label{sa3thm3}
\end{thm}

\begin{prop} Suppose $\phi:(X,s)\ra(Y,t)$ is a morphism of
d-critical loci with\/ $\phi:X\ra Y$ smooth, as in Proposition\/
{\rm\ref{sa3prop1}}. The \begin{bfseries}relative cotangent
bundle\end{bfseries} $T^*_{X/Y}$ is a vector bundle of mixed rank on
$X$ in the exact sequence of coherent sheaves on $X\!:$
\e
\smash{\xymatrix@C=35pt{0 \ar[r] & \phi^*(T^*Y) \ar[r]^(0.55){\d\phi^*} & T^*X \ar[r] & T^*_{X/Y} \ar[r] & 0. }}
\label{sa3eq6}
\e
There is a natural isomorphism of line bundles on $X^\red\!:$
\e
\smash{\Up_\phi:\phi\vert_{X^\red}^* (K_{Y,t})\ot\bigl(\La^{\rm
top}T^*_{X/Y}\bigr)\big\vert_{X^\red}^{\ot^2}
\,{\buildrel\cong\over\longra}\,K_{X,s},}
\label{sa3eq7}
\e
such that for each\/ $x\in X^\red$ the following diagram of
isomorphisms commutes:
\e
\begin{gathered}
\xymatrix@C=160pt@R=13pt{ *+[r]{K_{Y,t}
\vert_{\phi(x)}\ot\bigl(\La^{\rm top}T^*_{X/Y}\vert_x\bigr)^{\ot^2}}
\ar[r]_(0.7){\Up_\phi\vert_x} \ar[d]^{\ka_{\phi(x)}\ot\id} &
*+[l]{K_{X,s}\vert_x}
\ar[d]_{\ka_x} \\
*+[r]{\bigl(\La^{\rm top}T_{\phi(x)}^*Y\bigr)^{\ot^2}\ot
\bigl(\La^{\rm top}T^*_{X/Y}\vert_x\bigr)^{\ot^2}}
\ar[r]^(0.7){\up_x^{\ot^2}} & *+[l]{\bigl(\La^{\rm
top}T_x^*X\bigr)^{\ot^2},\!\!{}} }
\end{gathered}
\label{sa3eq8}
\e
where $\ka_x,\ka_{\phi(x)}$ are as in {\rm\eq{sa3eq3},} and\/
$\up_x:\La^{\rm top}T_{\phi(x)}^*Y\ot \La^{\rm top}T^*_{X/Y}
\vert_x\ra\La^{\rm top}T_x^*X$ is obtained by restricting
\eq{sa3eq6} to $x$ and taking top exterior powers.
\label{sa3prop2}
\end{prop}

\begin{dfn} Let $(X,s)$ be an algebraic d-critical locus, and
$K_{X,s}$ its canonical bundle from Theorem \ref{sa3thm3}. An {\it
orientation\/} on $(X,s)$ is a choice of square root line bundle
$K_{X,s}^{1/2}$ for $K_{X,s}$ on $X^\red$. That is, an orientation
is a line bundle $L$ on $X^\red$, together with an isomorphism
$L^{\ot^2}=L\ot L\cong K_{X,s}$. A d-critical locus with an
orientation will be called an {\it oriented d-critical locus}.
\label{sa3def2}
\end{dfn}

\begin{rem} In view of equation \eq{sa3eq3}, one might hope to
define a canonical orientation $K_{X,s}^{1/2}$ for a d-critical
locus $(X,s)$ by $K_{X,s}^{1/2}\big\vert_x=\La^{\rm top}T_x^*X$ for
$x\in X^\red$. However, {\it this does not work}, as the spaces
$\La^{\rm top}T_x^*X$ do not vary continuously with $x\in X^\red$ if
$X$ is not smooth. An example in \cite[Ex.~2.39]{Joyc2} shows that
d-critical loci need not admit orientations.
\label{sa3rem1}
\end{rem}

In the situation of Proposition \ref{sa3prop2}, the factor
$(\La^{\rm top}T^*_{X/Y})\vert_{X^\red}^{\ot^2}$ in \eq{sa3eq7} has
a natural square root $(\La^{\rm top}T^*_{X/Y})\vert_{X^\red}$. Thus
we deduce:

\begin{cor} Let\/ $\phi:(X,s)\ra(Y,t)$ be a morphism of
d-critical loci with\/ $\phi:X\ra Y$ smooth. Then each orientation
$K_{Y,t}^{1/2}$ for\/ $(Y,t)$ lifts to a natural orientation
$K_{X,s}^{1/2}=\phi\vert_{X^\red}^*(K_{Y,t}^{1/2})\ot(\La^{\rm
top}T^*_{X/Y}) \vert_{X^\red}$ for~$(X,s)$.
\label{sa3cor1}
\end{cor}

The following result from \cite{BBJ} will be generalized to stacks
in Theorem \ref{sa3thm6}.

\begin{thm}[Bussi, Brav and Joyce {\cite[Th.~6.6]{BBJ}}] Suppose\/
$(\bs X,\om)$ is a $-1$-shifted symplectic derived scheme in the
sense of Pantev et al.\ {\rm\cite{PTVV}} over an algebraically
closed field\/ $\K$ of characteristic zero, and let\/ $X=t_0(\bs X)$
be the associated classical\/ $\K$-scheme of\/ ${\bs X}$. Then $X$
extends naturally to an algebraic d-critical locus\/ $(X,s)$. The
canonical bundle $K_{X,s}$ from Theorem\/ {\rm\ref{sa3thm3}} is
naturally isomorphic to the determinant line bundle $\det(\bL_{\bs
X})\vert_{X^\red}$ of the cotangent complex\/ $\bL_{\bs X}$
of\/~$\bs X$.
\label{sa3thm4}
\end{thm}

\subsection{Extension to Artin stacks, and d-critical stacks}
\label{sa32}

In \cite[\S 2.7--\S 2.8]{Joyc2} we extend the material of
\S\ref{sa31} from $\K$-schemes to Artin $\K$-stacks. We work in the
context of the theory of {\it sheaves on Artin stacks} by Laumon and
Moret-Bailly \cite[\S\S 12, 13, 15, 18]{LaMo}, including {\it
quasi-coherent}, {\it coherent}, and {\it constructible\/} sheaves,
and their derived categories. Unfortunately, Laumon and Moret-Bailly
wrongly assume that 1-morphisms of algebraic stacks induce morphisms
of lisse-\'etale topoi, so parts of their theory concerning
pullbacks, etc., are unsatisfactory. Olsson \cite{Olss} rewrites the
theory, correcting this mistake. Laszlo and Olsson
\cite{LaOl1,LaOl2,LaOl3} study derived categories of constructible
sheaves, and perverse sheaves, on Artin stacks, in more detail.

All of \cite{LaMo,LaOl1,LaOl2,LaOl3,Olss} work with sheaves on Artin
stacks in the {\it lisse-\'etale topology}. We will not define these
directly, but instead quote an alternative description from Laumon
and Moret-Bailly \cite{LaMo} that we find more convenient.

\begin{prop}[Laumon and Moret-Bailly \cite{LaMo}] Let\/ $X$ be an
Artin\/ $\K$-stack. The category of sheaves of sets on $X$ in the
lisse-\'etale topology is equivalent to the category $\Sh(X)$
defined as follows:
\smallskip

\noindent{\bf(A)} Objects $\cA$ of\/ $\Sh(X)$ comprise the following
data:
\begin{itemize}
\setlength{\itemsep}{0pt}
\setlength{\parsep}{0pt}
\item[{\bf(a)}] For each\/ $\K$-scheme $T$ and smooth\/ $1$-morphism
$t:T\ra X$ in $\Art_\K,$ we are given a sheaf of sets $\cA(T,t)$
on $T,$ in the \'etale topology.
\item[{\bf(b)}] For each\/ $2$-commutative diagram in
$\Art_\K\!:$
\e
\begin{gathered}
\xymatrix@C=50pt@R=1pt{ & U \ar[ddr]^u \\
\rrtwocell_{}\omit^{}\omit{^{\eta}} && \\
T  \ar[uur]^{\phi} \ar[rr]_t && X, }
\end{gathered}
\label{sa3eq9}
\e
where $T,U$ are schemes and\/ $t: T\ra X,$ $u:U\ra X$ are
smooth\/ $1$-morphisms in $\Art_\K,$ we are given a morphism
$\cA(\phi,\eta):\phi^{-1} (\cA(U,u)) \ra\cA(T,t)$ of \'etale
sheaves of sets on $T$.
\end{itemize}
This data must satisfy the following conditions:
\begin{itemize}
\setlength{\itemsep}{0pt}
\setlength{\parsep}{0pt}
\item[{\bf(i)}] If\/ $\phi:T\ra U$ in {\bf(b)} is \'etale, then
$\cA(\phi,\eta)$ is an isomorphism.
\item[{\bf(ii)}] For each\/ $2$-commutative diagram in $\Art_\K\!:$
\begin{equation*}
\xymatrix@C=70pt@R=1pt{ & V \ar[ddr]^v \\
\rrtwocell_{}\omit^{}\omit{^{\ze}} && \\
U  \ar[uur]^{\psi} \ar[rr]_(0.3)u && X, \\
\urrtwocell_{}\omit^{}\omit{^{\eta}} && \\
T \ar[uu]_{\phi} \ar@/_/[uurr]_t }
\end{equation*}
with $T,U,V$ schemes and\/ $t,u,v$ smooth, we must have
\begin{align*}
\cA\bigl(\psi\ci\phi,(\ze*\id_{\phi})\od\eta\bigr)
&=\cA(\phi,\eta)\ci\phi^{-1}(\cA(\psi,\ze))\quad\text{as
morphisms}\\
(\psi\ci\phi)^{-1}(\cA(V,v))&=\phi^{-1}\ci
\psi^{-1}(\cA(V,v))\longra\cA(T,t).
\end{align*}
\end{itemize}

\noindent{\bf(B)} Morphisms $\al:\cA\ra\cB$ of\/ $\Sh(X)$ comprise a
morphism $\al(T,t):\cA(T,t)\ra\cB(T,t)$ of \'etale sheaves of sets
on a scheme $T$ for all smooth\/ $1$-morphisms $t:T\ra X,$ such that
for each diagram \eq{sa3eq9} in {\bf(b)} the following commutes:
\begin{equation*}
\xymatrix@C=120pt@R=15pt{*+[r]{\phi^{-1}(\cA(U,u))}
\ar[d]^{\phi^{-1}(\al(U,u))} \ar[r]_(0.55){\cA(\phi,\eta)} &
*+[l]{\cA(T,t)} \ar[d]_{\al(T,t)} \\
*+[r]{\phi^{-1}(\cB(U,u))}
\ar[r]^(0.55){\cB(\phi,\eta)} & *+[l]{\cB(T,t).\!{}} }
\end{equation*}

\noindent{\bf(C)} Composition of morphisms $\cA\,{\buildrel\al
\over\longra}\,\cB\,{\buildrel\be\over\longra}\,\cC$ in $\Sh(X)$ is
$(\be\ci\al)(T,t)=\ab\be(T,t)\ab\ci\ab\al(T,t)$. Identity morphisms
$\id_\cA:\cA\ra\cA$ are $\id_\cA(T,t)=\id_{\cA(T,t)}$.
\smallskip

The analogue of all the above also holds for (\'etale) sheaves of\/
$\K$-vector spaces, sheaves of\/ $\K$-algebras, and so on, in place
of (\'etale) sheaves of sets.

Furthermore, the analogue of all the above holds for quasi-coherent
sheaves, (or coherent sheaves, or vector bundles, or line bundles)
on $X,$ where in {\bf(a)} $\cA(T,t)$ becomes a quasi-coherent sheaf
(or coherent sheaf, or vector bundle, or line bundle) on $T,$ in
{\bf(b)} we replace $\phi^{-1}(\cA(U,u))$ by the pullback\/
$\phi^*(\cA(U,u))$ of quasi-coherent sheaves (etc.), and\/
$\cA(\phi,\eta),\ab\al(T,t)$ become morphisms of quasi-coherent
sheaves (etc.) on\/~$T$.

We can also describe \begin{bfseries}global sections\end{bfseries}
of sheaves on Artin $\K$-stacks in the above framework: a global
section $s\in H^0(\cA)$ of\/ $\cA$ in part\/ {\bf(A)} assigns a
global section $s(T,t)\in H^0(\cA(T,t))$ of\/ $\cA(T,t)$ on\/ $T$
for all smooth\/ $t:T\ra X$ from a scheme $T,$ such that\/
$\cA(\phi,\eta)^*(s(U,u))=s(T,t)$ in $H^0(\cA(T,t))$ for all\/
$2$-commutative diagrams \eq{sa3eq9} with\/ $t,u$ smooth.
\label{sa3prop3}
\end{prop}

In the rest of the paper we will use the notation of Proposition
\ref{sa3prop3} for sheaves of all kinds on Artin $\K$-stacks. In
\cite[Cor.~2.52]{Joyc2} we generalize the sheaves $\cS_X,\cSz_X$ in
\S\ref{sa31} to Artin $\K$-stacks:

\begin{prop} Let\/ $X$ be an Artin $\K$-stack, and write\/
$\Sh(X)_\Kalg$ and\/ $\Sh(X)_\Kvect$ for the categories of sheaves
of\/ $\K$-algebras and\/ $\K$-vector spaces on $X$ defined in
Proposition\/ {\rm\ref{sa3prop3}}. Then:
\begin{itemize}
\setlength{\itemsep}{0pt}
\setlength{\parsep}{0pt}
\item[{\bf(a)}] We may define canonical objects\/ $\cS_X$ in
both\/ $\Sh(X)_\Kalg$ and\/ $\Sh(X)_\Kvect$ by
$\cS_X(T,t):=\cS_T$ for all smooth morphisms $t:T\ra X$ for
$T\in\Sch_\K,$ for $\cS_T$ as in\/ {\rm\S\ref{sa31}} taken to be
a sheaf of\/ $\K$-algebras (or $\K$-vector spaces) on $T$ in the
\'etale topology, and\/ $\cS_X(\phi,\eta)
:=\phi^\star:\phi^{-1}(\cS_X(U,u))=\phi^{-1}(\cS_U)\ra\cS_T
=\cS_X(T,t)$ for all\/ $2$-commutative diagrams \eq{sa3eq9} in
$\Art_\K$ with\/ $t,u$ smooth, where $\phi^\star$ is as in
{\rm\S\ref{sa31}}.
\item[{\bf(b)}] There is a natural decomposition
$\cS_X\!=\!\K_X\!\op\!\cSz_X$ in $\Sh(X)_\Kvect$ induced by the
splitting $\cS_X(T,t)\!=\!\cS_T\!=\!\K_T\op\cSz_T$ in
{\rm\S\ref{sa31},} where $\K_X$ is a sheaf of\/ $\K$-subalgebras
of\/ $\cS_X$ in $\Sh(X)_\Kalg,$ and\/ $\cSz_X$ a sheaf of
ideals~in\/~$\cS_X$.
\end{itemize}
\label{sa3prop4}
\end{prop}

Here \cite[Def. 2.53]{Joyc2} is the generalization of Definition
\ref{sa3def1} to Artin stacks.

\begin{dfn} A {\it d-critical stack\/} $(X,s)$ is an Artin
$\K$-stack $X$ and a global section $s\in H^0(\cSz_X)$, where
$\cSz_X$ is as in Proposition \ref{sa3prop4}, such that
$\bigl(T,s(T,t)\bigr)$ is an algebraic d-critical locus in the sense
of Definition \ref{sa3def1} for all smooth morphisms $t:T\ra X$
with~$T\in\Sch_\K$.
\label{sa3def3}
\end{dfn}

In \cite[Prop.~2.54]{Joyc2} we give a convenient way to understand
d-critical stacks $(X,s)$ in terms of d-critical structures on an
atlas $t:T\ra X$ for~$X$.

\begin{prop} Suppose we are given a $2$-commutative diagram in $\Art_\K\!:$
\e
\begin{gathered}
\xymatrix@C=90pt@R=11pt{ *+[r]{U} \ar[d]^{\pi_1} \ar[r]_(0.3){\pi_2}
\drtwocell_{}\omit^{}\omit{^{\eta}} &
*+[l]{
T} \ar[d]_t \\
*+[r]{T} \ar[r]^(0.7)t & *+[l]{X,\!} }
\end{gathered}
\label{sa3eq10}
\e
where $X$ is an Artin $\K$-stack, $T,U$ are $\K$-schemes, $t,\pi_1,\pi_2$ are smooth\/ $1$-morphisms, $t:T\ra X$ is surjective, and the\/ $1$-morphism\/ $U\ra T\t_{t,X,t}T$ induced by \eq{sa3eq10} is surjective. For instance, this happens if\/ $U\rra T$ is a groupoid in $\K$-schemes, and $X=[U\rra T]$ the associated groupoid stack. Then:
\begin{itemize}
\setlength{\itemsep}{0pt}
\setlength{\parsep}{0pt}
\item[{\bf(i)}] Let\/ $\cS_X$ be as in Proposition\/
{\rm\ref{sa3prop4},} and\/ $\cS_T,\cS_U$ be as in
{\rm\S\ref{sa31},} regarded as sheaves on $T,U$ in the \'etale
topology, and define $\pi_i^\star:\pi_i^{-1}(\cS_T)\ra \cS_U$ as
in {\rm\S\ref{sa31}} for $i=1,2$. Consider the map
$t^*:H^0(\cS_X)\ra H^0(\cS_T)$ mapping $t^*:s\mapsto s(T,t)$.
This is injective, and induces a bijection
\e
\smash{t^*:H^0(\cS_X)\,{\buildrel\cong\over\longra}\,\bigl\{s'\in
H^0(\cS_T):\text{$\pi_1^\star(s')= \pi_2^\star(s')$ in
$H^0(\cS_U)$}\bigr\}.}
\label{sa3eq11}
\e
The analogue holds for $\cSz_X,\cSz_T,\cSz_U$.
\item[{\bf(ii)}] Suppose $s\in H^0(\cSz_X),$ so that\/ $t^*(s)\in
H^0(\cSz_T)$ with\/ $\pi_1^\star\ci t^*(s)=\pi_2^\star\ci
t^*(s)$. Then $(X,s)$ is a d-critical stack if and only if\/
$\bigl(T,t^*(s)\bigr)$ is an algebraic d-critical locus, and
then\/ $\bigl(U,\pi_1^\star\ci t^*(s)\bigr)$ is also an
algebraic d-critical locus.
\end{itemize}
\label{sa3prop5}
\end{prop}

In \cite[Ex.~2.55]{Joyc2} we consider quotient stacks $X=[T/G]$.

\begin{ex} Suppose an algebraic $\K$-group $G$ acts on a $\K$-scheme
$T$ with action $\mu:G\t T\ra T$, and write $X$ for the quotient
Artin $\K$-stack $[T/G]$. Then as in \eq{sa3eq10} there is a natural
2-Cartesian diagram
\begin{equation*}
\xymatrix@C=110pt@R=11pt{ *+[r]{G\t T} \ar[d]^{\pi_T}
\ar[r]_(0.3){\mu} \drtwocell_{}\omit^{}\omit{^{\eta}} &
*+[l]{T} \ar[d]_t \\
*+[r]{T} \ar[r]^(0.6)t & *+[l]{X=[T/G],\!{}} }
\end{equation*}
where $t:T\ra X$ is a smooth atlas for $X$. If $s'\in H^0(\cSz_T)$
then $\pi_1^\star(s')=\pi_2^\star(s')$ in \eq{sa3eq11} becomes
$\pi_T^\star(s')=\mu^\star(s')$ on $G\t T$, that is, $s'$ is
$G$-invariant. Hence, Proposition \ref{sa3prop5} shows that
d-critical structures $s$ on $X=[T/G]$ are in 1-1 correspondence
with $G$-invariant d-critical structures $s'$ on~$T$.
\label{sa3ex}
\end{ex}

Here \cite[Th.~2.56]{Joyc2} is an analogue of Theorem~\ref{sa3thm3}.

\begin{thm} Let\/ $(X,s)$ be a d-critical stack. Using the
description of quasi-coherent sheaves on $X^\red$ in Proposition
{\rm\ref{sa3prop3}} there is a line bundle $K_{X,s}$ on the
reduced\/ $\K$-substack\/ $X^\red$ of\/ $X$ called the
\begin{bfseries}canonical bundle\end{bfseries} of\/ $(X,s),$
unique up to canonical isomorphism, such that:
\begin{itemize}
\setlength{\itemsep}{0pt}
\setlength{\parsep}{0pt}
\item[{\bf(a)}] For each point\/ $x\in X^\red\subseteq X$ we
have a canonical isomorphism
\e
\smash{\ka_x:K_{X,s}\vert_x\,{\buildrel\cong\over\longra}\,
\bigl(\La^{\rm top}T_x^*X\bigr)^{\ot^2} \ot\bigl(\La^{\rm
top}\fIso_x(X)\bigr)^{\ot^2},}
\label{sa3eq12}
\e
where $T_x^*X$ is the Zariski cotangent space of\/ $X$ at\/ $x,$
and\/ $\fIso_x(X)$ the Lie algebra of the isotropy group
(stabilizer group) $\Iso_x(X)$ of\/ $X$ at\/~$x$.
\item[{\bf(b)}] If\/ $T$ is a $\K$-scheme and\/ $t:
T\ra X$ a smooth $1$-morphism, so that\/ $t^\red:T^\red\ra
X^\red$ is also smooth, then there is a natural isomorphism of
line bundles on $T^\red\!:$
\e
\Ga_{T,t}:K_{X,s}(T^\red,t^\red)\,{\buildrel\cong\over\longra}\,
K_{T,s(T,t)}\ot \bigl(\La^{\rm top}T^*_{
T/X}\bigr)\big\vert_{T^\red}^{\ot^{-2}}.
\label{sa3eq13}
\e
Here $\bigl(T,s(T,t)\bigr)$ is an algebraic d-critical locus by
Definition\/ {\rm\ref{sa3def3},} and\/ $K_{T,s(T,t)}\ra T^\red$
is its canonical bundle from Theorem\/~{\rm\ref{sa3thm3}}.
\item[{\bf(c)}] If\/ $t:T\ra X$ is a smooth
$1$-morphism, we have a distinguished triangle
in~$D_{\qcoh}(T)\!:$
\e
\smash{\xymatrix@C=40pt{ t^*(\bL_X) \ar[r]^(0.55){\bL_t} & \bL_T \ar[r]
& T^*_{T/X} \ar[r] & t^*(\bL_X)[1], }}
\label{sa3eq14}
\e
where $\bL_T,\bL_X$ are the cotangent complexes of\/ $T,X,$
and\/ $T^*_{T/X}$ the relative cotangent bundle of\/ $t: T\ra
X,$ a vector bundle of mixed rank on\/ $T$. Let\/ $p\in
T^\red\subseteq T,$ so that\/ $t(p):=t\ci p\in X$. Taking the
long exact cohomology sequence of\/ \eq{sa3eq14} and restricting
to $p\in T$ gives an exact sequence
\e
\smash{0 \longra T^*_{t(p)}X \longra T^*_pT \longra T^*_{ T/X}\vert_p
\longra \fIso_{t(p)}(X)^* \longra 0.}
\label{sa3eq15}
\e
Then the following diagram commutes:
\begin{equation*}
\xymatrix@!0@C=98pt@R=35pt{*+[r]{K_{X,s}\vert_{t(p)}}
\ar[d]^{\ka_{t(p)}} \ar@{=}[r] & K_{X,s}(T^\red,t^\red)\vert_p
\ar[rr]_(0.3){\Ga_{T,t}\vert_p} && *+[l]{K_{T,s(T,t)}\vert_p\ot
\bigl(\La^{\rm top}T^*_{\smash{
T/X}}\bigr)\big\vert_p^{\ot^{-2}}} \ar[d]_{\ka_p\ot\id} \\
*+[r]{\bigl(\La^{\rm top}T_{t(p)}^*X\bigr)^{\ot^2}\!\!\ot\!\bigl(\La^{\rm
top}\fIso_{t(p)}(X)\bigr)^{\ot^2}} \ar[rrr]^(0.54){\al_p^2} &&&
*+[l]{\bigl(\La^{\rm top}T^*_pT\bigr)^{\ot^2}\!\!\ot\!
\bigl(\La^{\rm top}T^*_{T/X}\bigr) \big\vert_p^{\ot^{-2}},} }
\end{equation*}
where $\ka_p,\ka_{t(p)},\Ga_{T,t}$ are as in {\rm\eq{sa3eq3},
\eq{sa3eq12}} and\/ {\rm\eq{sa3eq13},} respectively, and\/
$\al_p:\La^{\rm top}T_{t(p)}^*X\ot\La^{\rm
top}\fIso_{t(p)}(X)\,{\buildrel\cong\over\longra}\,\La^{\rm
top}T^*_pT\ot\La^{\rm top}T^*_{T/X}\vert^{-1}_p$ is induced by
taking top exterior powers in\/~\eq{sa3eq15}.
\end{itemize}
\label{sa3thm5}
\end{thm}

Here \cite[Def.~2.57]{Joyc2} is the analogue of
Definition~\ref{sa3def2}:

\begin{dfn} Let $(X,s)$ be a d-critical stack, and $K_{X,s}$
its canonical bundle from Theorem \ref{sa3thm5}. An {\it
orientation\/} on $(X,s)$ is a choice of square root line bundle
$K_{X,s}^{1/2}$ for $K_{X,s}$ on $X^\red$. That is, an orientation
is a line bundle $L$ on $X^\red$, together with an isomorphism
$L^{\ot^2}=L\ot L\cong K_{X,s}$. A d-critical stack with an
orientation will be called an {\it oriented d-critical stack}.
\label{sa3def4}
\end{dfn}

Let $(X,s)$ be an oriented d-critical stack. Then for each smooth
$t:T\ra X$ we have a square root $K_{X,s}^{1/2} (T^\red,t^\red)$.
Thus by \eq{sa3eq13}, $K_{X,s}^{1/2}(T^\red,t^\red)\ot (\La^{\rm
top}\bL_{\smash{T/X}})\vert_{T^\red}$ is a square root for
$K_{T,s(T,t)}$. This proves~\cite[Lem.~2.58]{Joyc2}:

\begin{lem} Let\/ $(X,s)$ be a d-critical stack. Then an
orientation $K_{X,s}^{1/2}$ for $(X,s)$ determines a canonical
orientation\/ $K_{T,s(T,t)}^{1/2}$ for the algebraic
d-critical locus $\bigl(T,s(T,t)\bigr),$ for all smooth\/ $t: T\ra
X$ with\/ $T$ a $\K$-scheme.
\label{sa3lem1}
\end{lem}

\subsection[From $-1$-shifted symplectic derived stacks to
d-critical stacks]{From $-1$-shifted symplectic stacks to d-critical
stacks}
\label{sa33}

Here is the main result of this section, the analogue of Theorem
\ref{sa3thm4} from~\cite{BBJ}.

\begin{thm} Let\/ $\K$ be an algebraically closed field of
characteristic zero, $(\bX,\om_\bX)$ a $-1$-shifted symplectic
derived Artin $\K$-stack, and\/ $X=t_0(\bX)$ the corresponding
classical Artin $\K$-stack. Then there exists a unique d-critical
structure $s\in H^0(\cSz_X)$ on $X,$ making $(X,s)$ into a
d-critical stack, with the following properties:
\begin{itemize}
\setlength{\itemsep}{0pt}
\setlength{\parsep}{0pt}
\item[{\bf(a)}] Let\/ $U,$ $f:U\ra\bA^1,$ $T=\Crit(f)$ and\/
$\vp:T\ra X$ be as in Corollary\/ {\rm\ref{sa2cor1},} with\/
$f\vert_{T^\red}=0$. As in {\rm\S\ref{sa31},} there is a unique
$s_T\in H^0(\cSz_T)$ on $T$ with\/
$\io_{T,U}(s_T)=i^{-1}(f)+I_{T,U}^2,$ and\/ $(T,s_T)$ is an
algebraic d-critical locus. Then $s(T,\vp)=s_T$ in
$H^0(\cSz_T)$.
\item[{\bf(b)}] The canonical bundle $K_{X,s}$ of\/ $(X,s)$ from
Theorem\/ {\rm\ref{sa3thm5}} is naturally isomorphic to the
restriction $\det(\bL_\bX)\vert_{X^\red}$ to $X^\red\subseteq
X\subseteq\bX$ of the determinant line bundle $\det(\bL_\bX)$ of
the cotangent complex\/ $\bL_\bX$ of\/~$\bX$.
\end{itemize}
\label{sa3thm6}
\end{thm}

We can think of Theorem \ref{sa3thm6} as defining a {\it truncation
functor}
\begin{align*}
F:\bigl\{&\text{$\iy$-category of $-1$-shifted symplectic derived
Artin $\K$-stacks $(\bX,\om_\bX)$}\bigr\}\\
&\longra\bigl\{\text{2-category of d-critical stacks $(X,s)$ over
$\K$}\bigr\}.
\end{align*}

Let $Y$ be a Calabi--Yau 3-fold over $\K$, and $\cM$ a classical
moduli $\K$-stack of coherent sheaves in $\coh(Y)$, or complexes of
coherent sheaves in $D^b\coh(Y)$. There is a natural obstruction
theory $\phi:\cE^\bu\ra\bL_\cM$ on $\cM$, where $\cE^\bu\in
D_{\qcoh}(\cM)$ is perfect in the interval $[-2,1]$, and
$h^i(\cE^\bu)\vert_F\cong\Ext^{1-i}(F,F)^*$ for each $\K$-point
$F\in\cM$, regarding $F$ as an object in $\coh(Y)$ or $D^b\coh(Y)$.
Now in derived algebraic geometry $\cM=t_0(\bcM)$ for $\bcM$ the
corresponding derived moduli $\K$-stack, and
$\phi:\cE^\bu\ra\bL_\cM$ is $\bL_{t_0}:\bL_{\bcM}
\vert_\cM\ra\bL_\cM$. Pantev et al.\ \cite[\S 2.1]{PTVV} prove
$\bcM$ has a $-1$-shifted symplectic structure $\om$. Thus Theorem
\ref{sa3thm6} implies:

\begin{cor} Suppose $Y$ is a Calabi--Yau\/ $3$-fold over\/ $\K$ of
characteristic zero, and\/ $\cM$ a classical moduli\/ $\K$-stack of
coherent sheaves $F$ in $\coh(Y),$ or complexes of coherent sheaves
$F^\bu$ in $D^b\coh(Y)$ with $\Ext^{<0}(F^\bu,F^\bu)=0,$ with
obstruction theory\/ $\phi:\cE^\bu\ra\bL_\cM$. Then $\cM$ extends
naturally to an algebraic d-critical locus $(\cM,s)$. The canonical
bundle $K_{\cM,s}$ from Theorem\/ {\rm\ref{sa3thm5}} is naturally
isomorphic to $\det(\cE^\bu)\vert_{\cM^\red}$.
\label{sa3cor2}
\end{cor}

\subsection{Proof of Theorem \ref{sa3thm6}}
\label{sa34}

Let $(\bX,\om_\bX)$ be a $-1$-shifted symplectic derived Artin
$\K$-stack, with $\mathop{\rm char}\K=0$, and $X=t_0(\bX)$. For each
$p\in X$, Corollary \ref{sa2cor1} gives data $T=\Crit(f:U\ra\bA^1)$
with $f\vert_{T^\red}=0$, $t\in T$ and a smooth $\vp:T\ra X$ with
$\vp(t)=p$. Choose $U_j,f_j,T_j,\vp_j$ from Corollary \ref{sa2cor1}
for $j$ in an indexing set $J$, such that $\coprod_{j\in
J}\vp_j:\coprod_{j\in J}T_j\ra X$ is surjective. Then $\coprod_{j\in
J}\vp_j:\coprod_{j\in J}T_j\ra X$ is a smooth atlas for $X$. As in
\S\ref{sa31}, there is a unique $s_j\in H^0(\cSz_{T_j})$ with
$\io_{T_j,U_j}(s_j)= i_j^{-1}(f_j)+I_{T_j,U_j}^2,$ and $(T_j,s_j)$
is an algebraic d-critical locus for each~$j\in J$.

Let $j,k\in J$, and $q\in T_j\t_{\vp_j,X,\vp_k}T_k$. Applying
Proposition \ref{sa2prop2} gives a smooth $\K$-scheme $V_{jk}$, a closed
$\K$-subscheme $R_{jk}\subseteq V_{jk}$, a point $r\in R_{jk},$ and morphisms $\th_{jk}:V_{jk}\ra U_j,$ $\th_{jk}':V_{jk}\ra U_k$ with $\th_{jk}(R_{jk})\subseteq T_j$, $\th_{jk}'(R_{jk})\subseteq T_k$, such that the following diagram 2-commutes in $\Art_\K\!:$
\begin{equation*}
\xymatrix@!0@C=80pt@R=25pt{ V_{jk} \ar[rr]_{\th_{jk}'} \ar[dd]^{\th_{jk}} && U_k \ar[r]_(0.7){f_k} & \bA^1 \\
& \,\,R_{jk} \ar@{_(->}[ul]^{i_{jk}} \ar[rr]^{\th_{jk}'\vert_{R_{jk}}}
\ar[dd]_{\th_{jk}\vert_{R_{jk}}} \ddrrtwocell_{}\omit^{}\omit{^{\eta_{jk}\,\,\,\,\,\,\,}} && \,\,T_k \ar[dd]_{\vp_k}
\ar@{_(->}[ul]^(0.7){i_k} \\ U_j \ar[d]^{f_j} \\
\bA^1 & \,\,T_j \ar@{_(->}[ul]^{i_j} \ar[rr]^{\vp_j} && {X,\!\!} }
\end{equation*}
where $i_j,i_k,i_{jk}$ are the inclusions, and the induced morphism $R_{jk}\ra T_j\t_XT_k$ is \'etale and maps $r\mapsto q$, and $f_j\ci\th_{jk}-f_k\ci\th_{jk}'\in I_{R_{jk},V_{jk}}^2$. 

As we can do this for each $q\in T_j\t_XT_k$, we can choose a family of such $V_{jk}^l,R_{jk}^l,\th_{jk}^l,\th_{jk}^{\prime l},\eta_{jk}^l,i_{jk}^l$ for $l\in K_{jk}$, where $K_{jk}$ is an indexing set, such that the induced morphism $\coprod_{l\in K_{jk}}R_{jk}^l\ra T_j\t_XT_k$ is \'etale and surjective. We apply Proposition \ref{sa3prop5} to the 2-commutative diagram
\begin{equation*}
\xymatrix@C=180pt@R=17pt{ *+[r]{\coprod_{j,k\in J}\coprod_{l\in K_{jk}} R_{jk}^l}
\ar[d]^{\amalg_{j,k,l}\th_{jk}^l\vert_{R_{jk}^l}} \ar[r]_(0.7){\amalg_{j,k,l}\th_{jk}^{\prime l}\vert_{R_{jk}^l}}
\drtwocell_{}\omit^{}\omit{^{\amalg_{j,k,l}\eta_{jk}^l\,\,\,\,\,\,\,\,\,\,\,\,\,\,\,\,\,\,\,\,}} &
*+[l]{\coprod_{k\in J}T_k} \ar[d]_{\amalg_k\vp_k} \\
*+[r]{\coprod_{j\in J}T_j} \ar[r]^(0.7){\amalg_j\vp_j} & *+[l]{X.\!\!} }
\end{equation*}
Here $\coprod_j\vp_j:\coprod_jT_j\ra X$ is smooth and surjective, and $\coprod_{j,k,l}R_{jk}^l\ra \coprod_{j,k}T_j\t_XT_k$ \'etale and surjective, so the hypotheses of Proposition \ref{sa3prop5} hold.

Now for all $j,k\in J$ and $l\in K_{jk}$, in the notation of \S\ref{sa31}, we have\vskip -15pt
\begin{align*}
&\io_{R_{jk}^l,V_{jk}^l}\ci\th_{jk}^l\vert_{R_{jk}^l}^\star(s_j)=(i_{jk}^l)^{-1}(\th_{jk}^{l\sharp})\ci
\th_{jk}^l\vert_{R_{jk}^l}^{-1}(\io_{T_j,U_j}(s_j))\\
&=\th_{jk}^l\vert_{R_{jk}^l}^{-1}(i_j^{-1}(f_j)+I_{T_j,U_j}^2)
=(i_{jk}^l)^{-1}(f_j\ci\th_{jk}^l+I_{R_{jk}^l,V_{jk}^l}^2)\\
&=(i_{jk}^l)^{-1}(f_k\ci\th_{jk}^{l\prime}+I_{R_{jk}^l,V_{jk}^l}^2)
=\th_{jk}^{\prime l}\vert_{R_{jk}^l}^{-1}(i_k^{-1}(f_k)+I_{T_k,U_k}^2)\\
&=(i_{jk}^l)^{-1}(\th_{jk}^{\prime l\sharp})\ci
\th_{jk}^{\prime l}\vert_{R_{jk}^l}^{-1}(\io_{T_k,U_k}(s_k))=\io_{R_{jk}^l,V_{jk}^l}\ci\th_{jk}^{\prime l}\vert_{R_{jk}^l}^\star(s_k),
\end{align*}
using \eq{sa3eq2} in the first and seventh steps, the definitions of
$s_j,s_k$ in the second and sixth, and $f_j\ci\th_{jk}^l-f_k\ci\th_{jk}^{\prime l}\in I_{R_{jk}^l,V_{jk}^l}^2$ in the fourth. As $\io_{R_{jk}^l,V_{jk}^l}$ is injective, this implies that $\th_{jk}^l\vert_{R_{jk}^l}^\star(s_j)=\th_{jk}^{\prime l}\vert_{R_{jk}^l}^\star(s_k)$ in $H^0(\cSz_{R_{jk}^l})$. Since this holds for all $j,k,l$, we see that 
$(\coprod_{j,k,l}\th_{jk}^l\vert_{R_{jk}^l})^\star(\coprod_js_j)=(\coprod_{j,k,l}\th_{jk}^{\prime l}\vert_{R_{jk}^l})^\star(\coprod_ks_k)$ in $H^0(\cSz_{\amalg_{j,k,l}R_{jk}^l})$. Therefore Proposition \ref{sa3prop5}(i) shows that there exists a unique $s\in H^0(\cS_X)$ with $s\bigl(\coprod_jT_j,\coprod_j\vp_j\bigr)=\coprod_js_j$, that is, with $s(T_j,\vp_j)=s_j$ for all $j\in J$. Also, as $\bigl(\coprod_jT_j,\coprod_js_j\bigr)$ is an algebraic d-critical locus, Proposition \ref{sa3prop5}(ii) shows that $(X,s)$ is a d-critical stack.

To show $s\in H^0(\cS_X)$ is independent of the choice of data $J,U_j,f_j,T_j,\ab\vp_j,\ab K_{jk},\ab V_{jk}^l,\ab R_{jk}^l,\ab\th_{jk}^l,\th_{jk}^{\prime l},\eta_{jk}^l,i_{jk}^l$, suppose $J',U'_{j'},f'_{j'},\ldots$ is another set of choices yielding $s'\in H^0(\cS_X)$ with $s'(T'_{j'},\vp'_{j'})=s'_{j'}$ for all $j'\in J'$. Applying the same argument with $J''=J\amalg J'$ and data $U_j,f_j,T_j,\vp_j,$ $j\in J$ and $U'_{j'},f'_{j'},T'_{j',}\vp'_{j'},$ $j'\in J'$, with $K_{jk}''=K_{jk}$, $V_{jk}^{\prime\prime l}=V_{jk}^l,\ldots$ for $j,k\in J\subset J''$, and $K_{j'k'}''=K'_{j'k'}$, $V_{j'k'}^{\prime\prime l}=V_{j'k'}^{\prime l},\ldots$ for $j',k'\in J'\subset J''$, and the remaining $K_{jk}'',V_{jk}^{\prime\prime l},\ldots$ arbitrary, yields a third section $s''\in H^0(\cS_X)$ satisfying $s''(T_j,\vp_j)=s_j$ for all $j\in J$ and $s''(T'_{j'},\vp'_{j'})=s'_{j'}$ for all $j'\in J'$. So the uniqueness property of $s,s'$ gives $s=s''=s'$, and $s$ is independent of the choice of data~$J,U_j,f_j,\ldots.$

Let $U,$ $f:U\ra\bA^1,$ $T=\Crit(f)$ and $\vp:T\ra X$ be as in
Corollary\ \ref{sa2cor1}, with $f\vert_{T^\red}=0$. By defining
$s\in H^0(\cS_X)$ above using data $J,U_j,f_j,\ldots$ chosen
such that $U_j=U$, $f_j=f$, $T_j=T$, $\vp_j=\vp$ for some $j\in J$,
which is allowed as $s$ is independent of this choice, we see that
$s(T,\vp)=s_T$ in $H^0(\cSz_T)$. This proves
Theorem~\ref{sa3thm6}(a).

For part (b), let $U$, $f:U\ra\bA^1$, $T=\Crit(f)$ and $\vp:T\ra X$
be as in Corollary \ref{sa2cor1}, with $i:T\hookra U$ the inclusion,
so that $s(T,\vp)=s_T$ in $H^0(\cSz_T)$ with
$\io_{T,U}(s_T)=i^{-1}(f)+I_{T,U}^2$ by (a). Then $(T,U,f,i)$ is a
critical chart on the algebraic d-critical locus $(T,s_T)$, so
Theorem \ref{sa3thm3}(b) gives an isomorphism
\e
\smash{\io_{T,U,f,i}:K_{T,s_T}\longra
i^*\bigl(K_U^{\ot^2}\bigr)\vert_{T^\red}.}
\label{sa3eq16}
\e

The data in Corollary \ref{sa2cor1} come from Theorem
\ref{sa2thm6}(a),(b) with $k=-1$, but with different notation. To
distinguish the two, we write `$\,\,\check{}\,\,$' over notation
from Theorem \ref{sa2thm6}. Then Theorem \ref{sa2thm6}(a),(b) give
affine derived $\K$-schemes $\bs{\check U},\bs{\check V}$, a
$-1$-shifted symplectic structure $\check\om_B$ on $\bs{\check V}$,
and morphisms $\bs{\check\imath}:\bs{\check U}\ra\bs{\check V},$
$\bs{\check\vp}:\bs{\check U}\ra\bX$ such that $\bs{\check\vp}{}^*
(\om_\bX)\sim\bs{\check\imath}^*(\check\om_B)$, and
$\check\imath=t_0(\bs{\check\imath}):\check U=t_0(\bs{\check U})\ra
\check V=t_0(\bs{\check V})$ is an isomorphism on the classical
schemes. These are related to the data of Corollary \ref{sa2cor1} by
$\bs{\check V}$ is the derived critical locus
$\bs\Crit(f:U\ra\bA^1)$, and $\check V$ the classical critical locus
$T=\Crit(f),$ and $\vp=\check\vp\ci\check\imath^{-1}:T=\check V\ra
X$.

We have standard fibre sequences on $\bs{\check U}:$\vskip -15pt
\begin{align*}
\xymatrix@C=30pt@R=8pt{ \bs{\check\vp}^*(\bL_{\bX})
\ar[r]_(0.55){\bL_{\bs{\check\vp}}} & \bL_{\smash{\bs{\check U}}} \ar[r] &
\bL_{\smash{\bs{\check U}/\bX}} \ar[r] & \bs{\check\vp}^*(\bL_{\bX})[1], } \\
\xymatrix@C=32pt{ \bs{\check\imath}^*(\bL_{\smash{\bs{\check V}}})
\ar[r]^(0.55){\bL_{\bs{\check\imath}}} & \bL_{\smash{\bs{\check U}}} \ar[r]
& \bL_{\smash{\bs{\check U}/\bs{\check V}}} \ar[r] &
\bs{\check\imath}^*(\bL_{\smash{\bs{\check V}}})[1]. }
\end{align*}
Taking determinants gives natural isomorphisms of line bundles on
$\bs{\check U}$:
\e
\begin{split}
\det\bL_{\smash{\bs{\check U}}}&\cong
\bs{\check\vp}^*(\det\bL_{\bX})\ot
\det\bL_{\smash{\bs{\check U}/\bX}},\\
\det\bL_{\smash{\bs{\check U}}}&\cong \bs{\check\imath}^*
(\det\bL_{\smash{\bs{\check V}}})\ot \det\bL_{\smash{\bs{\check
U}/\bs{\check V}}}.
\end{split}
\label{sa3eq17}
\e
Equation \eq{sa2eq14} gives $\bL_{\smash{\bs{\check U}/\bs{\check
V}}} \simeq \bT_{\smash{\bs{\check U}/\bX}}[2]$. So taking
determinants we have
\e
\smash{\det\bL_{\smash{\bs{\check U}/\bs{\check V}}}\cong
\det\bT_{\smash{\bs{\check U}/\bX}}\cong (\det\bL_{\smash{\bs{\check
U}/\bX}})^*.}
\label{sa3eq18}
\e

Combining \eq{sa3eq17}--\eq{sa3eq18} and restricting to $\check
U=t_0(\bs{\check U})\subseteq\bs{\check U}$ yields
\e
\smash{\check\vp^*\bigl(\det\bL_{\bX}\vert_X\bigr)\cong
{\check\imath}{}^*\bigl(\det\bL_{\smash{\bs{\check V}}}\vert_{\check
V}\bigr) \ot\bigl(\det\bL_{\smash{\bs{\check U}/\bX}}\vert_{\check
U}\bigr){}^{\ot^{-2}}.}
\label{sa3eq19}
\e
Since $\bs{\check\vp}:\bs{\check U}\ra\bX$ is smooth, so is
$\check\vp:\check U\ra X$, and
\e
\smash{\bL_{\bs{\check U}/\bX}\vert_{\check U}\cong\bL_{\check U/X}\cong
T^*_{\check U/X}.}
\label{sa3eq20}
\e
As $\check\imath:\check U\ra\check V=T$ is an isomorphism, we may
apply $(\check\imath^{-1})^*$ to \eq{sa3eq19}. Using \eq{sa3eq20}
and $(\check\imath^{-1})^*\ci {\check\imath}{}^*=\id$,
$(\check\imath^{-1})^*\ci\check\vp^*=\vp^*$ as
$\vp=\check\vp\ci\check\imath^{-1}$ gives
\e
\smash{\vp^*\bigl(\det\bL_{\bX}\vert_X\bigr)\cong
\bigl(\det\bL_{\smash{\bs{\check V}}}\vert_T\bigr)\ot
(\check\imath^{-1})^*\bigl(\La^{\rm top}T^*_{\check
U/X}\bigr){}^{\ot^{-2}}.}
\label{sa3eq21}
\e

Since $\bs{\check V}=\bs\Crit(f:U\ra\bA^1)$, we have
\begin{equation*}
\smash{\bL_{\smash{\bs{\check V}}}\vert_T\simeq\bigl[
\xymatrix@C=30pt{TU\vert_T \ar[r]^{\pd^2f\vert_T} & T^*U\vert_T}\bigr],}
\end{equation*}
with $TU\vert_T$ in degree $-1$ and $T^*U\vert_T$ in degree 0.
Therefore
\e
\smash{\det\bL_{\smash{\bs{\check V}}}\vert_T\cong
i^*\bigl(K_U^{\ot^2}\bigr).}
\label{sa3eq22}
\e
Also, as $\check\imath^{-1}:T\ra\hat U$ is an isomorphism, we have
\e
\smash{(\check\imath^{-1})^*(T^*_{\check U/X})\cong T^*_{T/X}.}
\label{sa3eq23}
\e
Combining \eq{sa3eq21}--\eq{sa3eq23}, restricting to $T^\red$ and
using \eq{sa3eq16} gives
\e
\smash{(\vp^\red)^*\bigl(\det\bL_{\bX}\vert_{X^\red}\bigr)\cong K_{T,s_T}\ot \bigl(\La^{\rm
top}T^*_{T/X}\bigr)\big\vert{}_{T^\red}^{\ot^{-2}}.}
\label{sa3eq24}
\e
Substituting in the isomorphism $\Ga_{T,\vp}$ in Theorem
\ref{sa3thm5}(b) from the smooth morphism $\vp:T\ra X$ gives a
canonical isomorphism of line bundles on $T^\red$:
\e
\smash{(\vp^\red)^*\bigl(\det\bL_{\bX}\vert_{X^\red}\bigr) \cong
K_{X,s}(T^\red,\vp^\red).}
\label{sa3eq25}
\e

This establishes the isomorphism $K_{X,s}\cong\det(\bL_\bX)
\vert_{X^\red}$ in Theorem \ref{sa3thm6}(b) evaluated on
$(T^\red,\vp^\red)$ for any $U,f,T,\vp$ coming from Corollary
\ref{sa2cor1}. Such $\vp^\red:T^\red\ra X^\red$ form an open cover
of $X^\red$ in the smooth topology. To prove the isomorphism
$K_{X,s}\cong\det(\bL_\bX)\vert_{X^\red}$ globally and complete the
proof, there are two possible methods. Firstly, we could prove that
given two choices $U,f,T,\vp$ and $U',f',T',\vp'$ in Corollary
\ref{sa2cor1}, the corresponding isomorphisms \eq{sa3eq25} agree on
the overlap $T^\red\t_{\vp^\red,X,\vp^{\prime\red}}T^{\prime\red}$.

But as we are dealing with line bundles on a reduced stack $X^\red$,
there is a second, easier way: we can show that for each $t\in
T^\red$ with $\vp^\red(t)=x\in X^\red$, the isomorphism
$\det\bL_{\bX}\vert_x\cong K_{X,s}\vert_x$ from restricting
\eq{sa3eq25} to $t$ depends only on $x\in X^\red$, and not on the choice
of $U,f,T,\vp,t$. This holds as by Theorem \ref{sa3thm5}(a) we have
an isomorphism
\e
\smash{K_{X,s}\vert_x\cong\bigl(\La^{\rm top}T_x^*X\bigr)^{\ot^2}
\ot\bigl(\La^{\rm top}\fIso_x(X)\bigr)^{\ot^2}.}
\label{sa3eq26}
\e
Since $\bL_\bX$ is perfect in the interval $[-2,1]$, we have
\e
\smash{\det\bL_{\bX}\vert_x\cong \ts\bigot_{i=-2}^1 \bigl(\La^{\rm top} H^i(\bL_{\bX}\vert_x)\bigr)^{(-1)^i},}
\label{sa3eq27}
\e
where we have canonical isomorphisms
\e
\begin{aligned}
H^0(\bL_{\bX}\vert_x)&\cong T_x^*X, & H^1(\bL_{\bX}\vert_x)&\cong
\fIso_x(X)^*,\\
H^{-1}(\bL_{\bX}\vert_x)&\cong T_xX, &
H^{-2}(\bL_{\bX}\vert_x)&\cong \fIso_x(X),
\end{aligned}
\label{sa3eq28}
\e
the first line holding for any derived Artin stack $\bX$, and the
second line from $H^i(\bL_{\bX}\vert_x)\cong
H^{-1-i}(\bL_{\bX}\vert_x)^*$ as $(\bX,\om_\bX)$ is $-1$-shifted
symplectic.

Combining \eq{sa3eq26}--\eq{sa3eq28} gives a canonical isomorphism
$\det\bL_{\bX}\vert_x\cong K_{X,s}\vert_x$ depending only on $x\in
X^\red$. Following through \eq{sa3eq16}--\eq{sa3eq25} restricted to
$t\in T^\red$ with $\vp^\red(t)=x$, we find that the restriction of
\eq{sa3eq25} to $t$ gives the same isomorphism. This completes the
proof of Theorem~\ref{sa3thm6}(b).

\section{Perverse sheaves on d-critical stacks}
\label{sa4}

In \cite[Th.~6.9]{BBDJS}, given in Theorem \ref{sa4thm2} below, we
constructed a natural perverse sheaf $P^\bu_{X,s}$ on an oriented
algebraic d-critical locus $(X,s)$. The main result of this section,
Theorem \ref{sa4thm3}, generalizes this to oriented d-critical
stacks.

We begin in \S\ref{sa41} with some background on perverse sheaves on
schemes. Section \ref{sa42} recalls results from \cite{BBDJS}, and
proves in Proposition \ref{sa4prop3} a smooth pullback property of
the $P^\bu_{X,s}$ in Theorem \ref{sa4thm2}. Section \ref{sa43} discusses perverse sheaves on Artin stacks. Once we have set up all the notation, Theorem \ref{sa4thm3} in \S\ref{sa44} follows almost immediately from Theorem \ref{sa4thm2} and Proposition \ref{sa4prop3}. In this section the base field $\K$ may be algebraically closed with
$\mathop{\rm char}\K\ne 2$, except in Corollaries \ref{sa4cor1} and
\ref{sa4cor2} when we require $\mathop{\rm char}\K=0$ to apply the
results of~\S\ref{sa3}.

\subsection{Perverse sheaves on schemes}
\label{sa41}

We will assume the reader is familiar with the theory of perverse sheaves on $\C$-schemes and $\K$-schemes. An introduction to perverse sheaves on schemes suited to our purposes can be found in \cite[\S 2]{BBDJS}, and our definitions and notation follow that paper. Here is a brief survival guide:
\begin{itemize}
\setlength{\itemsep}{0pt}
\setlength{\parsep}{0pt}
\item We work throughout this section over an algebraically closed field $\K$ with $\mathop{\rm char}\K\ne 2$, for instance $\K=\C$. All $\K$-schemes $X,Y,Z,\ldots$ are assumed separated and of finite type.
\item We work with constructible complexes and perverse sheaves over a commutative base ring $A$. The allowed rings $A$ depend on the field $\K$. For $\K=\C$ one can define perverse sheaves using the complex analytic topology as in Dimca \cite{Dimc}, and then $A$ can be essentially arbitrary, e.g. $A=\Q$~or~$\Z$.

If $\K\ne\C$ then one must define perverse sheaves using the \'etale topology, as in Beilinson, Bernstein and Deligne \cite{BBD}. Then the allowed possibilities are $A$ with $\mathop{\rm char}A>0$ coprime to $\mathop{\rm char}\K,$ or the $l$-adic integers $\Z_l$, or the $l$-adic rationals $\Q_l$, or its algebraic closure $\bar\Q_l$, for $l$ a prime coprime to $\mathop{\rm char}\K$. We will refer to all these possibilities as $l$-{\it adic perverse sheaves}.
\item For a $\K$-scheme $X$, one defines the derived category $D^b_c(X)$ of {\it constructible complexes\/} of $A$-modules on $X$. There is a natural t-structure on $D^b_c(X)$, with heart the abelian category $\Perv(X)$ of {\it perverse sheaves\/} on $X$.
\item An example of a constructible complex on $X$ is the constant
sheaf $A_X$ with fibre $A$ at each point. If $X$ is smooth then $A_X[\dim X]\in\Perv(X)$.
\item Grothendieck's ``six operations on sheaves'' $f^*,f^!,Rf_*,Rf_!,
{\cal RH}om,\smash{\otL}$ act on the categories $D^b_c(X)$. There is a functor $\bD_X:D^b_c(X)\ra D^b_c(X)^{\rm op}$
with $\bD_X\ci\bD_X\cong\id:D^b_c(X)\ra D^b_c(X)$, called {\it
Verdier duality}.
\item Let $U$ be a $\K$-scheme and $f:U\ra\bA^1$ a regular
function, and write $U_0$ for the subscheme $f^{-1}(0)\subseteq U$.
Then one can define the {\it nearby cycle functor\/} $\psi_f^p:D^b_c(U)\ra D^b_c(U_0)$ and the {\it vanishing cycle functor\/} $\phi_f^p:D^b_c(U)\ra
D^b_c(U_0)$. Both map $\Perv(U)\ra\Perv(U_0)$. 
\item Let $U$ be a smooth $\K$-scheme and $f:U\ra\bA^1$ a regular
function, and write $X=\Crit(f)$. Then we have a decomposition $X=\coprod_{c\in f(X)}X_c$, where $X_c\subseteq X$ is the open and closed subscheme of points $p\in X$
with $f(p)=c$. It turns out that $\phi_f^p(A_U[\dim U])$ is
supported on $X_0\subseteq X\subseteq U$. 

Following \cite[\S 2.4]{BBDJS}, define the {\it perverse sheaf of vanishing cycles\/ $\PV_{U,f}^\bu$ of\/} $U,f$ in $\Perv(X)$ or $\Perv(U)$ to be $\PV_{U,f}^\bu=\bigop_{c\in
f(X)}\phi_{f-c}^p(A_U[\dim U])\vert_{X_c}$. We also define a canonical {\it Verdier duality isomorphism\/}
\begin{equation*}
\smash{\si_{U,f}:\PV_{U,f}^\bu\,{\buildrel\cong\over\longra}\,
\bD_X\bigl(\PV_{U,f}^\bu\bigr)}
\end{equation*}
and {\it twisted monodromy
operator\/}
\begin{equation*}
\smash{\tau_{U,f}:\PV_{U,f}^\bu\,{\buildrel\cong\over\longra}\,\PV_{U,f}^\bu.}
\end{equation*}
\end{itemize}
Some references are \cite[\S 2]{BBDJS}, Dimca \cite{Dimc} for perverse sheaves on $\C$-schemes, and Beilinson, Bernstein and Deligne \cite{BBD}, Ekedahl \cite{Eked}, Freitag and Kiehl \cite{FrKi}, and Kiehl and Weissauer \cite{KiWe} for perverse sheaves on $\K$-schemes.

The theories of $\cD$-modules on $\K$-schemes, and Saito's mixed Hodge modules on $\C$-schemes, also share this whole package of properties, and our results also generalize to $\cD$-modules and mixed Hodge modules, as in~\cite{BBDJS}.

Here are some results connecting perverse sheaves and smooth
morphisms. Theorem \ref{sa4thm1} (proved in \cite[Th.~3.2.4]{BBD},
see also \cite[\S 2.3]{LaOl1}) is the reason why perverse sheaves
extend to Artin stacks, as we discuss in~\S\ref{sa43}.

\begin{prop} Let\/ $\Phi:X\ra Y$ be a scheme morphism smooth of
relative dimension $d$. Then the (exceptional) inverse image
functors $\Phi^*,\Phi^!:D^b_c(Y)\ra D^b_c(X)$ satisfy
$\Phi^*[d]\cong \Phi^![-d],$ where $\Phi^*[d],\Phi^![-d]$ are
$\Phi^*,\Phi^!$ shifted by $\pm d$. Furthermore
$\Phi^*[d],\Phi^![-d]$ map $\Perv(Y)\ra\Perv(X)$.
\label{sa4prop1}
\end{prop}

\begin{thm} Let\/ $X$ be a scheme. Then perverse sheaves on $X$ form
a \begin{bfseries}stack\end{bfseries} (a kind of sheaf of
categories) on $X$ \begin{bfseries}in the smooth
topology\end{bfseries}.

Explicitly, this means the following. Let $\{u_i:U_i\ra X\}_{i\in
I}$ be a \begin{bfseries}smooth open cover\end{bfseries} for $X,$ so
that\/ $u_i:U_i\ra X$ is a scheme morphism smooth of relative
dimension $d_i$ for\/ $i\in I,$ with\/ $\coprod_iu_i$ surjective.
Write $U_{ij}=U_i\t_{u_i,X,u_j}U_j$ for $i,j\in I$ with projections
\begin{equation*}
\smash{\pi_{ij}^i:U_{ij}\longra U_i,\quad \pi_{ij}^j:U_{ij}\longra U_j,\quad u_{ij}\!=\!u_i\ci\pi_{ij}^i\!=\!u_j\ci\pi_{ij}^j:U_{ij}\!\longra\! X.}
\end{equation*}
Similarly, write $U_{ijk}=U_i\!\t_X\!U_j\!\t_X\!U_k$ for $i,j,k\in
I$ with projections
\begin{gather*}
\pi_{ijk}^{ij}:U_{ijk}\longra U_{ij},\quad
\pi_{ijk}^{ik}:U_{ijk}\longra
U_{ik},\quad \pi_{ijk}^{jk}:U_{ijk}\longra U_{jk},\\
\pi_{ijk}^i:U_{ijk}\ra U_i,\;\> \pi_{ijk}^j:U_{ijk}\ra U_j,\;\>
\pi_{ijk}^k:U_{ijk}\ra U_k, \;\> u_{ijk}:U_{ijk}\ra X,
\end{gather*}
so that\/ $\pi_{ijk}^i=\pi_{ij}^i\ci\pi_{ijk}^{ij},$
$u_{ijk}=u_{ij}\ci\pi_{ijk}^{ij}=u_i\ci\pi_{ijk}^i,$ and so on. All
these morphisms $u_i,\pi_{ij}^i,\ldots,u_{ijk}$ are smooth of known
relative dimensions, so $u_i^*[d_i]\cong u_i^![-d_i]$ maps
$\Perv(X)\ra\Perv(U_i)$ by Proposition\/ {\rm\ref{sa4prop1},} and
similarly for $\pi_{ij}^i,\ldots,u_{ijk}$. With this notation:
\smallskip

\noindent{\bf(i)} Suppose $\cP^\bu,\cQ^\bu\in\Perv(X),$ and we are
given $\al_i:u_i^*[d_i](\cP^\bu)\ra u_i^*[d_i](\cQ^\bu)$ in
$\Perv(U_i)$ for all\/ $i\in I$ such that for all\/ $i,j\in I$ we
have
\begin{equation*}
\smash{(\pi_{ij}^i)^*[d_j](\al_i)= (\pi_{ij}^j)^*[d_i](\al_j):
u_{ij}^*[d_i+d_j](\cP^\bu)\longra u_{ij}^*[d_i+d_j](\cQ^\bu).}
\end{equation*}
Then there is a unique $\al:\cP^\bu\ra\cQ^\bu$ with\/
$\al_i=u_i^*[d_i](\al)$ for all\/~$i\in I$.
\smallskip

\noindent{\bf(ii)} Suppose we are given\/ $\cP^\bu_i\in \Perv(U_i)$
for all\/ $i\in I$ and isomorphisms
$\al_{ij}:(\pi_{ij}^i)^*[d_j](\cP^\bu_i)\ra
(\pi_{ij}^j)^*[d_i](\cP^\bu_j)$ in $\Perv(U_{ij})$ for all\/ $i,j\in
I$ with\/ $\al_{ii}=\id$ and\/
\begin{align*}
(\pi_{ijk}^{jk})^*[d_i](\al_{jk})\ci (\pi_{ijk}^{ij})^*&[d_k](\al_{ij})=
(\pi_{ijk}^{ik})^*[d_j](\al_{ik}):(\pi_{ijk}^i)^*[d_j+d_k](\cP_i)\\
&\longra(\pi_{ijk}^k)^*[d_i+d_j](\cP_k)
\end{align*}
in $\Perv(U_{ijk})$ for all\/ $i,j,k\in I$. Then there exists\/
$\cP^\bu$ in $\Perv(X),$ unique up to canonical isomorphism, with
isomorphisms $\be_i:u_i^*(\cP^\bu)\ra\cP^\bu_i$ for each\/ $i\in I,$
satisfying $\al_{ij}\ci(\pi_{ij}^i)^*(\be_i)=(\pi_{ij}^j)^*(\be_j):
u_{ij}^*(\cP^\bu)\ra(\pi_{ij}^j)^*(\cP^\bu_j)$ for all\/~$i,j\in I$.
\label{sa4thm1}
\end{thm}

\begin{prop} Let\/ $\Phi:U\ra V$ be a scheme morphism smooth of
relative dimension $d$ and\/ $g:V\ra\bA^1$ be regular, and set\/
$f=g\ci\Phi:U\ra\bA^1$. Then
\begin{itemize}
\setlength{\itemsep}{0pt}
\setlength{\parsep}{0pt}
\item[{\bf(a)}] There are natural isomorphisms of functors
$\Perv(V)\ra\Perv(U_0):$
\e
\smash{\Phi_0^*[d]\ci\psi^p_g\cong\psi^p_f\ci\Phi^*[d]\quad\text{and\/}\quad\Phi_0^*[d]\ci\phi^p_g\cong\phi^p_f\ci\Phi^*[d],}
\label{sa4eq1}
\e
where $U_0=f^{-1}(0)\subseteq U,$ $V_0=g^{-1}(0)\subseteq V$
and\/ $\Phi_0=\Phi\vert_{U_0}:U_0\ra V_0$.
\item[{\bf(b)}] Write $X=\Crit(f)$ and\/ $Y=\Crit(g),$ so that
$\Phi\vert_X:X\ra Y$ is smooth of dimension $d$. Then there is a
canonical isomorphism
\e
\smash{\Xi_\Phi:\Phi\vert_X^*[d](\PV_{V,g}^\bu)\,{\buildrel\cong\over
\longra}\,\PV_{U,f}^\bu \quad\text{in\/ $\Perv(X),$}}
\label{sa4eq2}
\e
which identifies $\Phi\vert_X^*[d](\si_{V,g}),\Phi\vert_X^*[d]
(\tau_{V,g})$ with\/ $\si_{U,f},\tau_{U,f}$.
\end{itemize}
\label{sa4prop2}
\end{prop}

\subsection{Perverse sheaves on d-critical loci}
\label{sa42}

Here is \cite[Th.~6.9]{BBDJS}, which we will generalize to stacks in
Theorem \ref{sa4thm3} below. We use the notation of \S\ref{sa31} and
\S\ref{sa41} throughout.

\begin{thm} Let\/ $(X,s)$ be an oriented algebraic d-critical locus
over $\C,$ with orientation $K_{X,s}^{1/2}$. Then for any
well-behaved base ring $A,$ such as $\Z,\Q$ or $\C,$ there exists a
perverse sheaf\/ $P_{X,s}^\bu$ in $\Perv(X)$ over $A,$ which is
natural up to canonical isomorphism, and Verdier duality and
monodromy isomorphisms
\begin{equation*}
\smash{\Si_{X,s}:P_{X,s}^\bu\longra \bD_X\bigl(P_{X,s}^\bu\bigr),\qquad
\Tau_{X,s}:P_{X,s}^\bu\longra P_{X,s}^\bu,} 
\end{equation*}
which are characterized by the following properties:
\begin{itemize}
\setlength{\itemsep}{0pt}
\setlength{\parsep}{0pt}
\item[{\bf(i)}] If\/ $(R,U,f,i)$ is a critical chart on
$(X,s),$ there is a natural isomorphism
\begin{equation*}
\smash{\om_{R,U,f,i}:P_{X,s}^\bu\vert_R\longra
i^*\bigl(\PV_{U,f}^\bu\bigr)\ot_{\Z/2\Z} Q_{R,U,f,i},}
\end{equation*}
where $\pi_{R,U,f,i}:Q_{R,U,f,i}\ra R$ is the principal\/
$\Z/2\Z$-bundle parametrizing local isomorphisms
$\al:K_{X,s}^{1/2}\ra i^*(K_U)\vert_{R^\red}$ with\/ $\al\ot\al=
\io_{R,U,f,i},$ for $\io_{R,U,f,i}$ as in\/ \eq{sa3eq4}.
Furthermore the following commute in $\Perv(R)\!:$
\ea
&\begin{gathered} \xymatrix@!0@C=260pt@R=38pt{
*+[r]{P_{X,s}^\bu\vert_R} \ar[d]^(0.5){\Si_{X,s}\vert_R}
\ar[r]_(0.35){\om_{R,U,f,i}} &
*+[l]{i^*\bigl(\PV_{U,f}^\bu\bigr)\ot_{\Z/2\Z} Q_{R,U,f,i}}
\ar[d]_(0.4){i^*(\si_{U,f})\ot\id_{Q_{R,U,f,i}}} \\
*+[r]{\bD_R\bigl(P_{X,s}^\bu\vert_R\bigr)} &
*+[l]{\begin{aligned}[b]
\ts i^*\bigl(\bD_{\Crit(f)}(\PV_{U,f}^\bu)\bigr)\ot_{\Z/2\Z}
Q_{R,U,f,i}&\\
\ts \cong\bD_R\bigl(i^*(\PV_{U,f}^\bu)\ot_{\Z/2\Z}
Q_{R,U,f,i}\bigr)&,\end{aligned}}
\ar[l]_(0.65){\bD_R(\om_{R,U,f,i})}}
\end{gathered}
\label{sa4eq3}\\
&\begin{gathered} \xymatrix@!0@C=260pt@R=33pt{
*+[r]{P_{X,s}^\bu\vert_R} \ar[d]^{\Tau_{X,s}\vert_R}
\ar[r]_(0.35){\om_{R,U,f,i}} &
*+[l]{i^*\bigl(\PV_{U,f}^\bu\bigr)\ot_{\Z/2\Z} Q_{R,U,f,i}}
\ar[d]_{i^*(\tau_{U,f})\ot\id_{Q_{R,U,f,i}}} \\
*+[r]{P_{X,s}^\bu\vert_R} \ar[r]^(0.35){\om_{R,U,f,i}} &
*+[l]{i^*\bigl(\PV_{U,f}^\bu\bigr)\ot_{\Z/2\Z} Q_{R,U,f,i}.} }
\end{gathered}
\label{sa4eq4}
\ea
\item[{\bf(ii)}] If\/ $\Phi:(R,U,f,i)\hookra(S,V,g,j)$ is an
embedding of critical charts on $(X,s),$ there is a
compatibility condition {\rm\cite[Th.~6.9(ii)]{BBDJS}} between
$\om_{R,U,f,i},\om_{S,V,g,j}$ which we will not give.
\end{itemize}

Analogues hold for oriented algebraic d-critical loci $(X,s)$ over
general fields\/ $\K$ in the settings of\/ $l$-adic perverse sheaves
and of\/ $\cD$-modules, and for oriented algebraic d-critical loci
$(X,s)$ over $\C$ in the setting of mixed Hodge modules.
\label{sa4thm2}
\end{thm}

We prove a proposition on the behaviour of the perverse sheaves
$P_{X,s}^\bu$ of Theorem \ref{sa4thm2} under smooth pullback, which
will be the main ingredient in the proof of our main result
Theorem~\ref{sa4thm3}.

\begin{prop}{\bf(a)} Let\/ $\phi:(X,s)\ra (Y,t)$ be a morphism of
algebraic d-critical loci over $\C,$ in the sense of\/
{\rm\S\ref{sa31},} and suppose $\phi:X\ra Y$ is smooth of relative
dimension $d$. Let\/ $K_{Y,t}^{1/2}$ be an orientation for $(Y,t),$
so that Corollary\/ {\rm\ref{sa3cor1}} defines an induced
orientation\/ $K_{X,s}^{1/2}$ for $(X,s)$. Theorem\/
{\rm\ref{sa4thm2}} defines perverse sheaves
$P_{X,s}^\bu,P_{Y,t}^\bu$ on $X,Y$. Then there is a natural
isomorphism
\e
\smash{\De_\phi:\phi^*[d](P_{Y,t}^\bu)
\,{\buildrel\cong\over\longra}\,P_{X,s}^\bu \quad\text{in
$\Perv(X)$}}
\label{sa4eq5}
\e
which is characterized by the property that if\/ $(R,U,f,i),
(S,V,g,j)$ are critical charts on $(X,s),(Y,t)$ with\/
$\phi(R)\subseteq S$ and\/ $\Phi:U\ra V$ is smooth of relative
dimension $d$ with\/ $f=g\ci\Phi$ and\/ $\Phi\ci i=j\ci\phi,$ then
the following commutes
\e
\begin{gathered} \xymatrix@!0@C=285pt@R=37pt{
*+[r]{\phi\vert_R^*[d](P_{Y,t}^\bu)} \ar[d]^{\De_\phi\vert_R}
\ar[r]_(0.28){\phi\vert_R^*[d](\om_{S,V,g,j})} &
*+[l]{\phi\vert_R^*[d]\bigl(
j^*\bigl(\PV_{V,g}^\bu\bigr)\!\ot_{\Z/2\Z}\!Q_{S,V,g,j}\bigr)}
\ar[d]_{i^*(\Xi_\Phi)\ot \al_\Phi}
\\
*+[r]{P_{X,s}^\bu\vert_R}
\ar[r]^(0.28){\om_{R,U,f,i}} &
*+[l]{i^*\bigl(\PV_{U,f}^\bu\bigr)\ot_{\Z/2\Z} Q_{R,U,f,i},} }
\end{gathered}
\label{sa4eq6}
\e
where $\Xi_\Phi$ is as in \eq{sa4eq2} and\/
$\al_\Phi:\phi\vert_R^*[d](Q_{S,V,g,j})\ra  Q_{R,U,f,i}$ is the
natural isomorphism. Also\/ $\De_\phi$ identifies\/
$\phi^*[d](\Si_{Y,t}),\phi^*[d](\Tau_{Y,t})$
with\/~$\Si_{X,s},\Tau_{X,s}$.
\medskip

\noindent{\bf(b)} If\/ $\psi:(Y,t)\ra (Z,u)$ is another morphism of
algebraic d-critical loci over $\C$ smooth of relative dimension
$e,$ then
\e
\smash{\De_{\psi\ci\phi}=\De_\phi\ci
\phi^*[d](\De_\psi):(\psi\ci\phi)^*[d+e](P_{Z,u}^\bu)
\,{\buildrel\cong\over\longra}\,P_{X,s}^\bu.}
\label{sa4eq7}
\e

\noindent{\bf(c)} Analogues of\/ {\bf(a)\rm,\bf(b)} hold for
algebraic d-critical loci $(X,s)$ over general fields\/ $\K$ in the
settings of\/ $l$-adic perverse sheaves and of\/ $\cD$-modules, and
for algebraic d-critical loci $(X,s)$ over $\C$ in the setting of
mixed Hodge modules.
\label{sa4prop3}
\end{prop}

\begin{proof} Let $\phi:(X,s)\ra (Y,t)$, $d,K_{Y,t}^{1/2},
K_{X,s}^{1/2},P_{X,s}^\bu,P_{Y,t}^\bu$ be as in (a). If $x\in X$
with $\phi(x)=y\in Y$ then the proof of \cite[Prop.~2.8]{Joyc2} shows
that we may choose critical charts $(R,U,f,i), (S,V,g,j)$ on
$(X,s),(Y,t)$ with $x\in R$, $y\in\phi(R)\subseteq S$ of minimal
dimensions $\dim U=\dim T_xX$, $\dim V=\dim T_yY$, and $\Phi:U\ra V$
smooth of relative dimension $d$ with $f=g\ci\Phi$ and~$\Phi\ci
i=j\ci\phi$.

Choose such data $(R_a,U_a,f_a,i_a),(S_a,V_a,g_a,j_a),\Phi_a$ for
$a\in A$, an indexing set, such that $\{R_a:a\in A\}$ is an open
cover for $X$. For each $a\in A$, define an isomorphism
$\De_a:\phi\vert_{R_a}^*[d](P_{Y,t}^\bu)\ra P_{X,s}^\bu\vert_{R_a}$
to make the following diagram of isomorphisms commute, the analogue
of~\eq{sa4eq6}:
\e
\begin{gathered} \xymatrix@!0@C=285pt@R=35pt{
*+[r]{\phi\vert_{R_a}^*[d](P_{Y,t}^\bu)} \ar[d]^{\De_a}
\ar[r]_(0.28){\raisebox{-11pt}{$\st\phi\vert_{R_a}^*[d]
(\om_{S_a,V_a,g_a,j_a})$}} &
*+[l]{\phi\vert_{R_a}^*[d]\bigl(
j_a^*\bigl(\PV_{V_a,g_a}^\bu\bigr)\!\ot_{\Z/2\Z}\!
Q_{S_a,V_a,g_a,j_a}\bigr)}
\ar[d]_{i_a^*(\Xi_{\Phi_a})\ot \al_{\Phi_a}} \\
*+[r]{P_{X,s}^\bu\vert_{R_a}} \ar[r]^(0.28){\om_{R_a,U_a,f_a,i_a}} &
*+[l]{i_a^*\bigl(\PV_{U_a,f_a}^\bu\bigr)\ot_{\Z/2\Z}
Q_{R_a,U_a,f_a,i_a}.} }
\end{gathered}
\label{sa4eq8}
\e
Combining the last part of Proposition \ref{sa4prop2}(b) with
\eq{sa4eq3}--\eq{sa4eq4} shows that this $\De_a$ identifies
$\phi^*[d](\Si_{Y,t})\vert_{R_a},\phi^*[d](\Tau_{Y,t})\vert_{R_a}$
with~$\Si_{X,s}\vert_{R_a},\Tau_{X,s}\vert_{R_a}$.

We claim that for all $a,b\in A$ we have $\De_a\vert_{R_a\cap R_b}=
\De_b\vert_{R_a\cap R_b}$. To prove this, let $x\in R_a\cap R_b$,
with $y=f(x)\in S_a\cap S_b$. By Theorem \ref{sa3thm2} we can
choose subcharts $(R_a',U_a',f_a',i_a')\subseteq(R_a,U_a,f_a,i_a)$,
$(R_b',U_b',f_b',i_b')\subseteq(R_b,U_b,f_b,i_b)$,
$(S_a',V_a',g_a',j_a')\subseteq(S_a,V_a,g_a,j_a)$,
$(S_b',V_b',g_b',j_b')\subseteq(S_b,V_b,g_b,j_b)$ with $x\in
R_a'\cap R_b'$, $y\in S_a'\cap S_b'$, critical charts
$(R_{ab},U_{ab},f_{ab},i_{ab}),(S_{ab},V_{ab},g_{ab},j_{ab})$ on
$(X,s)$, $(Y,t)$, and embeddings
$\Psi_a:(R_a',U_a',f_a',i_a')\hookra(R_{ab},U_{ab},f_{ab},i_{ab})$,
$\Psi_b:(R_b',\ab U_b',\ab f_b',\ab i_b')\hookra(R_{ab},\ab
U_{ab},\ab f_{ab},\ab i_{ab})$, $\Om_a:(S_a',V_a',g_a',j_a')\hookra
(S_{ab},V_{ab},g_{ab},j_{ab})$, and~$\Om_b:(S_b',V_b',g_b',j_b')
\hookra(S_{ab},V_{ab},g_{ab},j_{ab})$.

By combining the proofs of Proposition \ref{sa3prop1} and Theorem
\ref{sa3thm2} in \cite{Joyc2}, we can show that we can choose this
data such that $\Phi_a(U_a')\subseteq V_a'$, $\Phi_b(U_b')\subseteq
V_b'$, and with a morphism $\Phi_{ab}:U_{ab}\ra V_{ab}$ smooth of
relative dimension $d$ such that
\begin{align*}
f_{ab}&=g_{ab}\ci\Phi_{ab}, & \Phi_{ab}\ci i_{ab}&=j_{ab}\ci\phi_{ab}, \\
\Phi_{ab}\ci\Psi_a&=\Om_a\ci\Phi_a\vert_{U_a'}, &
\Phi_{ab}\ci\Psi_b&=\Om_b\ci\Phi_a\vert_{U_a'}.
\end{align*}
As for \eq{sa4eq8} we have a commutative diagram
\e
\begin{gathered} \xymatrix@!0@C=285pt@R=38pt{
*+[r]{\phi\vert_{R_{ab}}^*[d](P_{Y,t}^\bu)} \ar[d]^{\De_{ab}}
\ar[r]_(0.28){\raisebox{-11pt}{$\st\phi\vert_{R_{ab}}^*[d]
(\om_{S_{ab},V_{ab},g_{ab},j_{ab}})$}} &
*+[l]{\phi\vert_{R_{ab}}^*[d]\bigl(j_{ab}^*\bigl(\PV_{V_{ab},
g_{ab}}^\bu\bigr)\!\ot_{\Z/2\Z}\!
Q_{S_{ab},V_{ab},g_{ab},j_{ab}}\bigr)}
\ar[d]_{i_{ab}^*(\Xi_{\Phi_{ab}})\ot \al_{\Phi_{ab}}} \\
*+[r]{P_{X,s}^\bu\vert_{R_{ab}}}
\ar[r]^(0.28){\om_{R_{ab},U_{ab},f_{ab},i_{ab}}} &
*+[l]{i_{ab}^*\bigl(\PV_{U_{ab},f_{ab}}^\bu\bigr)\ot_{\Z/2\Z}
Q_{R_{ab},U_{ab},f_{ab},i_{ab}}.} }
\end{gathered}
\label{sa4eq9}
\e

Using \cite[Th.~6.9(ii)]{BBDJS} for the embeddings $\Psi_a,\Om_a$
gives commutative diagrams
\ea
\begin{gathered}
\xymatrix@C=140pt@R=18pt{ *+[r]{P_{X,s}^\bu\vert_{R_a'}}
\ar[r]_(0.34){\om_{R_a',U_a',f_a',i_a'}}
\ar[d]^(0.4){\om_{R_{ab},U_{ab},f_{ab},i_{ab}}\vert_{R_a'}} &
*+[l]{i_a^{\prime *}\bigl(\PV_{U_a',f_a'}^\bu\bigr)\ot_{\Z/2\Z}
Q_{R_a',U_a',f_a',i_a'}}
\ar[d]_(0.4){i_a^{\prime *}(\Th_{\Psi_a})\ot\id} \\
*+[r]{\begin{subarray}{l}\ts i_{ab}^*\bigl(\PV_{U_{ab},f_{ab}}^\bu\bigr)
\vert_{R_a'}\\ \ts \ot_{\Z/2\Z}Q_{R_{ab},U_{ab},f_{ab},i_{ab}}
\vert_{R_a'}\end{subarray}} \ar[r]^(0.38){\id\ot\La_{\Psi_a}} &
*+[l]{\begin{subarray}{l}\ts
i_a^{\prime *}\bigl(\Psi_a^*(\PV_{U_{ab},f_{ab}}^\bu)
\ot_{\Z/2\Z}P_{\Psi_a}\bigr)\\
\ts\ot_{\Z/2\Z} Q_{R_a',U_a',f_a',i_a'},\end{subarray}} }
\end{gathered}
\label{sa4eq10}\\
\begin{gathered}
\xymatrix@C=140pt@R=18pt{ *+[r]{P_{Y,t}^\bu\vert_{S_a'}}
\ar[r]_(0.34){\om_{S_a',V_a',g_a',i_a'}}
\ar[d]^(0.4){\om_{S_{ab},V_{ab},g_{ab},j_{ab}}\vert_{S_a'}} &
*+[l]{j_a^{\prime *}\bigl(\PV_{V_a',g_a'}^\bu\bigr)\ot_{\Z/2\Z}
Q_{S_a',V_a',g_a',j_a'}}
\ar[d]_(0.4){j_a^*(\Th_{\Om_a})\ot\id} \\
*+[r]{\begin{subarray}{l}\ts j_{ab}^*\bigl(\PV_{V_{ab},g_{ab}}^\bu\bigr)
\vert_{S_a'}\\ \ts \ot_{\Z/2\Z}Q_{S_{ab},V_{ab},g_{ab},j_{ab}}
\vert_{S_a'}\end{subarray}} \ar[r]^(0.38){\id\ot\La_{\Om_a}} &
*+[l]{\begin{subarray}{l}\ts
j_a^{\prime *}\bigl(\Om_a^*(\PV_{V_{ab},g_{ab}}^\bu)
\ot_{\Z/2\Z}P_{\Om_a}\bigr)\\
\ts\ot_{\Z/2\Z} Q_{S_a',V_a',g_a',j_a'}.\end{subarray}} }
\end{gathered}
\label{sa4eq11}
\ea
Here $P_{\Psi_a},P_{\Om_a}$ are principal $\Z/2\Z$-bundles on
$R_a',S_a'$ from \cite[Def.~5.2]{BBDJS}, and
$\Th_{\Psi_a},\Th_{\Om_a}$ are isomorphisms of perverse sheaves from
\cite[Th.~5.4(a)]{BBDJS}, and $\La_{\Psi_a},\ab\La_{\Om_a}$ are
isomorphisms of principal $\Z/2\Z$-bundles
from~\cite[Th.~6.9(ii)]{BBDJS}.

From the definitions of $P_{\Psi_a},P_{\Om_a},\Th_{\Psi_a},
\Th_{\Om_a},\La_{\Psi_a},\La_{\Om_a}$ one can show that there is a
natural isomorphism $\be_a:\Phi_a^*[d](P_{\Om_a})\ra P_{\Psi_a}$
such that the following commute:
\ea
\begin{gathered} \xymatrix@!0@C=280pt@R=38pt{
*+[r]{\Phi_a^*[d](\PV_{V_a',g_a'}^\bu)}
 \ar[d]^{\Xi_{\Phi_a}}
\ar[r]_(0.32){\Phi_a^*[d](\Th_{\Om_a})} &
*+[l]{\begin{subarray}{l}\ts \Phi_a^*[d]\bigl(\Om_a^*
(\PV_{V_{ab},g_{ab}}^\bu)\ot_{\Z/2\Z}P_{\Om_a}\bigr)=\\
\ts\Psi_a^*\!\ci\!\Phi_{ab}^*[d](\PV_{V_{ab},g_{ab}}^\bu)
\!\ot_{\Z/2\Z}\!\Phi_a^*[d](P_{\Om_a})\end{subarray}}
\ar[d]_(0.6){\Psi_a^*(\Xi_{\Phi_{ab}})\ot \be_a} \\
*+[r]{\PV_{U_a',f_a'}^\bu} \ar[r]^(0.28){\Th_{\Psi_a}} &
*+[l]{\Psi_a^*(\PV_{U_{ab},f_{ab}}^\bu)
\ot_{\Z/2\Z}P_{\Psi_a},} }
\end{gathered}
\label{sa4eq12}\\
\begin{gathered}
\xymatrix@!0@C=285pt@R=38pt{
*+[r]{\phi\vert_{R_{ab}}^*[d](Q_{S_{ab},V_{ab},g_{ab},j_{ab}})}
\ar[r]_(0.38){\raisebox{-11pt}{$\st\phi\vert_{R_{ab}}^*[d](\La_{\Om_a})$}}
\ar[d]^(0.4){\al_{\Phi_{ab}}} &
*+[l]{\phi\vert_{R_{ab}}^*[d](j_a^{\prime *}(P_{\Om_a})\ot_{\Z/2\Z}
Q_{S_a',V_a',g_a',j_a'})}
\ar[d]_(0.4){i_a^{\prime *}(\be_a)\ot\al_{\Phi_a}} \\
*+[r]{Q_{R_{ab},U_{ab},f_{ab},i_{ab}}
\vert_{R_a'}} \ar[r]^(0.38){\La_{\Psi_a}} &
*+[l]{i_a^{\prime *}(P_{\Psi_a})\ot_{\Z/2\Z} Q_{R_a',U_a',f_a',i_a'}.} }
\end{gathered}
\label{sa4eq13}
\ea
Combining \eq{sa4eq8}--\eq{sa4eq13} we see that
$\De_a\vert_{R_a'}=\De_{ab}\vert_{R_a'}$. Similarly
$\De_b\vert_{R_b'}=\De_{ab}\vert_{R_b'}$, so $\De_a\vert_{R_a'\cap
R_b'}=\De_b\vert_{R_a'\cap R_b'}$, where $R_a'\cap R_b'$ is an open
neighbourhood of $x$ in $R_a\cap R_b$. As we can cover $R_a\cap R_b$
by such open $R_a'\cap R_b'$, and (iso)morphisms of perverse sheaves
form a sheaf, it follows that~$\De_a\vert_{R_a\cap R_b}=
\De_b\vert_{R_a\cap R_b}$.

By the Zariski topology version of Theorem \ref{sa4thm1}(i), there
exists a unique isomorphism $\De_\phi$ in \eq{sa4eq5} such that
$\De_\phi\vert_{R_a}=\De_a$ for all $a\in A$. As each $\De_a$
identifies $\phi^*[d](\Si_{Y,t})\vert_{R_a},\phi^*[d](\Tau_{Y,t})
\vert_{R_a}$ with $\Si_{X,s}\vert_{R_a},\Tau_{X,s}\vert_{R_a}$ from
above, $\De_\phi$ identifies $\phi^*[d](\Si_{Y,t}),
\phi^*[d](\Tau_{Y,t})$ with $\Si_{X,s},\Tau_{X,s}$. By our usual
argument involving taking disjoint union of two open covers, we see
that $\De_\phi$ is independent of the choice of data $A$ and
$(R_a,U_a,f_a,i_a),(S_a,V_a,g_a,j_a),\Phi_a$ for $a\in A$. Let
$(R,U,f,i),(S,V,g,j),\Phi$ be as in (a). By defining $\De_\phi$
using data $A,(R_a,U_a,f_a,i_a), (S_a,V_a,g_a,j_a),\Phi_a$ with
$(R,U,f,i),(S,V,g,j),\Phi$ equal to $(R_a,\ab U_a,\ab f_a,\ab i_a),
(S_a,V_a,g_a,j_a),\Phi_a$ for some $a\in A$, we see that part (a)
holds.

For (b), let $x\in X$ with $y=\phi(x)\in Y$ and $z=\psi(z)\in Z$.
The proof of Proposition \ref{sa3prop1} in \cite{Joyc2} shows we may
choose critical charts $(R,U,f,i),(S,V,g,j),\ab(T,W,h,k)$ on
$(X,s),(Y,t),(Z,u)$ with $x\in R$, $y\in\phi(R)\subseteq S$,
$z\in\psi(S)\subseteq T$ of minimal dimensions $\dim U=\dim T_xX$,
$\dim V=\dim T_yY$, $\dim W=\dim T_zZ$, and $\Phi:U\ra V$,
$\Psi:V\ra W$ smooth of relative dimensions $d,e$ with $f=g\ci\Phi$,
$g=h\ci\Psi$ and $\Phi\ci i=j\ci\phi$, $\Psi\ci j=k\ci\psi$.
Consider the diagram of isomorphisms:
\begin{equation*}
\xymatrix@!0@C=251pt@R=40pt{
*+[r]{\begin{subarray}{l}\ts (\psi\ci\phi)\vert_R^*[d+e]\\
\ts (P_{Z,u}^\bu)\end{subarray}}
\ar[d]^(0.55){\phi\vert_R^*(\De_\psi)}
\ar@/_.5pc/@<-1ex>[dd]_{\De_{\psi\ci\phi}\vert_R}
\ar[r]_(0.35){\raisebox{-0pt}{$\st(\psi\ci\phi)
\vert_R^*[d+e](\om_{T,W,h,k})$}} &
*+[l]{\begin{subarray}{l}\ts (\psi\ci\phi)\vert_R^*[d+e]\\
\ts(k^*(\PV_{W,h}^\bu)\ot_{\Z/2\Z}Q_{T,W,h,k})
\end{subarray}}
\ar[d]_(0.55){\phi\vert_R^*(j^*(\Xi_\Psi)\ot \al_\Psi)}
\ar@/^.5pc/@<1ex>[dd]^{\begin{subarray}{l}\st
i^*(\Xi_{\Psi\ci\Phi})\\
\st \ot \al_{\Psi\ci\Phi}\end{subarray}}
\\
*+[r]{\phi\vert_R^*[d](P_{Y,t}^\bu)} \ar[d]^{\De_\phi\vert_R}
\ar[r]_(0.32){\phi\vert_R^*[d](\om_{S,V,g,j})} &
*+[l]{\phi\vert_R^*[d](
j^*(\PV_{V,g}^\bu)\!\ot_{\Z/2\Z}\!Q_{S,V,g,j})}
\ar[d]_{i^*(\Xi_\Phi)\ot \al_\Phi}
\\
*+[r]{P_{X,s}^\bu\vert_R}
\ar[r]^(0.35){\om_{R,U,f,i}} &
*+[l]{i^*(\PV_{U,f}^\bu)\ot_{\Z/2\Z} Q_{R,U,f,i}.} }
\end{equation*}
The two inner and the outer rectangles commute by \eq{sa4eq6}. Also
$\al_{\Psi\ci\Phi}=\al_\Phi\ci\phi\vert_R^*(\al_\Psi)$ is immediate
and $\Xi_{\Psi\ci\Phi}=\Xi_\Phi\ci\Phi\vert_{\Crit(f)}^*[d]
(\Xi_\Psi)$ follows from the definition of $\Xi_\Phi$ in Proposition
\ref{sa4prop2}(b), so the right hand semicircle commutes. Therefore
the left hand semicircle commutes. This proves the restriction of
\eq{sa4eq7} to $R\subseteq X$. As we can cover $X$ by such open $R$,
equation \eq{sa4eq7} follows.

For part (c), all the facts we have used about perverse sheaves on
$\C$-schemes above also hold in the other settings of $l$-adic
perverse sheaves on $\K$-schemes, $\cD$-modules, and mixed Hodge
modules. This completes the proof.
\end{proof}

\subsection{Perverse sheaves on Artin stacks}
\label{sa43}

We first note that because of Proposition \ref{sa4prop1} and Theorem
\ref{sa4thm1}, any of the theories of perverse sheaves on
$\C$-schemes or $\K$-schemes mentioned in \S\ref{sa41} can be
extended to Artin $\C$-stacks or Artin $\K$-stacks $X$ in a na\"\i ve
way, using the philosophy discussed in \S\ref{sa32} and \cite[\S
2.7]{Joyc2} of defining sheaves on $X$ in terms of sheaves on
schemes $T$ for smooth $t:T\ra X$, in particular
Proposition~\ref{sa3prop3}:

\begin{dfn} Fix one of the theories of perverse sheaves on
$\K$-schemes discussed in \S\ref{sa41}, over an allowed base ring
$A$, where we include the special case $\K=\C$ and $A$ is general as
in Dimca \cite{Dimc}. Let $X$ be an Artin $\K$-stack, always assumed
locally of finite type. We will explain how to define an abelian
category $\Perv_\nai(X)$ of {\it na\"\i ve perverse sheaves\/}
on~$X$:
\smallskip

\noindent(A) Define an object $\cP$ of $\Perv_\nai(X)$ to assign
\begin{itemize}
\setlength{\itemsep}{0pt}
\setlength{\parsep}{0pt}
\item[(a)] For each $\K$-scheme $T$ and smooth 1-morphism
$t:T\ra X$, a perverse sheaf $\cP(T,t)\in\Perv(T)$ on $T$ in our
chosen $\K$-scheme perverse sheaf theory.
\item[(b)] For each 2-commutative diagram in $\Art_\K$:
\e
\begin{gathered}
\xymatrix@C=50pt@R=1pt{ & U \ar[ddr]^u \\
\rrtwocell_{}\omit^{}\omit{^{\eta}} && \\
T  \ar[uur]^{\phi} \ar[rr]_t && X, }
\end{gathered}
\label{sa4eq14}
\e
where $T,U$ are $\K$-schemes and $\phi,t,u$ are smooth with
$\phi$ of dimension $d$, an isomorphism
$\cP(\phi,\eta):\phi^*[d](\cP(U,u))\ra\cP(T,t)$ in~$\Perv(T)$.
\end{itemize}
This data must satisfy the following condition:
\begin{itemize}
\setlength{\itemsep}{0pt}
\setlength{\parsep}{0pt}
\item[(i)] For each 2-commutative diagram in $\Art_\K\!:$
\begin{equation*}
\xymatrix@C=70pt@R=1pt{ & V \ar[ddr]^v \\
\rrtwocell_{}\omit^{}\omit{^{\ze}} && \\
U  \ar[uur]^{\psi} \ar[rr]_(0.3)u && X, \\
\urrtwocell_{}\omit^{}\omit{^{\eta}} && \\
T \ar[uu]_{\phi} \ar@/_/[uurr]_t }
\end{equation*}
with $T,U,V$ $\K$-schemes and $\phi,\psi,t,u,v$ smooth with
$\phi,\psi$ of dimensions $d,e$, we must have
\begin{align*}
\cP\bigl(\psi\ci\phi,(\ze*\id_{\phi})\od\eta\bigr)
&=\cP(\phi,\eta)\ci\phi^*[d](\cP(\psi,\ze))\quad\text{as
morphisms}\\
(\psi\ci\phi)^*[d+e](\cP(V,v))&=\phi*[d]\ci
\psi^*[e](\cP(V,v))\longra\cP(T,t).
\end{align*}
\end{itemize}

\noindent(B) Morphisms $\al:\cP\ra\cQ$ of $\Perv_\nai(X)$ comprise a
morphism $\al(T,t):\cP(T,t)\ra\cQ(T,t)$ in $\Perv(T)$ for all smooth
1-morphisms $t:T\ra X$ from a scheme $T$, such that for each diagram
\eq{sa4eq14} in (b) the following commutes:
\begin{equation*}
\xymatrix@C=120pt@R=15pt{*+[r]{\phi^*[d](\cP(U,u))}
\ar[d]^{\phi^*[d](\al(U,u))} \ar[r]_(0.55){\cP(\phi,\eta)} &
*+[l]{\cP(T,t)} \ar[d]_{\al(T,t)} \\
*+[r]{\phi^*[d](\cQ(U,u))}
\ar[r]^(0.55){\cQ(\phi,\eta)} & *+[l]{\cQ(T,t).\!{}} }
\end{equation*}

\noindent(C) Composition of morphisms $\cP\,{\buildrel\al
\over\longra}\,\cQ\,{\buildrel\be\over\longra}\,\cR$ in
$\Perv_\nai(X)$ is $(\be\ci\al)(T,t)=\ab\be(T,t)\ab\ci\ab\al(T,t)$.
Identity morphisms $\id_\cP:\cP\ra\cP$ are
$\id_\cP(T,t)=\id_{\cP(T,t)}$.
\smallskip

We can also define a category of {\it na\"\i ve\/ $\cD$-modules\/}
on $X$ in the same way.
\label{sa4def1}
\end{dfn}

\begin{rem} Definition \ref{sa4def1} for $\cP$ is modelled on
Proposition \ref{sa3prop3} for $\cA$, with the following
differences:
\begin{itemize}
\setlength{\itemsep}{0pt}
\setlength{\parsep}{0pt}
\item[(i)] $\cP(\phi,\eta)$ is an isomorphism always, but
$\cA(\phi,\eta)$ need only be an isomorphism if $\phi$ is
\'etale. Now $\cA$ in Proposition \ref{sa3prop3}(A) is called a
{\it Cartesian\/} sheaf on $X$ if $\cA(\phi,\eta)$ is an
isomorphism always. So $\cP$ is the perverse analogue of a
Cartesian sheaf $\cA$ on $X$.
\item[(ii)] $\cP(\phi,\eta)$ is defined only when $\phi$ is smooth,
but $\cA(\phi,\eta)$ is defined without requiring $\phi$ smooth.
For Cartesian sheaves $\cA$ on $X$, it is enough to give the
data $\cA(T,t),\cA(\phi,\eta)$ and check the conditions for
$\phi$ smooth; the remaining $\cA(\phi,\eta)$ for non-smooth
$\phi$ are then determined uniquely.
\item[(iii)] Definition \ref{sa4def1} uses shifted pullbacks
$\phi^*[d]$ where Proposition \ref{sa3prop3} uses sheaf
pullbacks $\phi^{-1}$. This is because of
Proposition~\ref{sa4prop1}.
\end{itemize}
\label{sa4rem1}
\end{rem}

Using Proposition \ref{sa4prop1}, Theorem \ref{sa4thm1} and formal
arguments, we can deduce:
\begin{itemize}
\setlength{\itemsep}{0pt}
\setlength{\parsep}{0pt}
\item[(a)] For any Artin stack $X$, $\Perv_\nai(X)$ is an abelian
category, and if $X$ is a scheme, the functor
$\Perv_\nai(X)\ra\Perv(X)$ mapping $\cP\mapsto\cP(X,\id_X)$ is
an equivalence of categories with the category $\Perv(X)$
discussed in \S\ref{sa41}.
\item[(b)] If $\Phi:X\ra Y$ is a 1-morphism of Artin stacks smooth of
relative dimension $d$ then as in Proposition \ref{sa4prop1}
there is a natural functor
$\Phi_\nai^*[d]:\Perv_\nai(Y)\ra\Perv_\nai(X)$.
\item[(c)] The analogue of Theorem \ref{sa4thm1} holds for the
categories $\Perv_\nai$ and pullbacks $\Phi_\nai^*[d]$, taking
the $U_i,U_{ij},U_{ijk}$ to be either schemes or stacks.
\end{itemize}
This `na\"\i ve' model of perverse sheaves on Artin stacks follows
from the scheme case in an essentially trivial way, and is
sufficient to prove the first part of the main result of this
section, Theorem \ref{sa4thm3} below.

{\it However}, for a satisfactory theory of perverse sheaves on
Artin stacks, we want more: we would like the category $\Perv(X)$ of
perverse sheaves on $X$ to be the heart of a t-structure on a
triangulated category $D^b_c(X)$ of `constructible complexes', which
may not be equivalent to $D^b\Perv(X)$, and we would like
Grothendieck's ``six operations on sheaves''
$f^*,f^!,Rf_*,Rf_!,{\cal RH}om,\otL$, and Verdier duality operators
$\bD_X$, to act on these ambient categories $D^b_c(X)$. Other than
pullbacks $f^*,f^!$ by smooth 1-morphisms $f:X\ra Y$ and operators
$\bD_X$, none of this is obvious using the definition of perverse
sheaves $\Perv_\nai(X)$ above.

Thus, the main issue in developing a good theories of perverse
sheaves on Artin stacks $X$ is not defining the categories
$\Perv(X)$ or $\Perv_\nai(X)$ themselves, but defining the
categories $D^b_c(X)$ and the six operations $f^*,\ldots,\otL$ upon them, and then defining a perverse t-structure on $D^b_c(X)$ with heart $\Perv(X)$. If (a)--(c) above hold for these $D^b_c(X),\Perv(X)$, it will then be automatic \cite[\S 7]{LaOl3} that $\Perv(X)\simeq\Perv_\nai(X)$ for $\Perv_\nai(X)$ as in Definition~\ref{sa4def1}.

Here are the foundational papers on perverse sheaves and
$\cD$-modules on Artin stacks known to the authors:
\begin{itemize}
\setlength{\itemsep}{0pt}
\setlength{\parsep}{0pt}
\item Laszlo and Olsson \cite{LaOl1,LaOl2,LaOl3} generalize the
Beilinson--Bernstein--Deligne theory of perverse sheaves on
$\K$-schemes with finite and $l$-adic coefficients \cite{BBD} to
Artin stacks. In \cite[\S 7]{LaOl3} they show that $\Perv(X)$ is equivalent to the category $\Perv_\nai(X)$ in Definition \ref{sa4def1}.
\item Liu and Zheng \cite{LiZh1,LiZh2} develop a theory of perverse
sheaves on higher Artin stacks using Lurie's $\iy$-categories,
and show it is equivalent to Laszlo and Olsson's version for
ordinary Artin stacks.
\item Gaitsgory and Rozenblyum \cite{GaRo} construct a theory of
{\it crystals\/} on (derived) schemes and stacks $\bX$. For
classical schemes $X$, the categories of crystals and
$\cD$-modules on $X$ are equivalent, so the authors argue that
$\cD$-modules on (derived) stacks should be defined to be crystals. The six functor formalism for crystals was not complete at the time of writing.
\item In a brief note, for an Artin $\C$-stack $X$, Paulin
\cite{Paul} proposes definitions of constructible complexes
$D^b_c(X)$ over $A=\C$, with its perverse t-structure, and (for
smooth $X$) of the derived category $D^b_{rh}(X)$ of
$\cD$-modules on $X$ with t-structure, claims the six
functor formalism holds, and proves a `Riemann--Hilbert' equivalence of these categories with
t-structures. 
\end{itemize}

\subsection{The main result}
\label{sa44}

Here is the main result of this section, the analogue of Theorem
\ref{sa4thm2} from \cite{BBDJS}. Apart from the material in our previous papers \cite{BBDJS,Joyc2} and general properties of perverse sheaves on Artin stacks, the only extra ingredient is Proposition~\ref{sa4prop3}.

We state Theorem \ref{sa4thm3} and Corollaries \ref{sa4cor1}, \ref{sa4cor2} using Laszlo and Olsson's $l$-adic perverse sheaves on Artin stacks \cite{LaOl1,LaOl2,LaOl3}, but they would also work for any other theory of perverse sheaves, or $\cD$-modules, or mixed Hodge modules, on Artin stacks, which has the expected package of properties discussed in~\S\ref{sa43}.

\begin{thm} Let\/ $(X,s)$ be an oriented d-critical stack over $\K$
(allowing $\K=\C$) with orientation $K_{X,s}^{1/2}$. Fix a theory of
perverse sheaves on $\K$-schemes from\/ {\rm\S\ref{sa41},} and let\/
$\Perv_\nai(X)$ be the corresponding category of na\"\i ve perverse
sheaves on $X$ from Definition\/ {\rm\ref{sa4def1}}. Then we may
define $\cP_{X,s}\in \Perv_\nai(X)$ and Verdier duality and
monodromy isomorphisms
\begin{equation*}
\smash{\Si_{X,s}:\cP_{X,s}\longra \bD_X(\cP_{X,s}),\qquad
\Tau_{X,s}:\cP_{X,s}\longra \cP_{X,s},}
\end{equation*}
as follows:
\begin{itemize}
\setlength{\itemsep}{0pt}
\setlength{\parsep}{0pt}
\item[{\bf(a)}] If\/ $t:T\ra X$ is smooth with\/ $T$ a
$\K$-scheme, so that\/ $(T,s(T,t))$ is an algebraic d-critical
locus with natural orientation $K_{T,s(T,t)}^{1/2}$ as in
Lemma\/ {\rm\ref{sa3lem1},} then
$\cP_{X,s}(T,t)=P_{T,s(T,t)}^\bu$ in $\Perv(T),$ where
$P_{T,s(T,t)}^\bu$ is the perverse sheaf on the oriented
algebraic d-critical locus $(T,s(T,t))$ over $\K$ given by
Theorem\/ {\rm\ref{sa4thm2}}. Also $\Si_{X,s}(T,t)=
\Si_{T,s(T,t)}$ and\/ $\Tau_{X,s}(T,t)=\Tau_{T,s(T,t)}$.
\item[{\bf(b)}] For each\/ $2$-commutative diagram in $\Art_\K$
\begin{equation*}
\xymatrix@C=50pt@R=1pt{ & U \ar[ddr]^u \\
\rrtwocell_{}\omit^{}\omit{^{\eta}} && \\
T \ar[uur]^{\phi} \ar[rr]_t && X }
\end{equation*}
with\/ $T,U$ $\K$-schemes and\/ $\phi,t,u$ smooth with\/ $\phi$
of dimension $d,$ we have
\begin{align*}
\cP_{X,s}(\phi,\eta)=\De_\phi:\phi^*[d]&(\cP_{X,s}(U,u))=
\phi^*[d](P_{U,s(U,u)}^\bu)\\
&\longra \cP_{X,s}(T,t)=P_{T,s(T,t)}^\bu,
\end{align*}
where $\De_\phi$ is as in Proposition\/ {\rm\ref{sa4prop3}}.
\end{itemize}

If we work with perverse sheaves on $\K$-schemes in the sense of\/
{\rm\cite{BBD}} over a base ring $A$ with either $\mathop{\rm
char}A>0$ coprime to $\mathop{\rm char}\K,$ or $A=\Z_l,\Q_l$ or
$\bar\Q_l$ with\/ $l$ coprime to $\mathop{\rm char}\K,$ then
$\Perv_\nai(X)\simeq\Perv(X)$ as in {\rm\S\ref{sa43},} where
$\Perv(X)\subset D^b_c(X)$ is the category of perverse sheaves on
$X$ over $A$ defined by Laszlo and Olsson
{\rm\cite{LaOl1,LaOl2,LaOl3}}. Thus $\cP_{X,s}$ corresponds to
$\check P_{X,s}^\bu\in\Perv(X)$ unique up to canonical isomorphism,
and\/ $\Si_{X,s},\Tau_{X,s}$ correspond to isomorphisms
\begin{equation*}
\smash{\check\Si_{X,s}:\check P_{X,s}^\bu\longra \bD_X(\check
P_{X,s}^\bu),\quad \check\Tau_{X,s}:\check P_{X,s}^\bu\longra \check
P_{X,s}^\bu\quad\text{in\/ $\Perv(X)$.}}
\end{equation*}

The analogue of the above will also hold in any other theory of
perverse sheaves or\/ $\cD$-modules on schemes and Artin stacks with
the package of properties discussed in\/
{\rm\S\ref{sa43},} including the six operations
$f^*,f^!,Rf_*,Rf_!,{\cal RH}om,\ab\otL,$ Verdier duality $\bD_X,$
and descent in the smooth topology as in Theorem\/
{\rm\ref{sa4thm1}}.
\label{sa4thm3}
\end{thm}

\begin{proof} Proposition \ref{sa4prop3}(b) implies that the data
$\cP_{X,s}(T,t),\cP_{X,s}(\phi,\eta)$ in (a),(b) satisfy Definition
\ref{sa4def1}(A)(i). Thus $\cP_{X,s}$ is an object of
$\Perv_\nai(X)$. Similarly, the last part of Proposition
\ref{sa4prop3}(a) implies that $\Si_{X,s},\Tau_{X,s}$ are morphisms
in $\Perv_\nai(X)$. The last part is immediate from the discussion
of~\S\ref{sa43}.
\end{proof}

Combining Theorems \ref{sa2thm6}, \ref{sa3thm6} and \ref{sa4thm3} and
Corollary \ref{sa3cor2} yields:

\begin{cor} Let\/ $\K$ be an algebraically closed field of
characteristic zero, $(\bX,\om)$ a $-1$-shifted symplectic derived
Artin $\K$-stack, and\/ $X=t_0(\bX)$ the associated classical
Artin\/ $\K$-stack. Suppose we are given a square root\/
$\smash{\det(\bL_\bX)\vert_X^{1/2}}$.

Then working in $l$-adic perverse sheaves on stacks
{\rm\cite{LaOl1,LaOl2,LaOl3},} we may define a perverse sheaf\/
$\check P_{\bX,\om}^\bu$ on $X$ uniquely up to canonical
isomorphism, and Verdier duality and monodromy isomorphisms
$\check\Si_{\bX,\om}:\check P_{\bX,\om}^\bu\ra \bD_X(\check
P_{\bX,\om}^\bu)$ and\/ $\check\Tau_{\bX,\om}:\check
P_{\bX,\om}^\bu\ra\check P_{\bX,\om}^\bu$. These are characterized
by the fact that given a diagram
\begin{equation*}
\smash{\xymatrix@C=60pt{ \bU=\bs\Crit(f:U\ra\bA^1) & \bV \ar[l]_(0.3){\bs i} \ar[r]^{\bs\vp} & \bX }}
\end{equation*}
such that\/ $U$ is a smooth\/ $\K$-scheme, $\bs\vp$ smooth of
dimension $n,$ $\bL_{\bV/\bU} \simeq \bT_{\bV/\bX}[2],$
$\bs\vp^*(\om_\bX)\sim \bs i^*(\om_\bU)$ for $\om_\bU$ the natural\/
$-1$-shifted symplectic structure on $\bU=\bs\Crit(f:U\ra\bA^1),$
and\/ $\vp^*(\det(\bL_\bX)\vert_X^{1/2})\!\cong\!
i^*(K_U)\!\ot\!\La^n\bT_{\bV/\bX},$ then $\vp^*(\check
P_{\bX,\om}^\bu)[n],$ $\vp^*(\check\Si_{\bX,\om}^\bu)[n],$
$\vp^*(\check\Tau_{\bX,\om}^\bu)[n]$ are canonically isomorphic to
$i^*(\PV_{U,f}),$ $i^*(\si_{U,f}),$ $i^*(\tau_{U,f}),$ for\/
$\PV_{U,f},\si_{U,f},\tau_{U,f}$ as in {\rm\S\ref{sa41}}. 
\label{sa4cor1}
\end{cor}

\begin{cor} Let\/ $Y$ be a Calabi--Yau\/ $3$-fold over an
algebraically closed field\/ $\K$ of characteristic zero, and\/
$\cM$ a classical moduli\/ $\K$-stack of coherent sheaves $F$ in
$\coh(Y),$ or of complexes $F^\bu$ in $D^b\coh(Y)$ with\/
$\Ext^{<0}(F^\bu,F^\bu)=0,$ with obstruction theory\/
$\phi:\cE^\bu\ra\bL_\cM$. Suppose we are given a square root\/
$\det(\cE^\bu)^{1/2}$.

Then working in $l$-adic perverse sheaves on stacks
{\rm\cite{LaOl1,LaOl2,LaOl3},} we may define a natural perverse
sheaf\/ $\check P_\cM^\bu\in\Perv(\cM),$ and Verdier duality and
monodromy isomorphisms $\check\Si_\cM:\check
P_\cM^\bu\ra\bD_\cM(\check P_\cM^\bu)$ and\/ $\check\Tau_\cM:\check
P_\cM^\bu\ra\check P_\cM^\bu$. The pointwise Euler characteristic
of\/ $\check P_\cM^\bu$ is the Behrend function $\nu_\cM$ of\/ $\cM$
from Joyce and Song {\rm\cite[\S 4]{JoSo},} so that\/ $\check
P_\cM^\bu$ is in effect a categorification of the Donaldson--Thomas
theory of $\cM$.
\label{sa4cor2}
\end{cor}

\begin{ex} Suppose an algebraic $\K$-group $G$ acts on a
$\K$-scheme $T$ with action $\mu:G\t T\ra T$, and write $X$ for the
quotient Artin $\K$-stack $[T/G]$, and $t:T\ra [T/G]$ for the
natural quotient 1-morphism.

As in Example \ref{sa3ex}, there is a 1-1 correspondence between
d-critical structures $s$ on $X=[T/G]$ and $G$-invariant d-critical
structures $s'$ on $T$, such that $s'=s(T,t)$. Also, from Lemma
\ref{sa3lem1} we see that there is a 1-1 correspondence between
orientations $K_{X,s}^{1/2}$ for $(X,s)$, and $G$-invariant
orientations $K_{T,s'}^{1/2}$ for $(T,s')$, given by
$K_{T,s'}^{1/2}=K_{X,s}^{1/2}(T^\red,t^\red)\ot (\La^{\rm
top}\bL_{\smash{T/X}})\vert_{T^\red}$.

Choose such $s,s',K_{X,s}^{1/2},K_{T,s'}^{1/2}$, so that Theorems
\ref{sa4thm2} and \ref{sa4thm3} give perverse sheaves $P^\bu_{T,s'},
\check P_{X,s}^\bu$ on $T,X$. We would like to relate the
hypercohomologies $\bH^*(T,P^\bu_{T,s'}),\bH^*(X,\check P_{X,s}^\bu)$.
We have $t^{*}(\check P_{X,s}^\bu)[\dim G]\cong P_{T,s'}^\bu$ and thus
\begin{equation*}
\smash{R^{q}t_{*}P_{T,s'}^\bu \cong R^{q}t_{*}t^{*}(\check P_{X,s}^\bu)[\dim G] \cong \check P_{X,s}^\bu \ot_{A_X} R^{q}t_{*}(A_T)[\dim G],}
\end{equation*}
where $A_T$ is the constant sheaf on $T$ with fibre the base ring $A$. Therefore, the Leray--Serre spectral sequence for the
fibration $t:T\ra X$ with fibre $G$, twisted by $\check
P_{X,s}^\bu$, can be interpreted as a spectral sequence
\begin{equation*}
\smash{E^{\bu,\bu} \Longra \bH^{\bu}(T,P_{T,s'}^\bu )\quad\text{with}\quad E^{p,q}_{2}=\bH^{p}\bigl(X, \check P_{X,s}^\bu \ot_{A_X} R^{q}t_{*}(A_T)[\dim G]\bigr),}
\end{equation*}
where $R^{q}t_{*}(A_T)[\dim G]$ is locally constant on $X$ with fibre~$H^{q-\dim G}(G,A)$.

We also have a projection $\pi:X=[T/G]\ra[*/G]$ for $*=\Spec\K$ with
fibre $T$. The Leray--Serre spectral sequence for $\pi$ gives a
spectral sequence
\begin{equation*}
\smash{E^{\bu,\bu} \Longra \bH^{\bu}(X,\check P_{X,s}^\bu )\quad\text{with}\quad
E^{p,q}_{2}=\bH^{p}\bigl([*/G],\bH^{q+\dim G}(T, P_{T,s'}^\bu)\bigr).}
\end{equation*}
If $G$ is finite we can consider the $\bH^*(T,P_{T,s'}^\bu)$ as
$G$-modules and $\bH^*([*/G],-)$ as group cohomology $H^*_{\rm
grp}\bigl(G,-)$, giving a spectral sequence
\begin{equation*}
\smash{H^p_{\rm grp}\bigl(G,\bH^q(T, P_{T,s'}^\bu)\bigr) \Longra
\bH^{p+q}(X,\check P_{X,s}^\bu ).}
\end{equation*}

\label{sa4ex1}
\end{ex}

\begin{ex} Suppose that $(\bX,\om_{\bX})$ is an oriented $-1$-shifted symplectic derived Artin $\K$-stack, and a finite group $G$ acts on $\bX$ preserving $\om_{\bX}$ and the orientation. Let $\bY$ be the derived Artin $\K$-stack $[\bX/G]$ equipped with the natural quotient $-1$-shifted symplectic structure $\om_{\bY}$ and orientation, and write $\bs f:\bX\ra \bY$ for the \'etale quotient morphism of derived Artin $\K$-stacks. Then we have $\bs f^*(\om_{\bY})\sim\om_{\bX}$ and $f^{*}(\check P_{\bY,\om_{\bY}}^\bu) \cong \check P_{\bX,\om_{\bX}}^\bu$, and therefore
\begin{equation*}
\smash{R^{q}f_{*}P_{\bX,\om_{\bX}}^\bu \cong R^{q}f_{*}f^{*}(\check P_{\bY,\om_{\bY}}^\bu) \cong \check P_{\bY,\om_{\bY}}^\bu \ot_{A_{\bY}} R^{q}f_{*}(A_\bX).}
\end{equation*}
Therefore, the Leray--Serre spectral sequence for the
fibration $\bs f:\bX\ra \bY$ with fibre $G$ can be interpreted as a spectral sequence
\begin{equation*}
\smash{E^{\bu,\bu} \Longra \bH^{\bu}(X,\check P_{\bX,\om_{\bX}}^\bu )\quad\text{with}\quad
E^{p,q}_{2}=\bH^{p}\bigl(Y,\check P_{\bY,\om_{\bY}}^\bu \ot_{A_{\bY}} R^{q}f_{*}(A_{\bX})\bigr).}
\end{equation*}
Since $G$ is finite, only $q=0$ contributes and we get isomorphisms
\begin{equation*}
\smash{\bH^{p}(X,\check P_{\bX,\om_{\bX}}^\bu ) \cong \bH^{p}(Y,\check P_{\bY,\om_{\bY}}^\bu \ot_{A_{\bY}} f_{*}(A_{\bX})).}
\end{equation*}

We also have a projection $\bs\pi:\bY=[\bX/G]\ra[*/G]$ for $*=\Spec\K$ with
fibre $\bX$. The Leray--Serre spectral sequence for $\pi$ gives a
spectral sequence
\begin{equation*}
\smash{E^{\bu,\bu} \Longra \bH^{\bu}(Y,\check P_{\bY,\om_{\bY}}^\bu)\quad\text{with}\quad
E^{p,q}_{2}=\bH^{p}\bigl([*/G],\bH^{q}(X,\check P_{\bX,\om_{\bX}}^\bu)\bigr).}
\end{equation*}
We consider the $\bH^{*}(X,\check P_{\bX,\om_{\bX}}^\bu)$ as
$G$-modules and observe $\bH^*([*/G],-)$ is the same as group cohomology $H^*_{\rm grp}\bigl(G,-)$, giving a spectral sequence
\begin{equation*}
\smash{H^p_{\rm grp}\bigl(G,\bH^q(X,\check P_{\bX,\om_{\bX}}^\bu)\bigr) \Longra \bH^{\bu}(Y,\check P_{\bY,\om_{\bY}}^\bu).}
\end{equation*}\vskip -10pt

\label{sa4ex2}
\end{ex}

\section{Motives on d-critical stacks}
\label{sa5}

We now extend the results of \cite{BJM} to d-critical stacks. Our
main result Theorem \ref{sa5thm2} in \S\ref{sa54}, proved in
\S\ref{sa55}, states that an oriented d-critical stack $(X,s)$ which is of
finite type and locally a global quotient carries a natural motive
in a certain ring of motives $\oM_X^\stm$, defined in~\S\ref{sa53}.

In this section, $\K$ is an algebraically closed field of
characteristic zero, and all $\K$-schemes and Artin $\K$-stacks will
be assumed to be of {\it finite type\/} unless we explicitly say
otherwise. From after Proposition \ref{sa5prop2}, all Artin $\K$-stacks
will also be assumed to have {\it affine geometric stabilizers}.

\subsection{Rings of motives on $\K$-schemes}
\label{sa51}

We begin by defining rings of motives $K_0(\Sch_X),\cM_X,
K_0^{\hat\mu}(\Sch_X),\cM^{\hat\mu}_X$ for a $\K$-scheme $X$. Some
references are Denef and Loeser \cite{DeLo}, Looijenga \cite{Looi},
and Joyce \cite{Joyc1}. Our notation follows Bussi, Joyce and
Meinhardt~\cite{BJM}.

\begin{dfn} Let $X$ be a $\K$-scheme (always assumed of finite
type). Consider pairs $(R,\rho)$, where $R$ is a $\K$-scheme and
$\rho:R\ra X$ is a morphism. Call two pairs $(R,\rho)$, $(R',\rho')$
{\it equivalent\/} if there is an isomorphism $\io:R\ra R'$ with
$\rho=\rho'\ci\io$. Write $[R,\rho]$ for the equivalence class of
$(R,\rho)$. If $(R,\rho)$ is a pair and $S$ is a closed
$\K$-subscheme of $R$ then $(S,\rho\vert_S)$, $(R\sm
S,\rho\vert_{R\sm S})$ are pairs of the same kind. Define the {\it
Grothendieck ring\/ $K_0(\Sch_X)$ of the category\/ $\Sch_X$ of\/
$\K$-schemes over\/} $X$ to be the abelian group generated by
equivalence classes $[R,\rho]$, with the relation that for each
closed $\K$-subscheme $S$ of $R$ we have
\e
\smash{[R,\rho]=[S,\rho\vert_S]+[R\sm S,\rho\vert_{R\sm S}].}
\label{sa5eq1}
\e

Define a product `$\,\cdot\,$' on $K_0(\Sch_X)$ by
\e
\smash{[R,\rho]\cdot[S,\si]=[R\t_{\rho,X,\si}S,\rho\ci\pi_R].}
\label{sa5eq2}
\e
This is compatible with \eq{sa5eq1}, and extends to a biadditive,
commutative, associative product $\cdot:K_0(\Sch_X)\t K_0(\Sch_X)\ra
K_0(\Sch_X)$. It makes $K_0(\Sch_X)$ into a commutative ring, with
identity $1_X=[X,\id_X]$.

Define $\bL=[\bA^1\t X,\pi_X]$ in $K_0(\Sch_X)$. We denote by
\e
\smash{\cM_X=K_0(\Sch_X)[\bL^{-1}]}
\label{sa5eq3}
\e
the ring obtained from $K_0(\Sch_X)$ by inverting $\bL$. When
$X=\Spec\K$ we write $K_0(\Sch_\K),\cM_\K$ instead of
$K_0(\Sch_X),\cM_X$.

The {\it external tensor products\/} $\boxt: K_0(\Sch_X)\t
K_0(\Sch_Y)\ra K_0(\Sch_{X\t Y})$ and $\boxt:
\cM_X\t\cM_Y\ra\cM_{X\t Y}$ are
\e
\smash{\ts\bigl(\sum_{i\in I}c_i[R_i,\rho_i]\bigr)\boxt
\bigl(\sum_{j\in J}d_j[S_j,\si_j]\bigr)=
\sum_{i\in I,\; j\in
J}c_id_j[R_i\t S_j,\rho_i\t\si_j],}
\label{sa5eq4}
\e
for finite $I,J$. They are biadditive, commutative, and associative.
Taking $Y=\Spec\K$, we see that
$\boxt$ makes $K_0(\Sch_X),\cM_X$ into modules over
$K_0(\Sch_\K),\cM_\K$.

Let $\phi:X\ra Y$ be a morphism of $\K$-schemes. Define the {\it
pushforwards\/} $\phi_*: K_0(\Sch_X)\!\ra\! K_0(\Sch_Y)$ and
$\phi_*:\cM_X\!\ra\! \cM_Y$ by
\e
\smash{\phi_*:\ts\sum_{i=1}^nc_i[R_i,\rho_i]\longmapsto
\sum_{i=1}^nc_i[R_i,\phi\ci\rho_i].}
\label{sa5eq5}
\e

Define {\it pullbacks\/} $\phi^*:K_0(\Sch_Y)\ra K_0(\Sch_X)$ and
$\phi^*:\cM_Y\ra\cM_X$ by
\e
\smash{\phi^*:\ts\sum_{i=1}^nc_i[R_i,\rho_i]\longmapsto
\sum_{i=1}^nc_i[R_i\t_{\rho_i,Y,\phi}X,\pi_X].}
\label{sa5eq6}
\e
Pushforwards and pullbacks have the obvious functoriality
properties. As in \cite[Th.~3.5]{Joyc1}, pushforwards and pullbacks
commute in Cartesian squares, that is, if
\e
\begin{gathered}
\xymatrix{
*+[r]{W} \ar[r]_\eta \ar[d]^\theta & *+[l]{Y} \ar[d]_\psi \\
*+[r]{X} \ar[r]^\phi & *+[l]{Z} }
\end{gathered}
\quad
\begin{gathered}
\text{is a Cartesian square in}\\
\text{the category $\Sch_\K$ then}\\
\text{the following commutes:}
\end{gathered}
\quad
\begin{gathered}
\xymatrix@C=35pt{
*+[r]{\cM_W} \ar[r]_{\eta_*} & *+[l]{\cM_Y} \\
*+[r]{\cM_X} \ar[r]^{\phi_*} \ar[u]_{\theta^*} & *+[l]{\cM_Z,}
\ar[u]^{\psi^*} }
\end{gathered}
\label{sa5eq7}
\e
and the analogue holds for $K_0(\Sch_W),\ldots,K_0(\Sch_Z)$.
\label{sa5def1}
\end{dfn}

\begin{dfn} For $n=1,2,\ldots,$ write $\mu_n$ for the group of all
$n^{\rm th}$ roots of unity in $\K$, which is assumed algebraically
closed of characteristic zero, so that $\mu_n\cong\Z_n$. The $\mu_n$ form a projective system, with respect to the maps $\mu_{nd}\ra\mu_n$
mapping $x\mapsto x^d$. Define the group $\hat\mu$ to be the projective limit of the~$\mu_n$. 

Let $R$ be a $\K$-scheme. A {\it good\/ $\mu_n$-action\/} on $R$ is
a group action $r_n:\mu_n\t R\ra R$ such that such that each orbit
is contained in an open affine subscheme of $R$ and $\rho\ci
r_n(\ga)\cong\rho$ for all $\ga \in \mu_n.$ A {\it good\/
$\hat\mu$-action on\/} $R$ is a group action $\hat r:\hat\mu\t R\ra
R$ which factors through a good $\mu_n$-action, for some $n$. We
will write $\hat\io:\hat\mu\t R\ra R$ for the trivial
$\hat\mu$-action on $R$, which is automatically good.

Consider triples $(R,\rho,\hat r)$, where $R$ is a $\K$-scheme,
$\rho:R\ra X$ a morphism, and $\hat r:\hat\mu\t R\ra R$ a good
$\hat\mu$-action on $R$. Call two such triples $(R,\rho,\hat
r),(R',\rho',\hat r')$ {\it equivalent\/} if there exists a
$\hat\mu$-equivariant isomorphism $\io:R\ra R'$ with
$\rho=\rho'\ci\io$. Write $[R,\rho,\hat r]$ for the equivalence
class of $(R,\rho,\hat r)$.

The {\it monodromic Grothendieck group} $K^{\hat\mu}_0(\Sch_X)$ is
the abelian group generated by such equivalence classes
$[R,\rho,\hat r]$, with the relations:
\begin{itemize}
\setlength{\itemsep}{0pt}
\setlength{\parsep}{0pt}
\item[(i)] for each closed $\hat\mu$-invariant $\K$-subscheme
$S$ of $R$, we have
\begin{equation*}
\smash{[R,\rho,\hat r]=[S,\rho\vert_S, \hat r\vert_{S}]+
[R\sm S,\rho\vert_{R\sm S},\hat r\vert_{R\sm S}];}
\end{equation*}
\item[(ii)] given $[R_1,\rho_1, \hat r_1],[R_2,\rho_2, \hat
r_2]$ with $\pi: R_2 \ra R_1$ a $\hat \mu$-equivariant vector
bundle of rank $d$ over $R_1$ and $\rho_2=\rho_1\ci \pi,$ then
\begin{equation*}
\smash{[R_2,\rho_2] = [R_1 \t \bA^d,\rho_1\ci \pi, \hat r_1 \t \hat
\iota].}
\end{equation*}
\end{itemize}
There is a natural biadditive product `$\,\cdot\,$' on
$K^{\hat\mu}_0(\Sch_X)$ given by
\e
\smash{[R,\rho,\hat r]\cdot[S,\si,\hat s]=[R\t_{\rho,X,\si}S,
\rho\ci\pi_R,\hat r \t \hat s],}
\label{sa5eq8}
\e
making $K^{\hat\mu}_0(\Sch_X)$ into a commutative ring, with
identity $1_X=[X,\id_X,\hat\io]$.

Define $\bL=[\bA^1\t X,\pi_X,\hat\io]$ in $K_0^{\hat\mu}(\Sch_X)$.
We denote by
\begin{equation*}
\smash{\cM_X^{\hat\mu}=K_0^{\hat\mu}(\Sch_X)[\bL^{-1}]}
\end{equation*}
the ring obtained from $K_0^{\hat\mu}(\Sch_X)$ by inverting $\bL$.
When $X=\Spec\K$ we write $K_0^{\hat\mu}(\Sch_\K),\cM^{\hat\mu}_\K$
instead of $K_0^{\hat\mu}(\Sch_X),\cM^{\hat\mu}_X$.

The {\it external tensor products\/} $\boxt: K^{\hat\mu}_0(\Sch_X)
\!\t\! K^{\hat\mu}_0(\Sch_Y) \!\ra\! K^{\hat\mu}_0(\Sch_{X\t Y})$
and $\boxt: \cM_X^{\hat\mu}\t\cM_Y^{\hat\mu}\ra\cM^{\hat\mu}_{X\t
Y}$ are
\e
\bigl(\mathop{\ts\sum}\limits_{i\in I}c_i[R_i,\rho_i,\hat
r_i]\bigr)\boxt \bigl(\mathop{\ts\sum}\limits_{j\in
J}d_j[S_j,\si_j,\hat s_j]\bigr)\!=\!\! \mathop{\ts\sum}\limits_{i\in
I,\; j\in J\!\!\!\!\!} c_id_j[R_i\!\t\! S_j,\rho_i\!\t\!\si_j,\hat
r_i\!\t\!\hat s_j],
\label{sa5eq9}
\e
for finite $I,J$. Pushforwards $\phi_*$ and pullbacks $\phi^*$ are
defined for $K_0^{\hat\mu}(\Sch_X),\cM_X^{\hat\mu}$ in the obvious
way, and the analogue of \eq{sa5eq7} holds.

There are natural morphisms of commutative rings
\e
\begin{aligned}
i_X:K_0(\Sch_X)&\longra K_0^{\hat\mu}(\Sch_X),& i_X:\cM_X
&\longra \cM^{\hat\mu}_X, \\
\Pi_X:K_0^{\hat\mu}(\Sch_X)&\longra K_0(\Sch_X),& \Pi_X:
\cM^{\hat\mu}_X &\longra \cM_X,
\end{aligned}
\label{sa5eq10}
\e
given by $i_X:[R,\rho]\mapsto[R,\rho,\hat\io]$ and
$\Pi_X:[R,\rho,\hat r]\mapsto[R,\rho]$.
\label{sa5def2}
\end{dfn}

Following Looijenga \cite[\S 7]{Looi} and Denef and Loeser \cite[\S
5]{DeLo}, we introduce a second multiplication `$\od$' on
$K_0^{\hat\mu}(\Sch_X),\cM_X^{\hat\mu}$ (written `$*$'
in~\cite{Looi,DeLo}).

\begin{dfn} Let $X$ be a $\K$-scheme and $[R,\rho,\hat
r],[S,\si,\hat s]$ be generators of $K_0^{\hat\mu}(\Sch_X)$. Then
there exists $n\ge 1$ such that the $\hat\mu$-actions $\hat r,\hat
s$ on $R,S$ factor through $\mu_n$-actions $r_n,s_n$. Define $J_n$
to be the Fermat curve
\begin{equation*}
\smash{J_n=\bigl\{(t,u)\in(\bA^1\sm\{0\})^2:t^n+u^n=1\bigr\}.}
\end{equation*}
Let $\mu_n\t\mu_n$ act on $J_n\t (R\t_X S)$ by
\begin{equation*}
\smash{(\al,\al')\cdot\bigl((t,u),(v,w)\bigr)=\bigl((\al\cdot t,\al'\cdot u),(r_n(\al)(v),s_n(\al')(w))\bigr).}
\end{equation*}
Write $J_n(R,S)=(J_n\t (R\t_X S))/(\mu_n\t\mu_n)$ for the quotient
$\K$-scheme, and define a $\mu_n$-action $\up_n$ on $J_n(R,S)$ by
\begin{equation*}
\smash{\up_n(\al)\bigl((t,u),v,w\bigr)(\mu_n\t\mu_n)= \bigl((\al\cdot
t,\al\cdot u),v,w\bigr)(\mu_n\t\mu_n).}
\end{equation*}
Let $\hat\up$ be the induced good $\hat\mu$-action on $J_n(R,S)$,
and set
\e
\smash{[R,\rho,\hat r]\od[S,\si,\hat s]=(\bL-1)\cdot\bigl[(R\t_X
S)/\mu_n,\hat\io\bigr]- \bigl[J_n(R,S),\hat\up\bigr]}
\label{sa5eq11}
\e
in $K_0^{\hat\mu}(\Sch_X)$ and $\cM_X^{\hat\mu}$. This turns out to
be independent of $n$, and defines commutative, associative products
$\od$ on $K_0^{\hat\mu}(\Sch_X)$ and $\cM_X^{\hat\mu}$.

Let $X,Y$ be $\K$-schemes. As for Definitions \ref{sa5def1} and
\ref{sa5def2}, we define products
\begin{equation*}
\smash{\bd:K^{\hat\mu}_0(\Sch_X) \!\t\! K^{\hat\mu}_0(\Sch_Y) \!\ra\!
K^{\hat\mu}_0(\Sch_{X\t Y}),\;\>
\bd:\cM^{\hat\mu}_X\t \cM^{\hat\mu}_X\!\ra\! \cM^{\hat\mu}_{X\t Y}}
\end{equation*}
by following the definition above for $\od$, but taking products $R\t S$ rather than fibre products $R\t_X S$. These $\bd$ are commutative and associative. Taking $Y=\Spec\K$, we see that $\bd$
makes $K^{\hat\mu}_0(\Sch_X),\cM_X^{\hat\mu}$ into modules over
$K^{\hat\mu}_0(\Sch_\K),\cM_\K^{\hat\mu}$.

For generators $[R,\rho,\hat r]$ and $[S,\si,\hat
\iota]=i_X([S,\si])$ in $K_0^{\hat\mu}(\Sch_X)$ or $\cM^{\hat\mu}_X$
where $[S,\si,\hat \iota]$ has trivial $\hat\mu$-action $\hat\io$,
one can show that $[R,\rho,\hat r]\od[S,\si,\hat \iota]=[R,\rho,\hat
r]\cdot[S,\si,\hat \iota]$. Thus $i_X$ is a ring morphism
$\bigl(K_0(\Sch_X),\cdot\bigr)\ra
\bigl(K_0^{\hat\mu}(\Sch_X),\od\bigr)$ and
$\bigl(\cM_X,\cdot\bigr)\ra \bigl(\cM^{\hat\mu}_X,\od\bigr)$.
However, $\Pi_X$ is not a ring morphism
$\bigl(K_0^{\hat\mu}(\Sch_X),\od\bigr)\ra
\bigl(K_0(\Sch_X),\cdot\bigr)$ or
$\bigl(\cM^{\hat\mu}_X,\od\bigr)\ra\bigl(\cM_X,\cdot\bigr)$. Since
$\bL=[\bA^1\t X,\pi_X,\hat\io]$ this implies that
$M\cdot\bL=M\od\bL$ for all $M$ in
$K_0^{\hat\mu}(\Sch_X),\cM^{\hat\mu}_X$.
\label{sa5def3}
\end{dfn}

\begin{dfn} Define the element $\bL^{1/2}$ in
$K^{\hat\mu}_0(\Sch_{X})$ and $\cM^{\hat\mu}_X$ by
\e
\smash{\bL^{1/2}=[X,\id_X,\hat\io]-[X\t\mu_2,\hat r],}
\label{sa5eq12}
\e
where $[X,\id_X,\hat\io]$ with trivial $\hat\mu$-action $\hat\io$ is
the identity $1_X$ in $K^{\hat\mu}_0(\Sch_{X}),\cM^{\hat\mu}_X$, and
$X\t\mu_2=X\t\{1,-1\}$ is two copies of $X$ with nontrivial
$\hat\mu$-action $\hat r$ induced by the left action of $\mu_2$ on
itself, exchanging the two copies of $X$. Applying \eq{sa5eq11} with
$n=2$, we can show that $\bL^{1/2}\od \bL^{1/2}=\bL$. Thus,
$\bL^{1/2}$ in \eq{sa5eq12} is a square root for $\bL$ in the rings
$\bigl(K_0^{\hat\mu}(\Sch_X),\od\bigr),\bigl(\cM^{\hat\mu}_X,\od\bigr)$.
Note that $\bL^{1/2}\cdot \bL^{1/2}\ne\bL$.

Equivalently, we could have defined
\e
\smash{\bL_X^{1/2}=[X,\id_X,\hat\io] \bd \bL^{1/2}_{\K}\in
K^{\hat\mu}_0(\Sch_{X}),}
\label{sa5eq13}
\e
where $\bL^{1/2}_{\K}\in K^{\hat\mu}_0(\Sch_{\K}).$ We can now
define $\bL^{n/2}\in K_0^{\hat\mu}(\Sch_X)$ for $n\ge 0$ and $\bL^{n/2}\in\cM^{\hat\mu}_X$ for $n\in \Z$ in the obvious way, such that
$\bL^{m/2}\od\bL^{n/2}=\bL^{(m+n)/2}$.
\label{sa5def4}
\end{dfn}

Next, following \cite[\S 2.5]{BJM}, which was motivated by ideas in
Kontsevich and Soibelman \cite[\S 4.5]{KoSo1}, we define principal
$\Z/2\Z$-bundles $P\ra X$, associated motives $\Up(P)$, and a quotient
ring of motives $\oM_X^{\hat\mu}$ in
which~$\Up(P\ot_{\Z/2\Z}Q)=\Up(P)\od\Up(Q)$ for all $P,Q$.

\begin{dfn} Let $X$ be a $\K$-scheme. A {\it principal\/
$\Z/2\Z$-bundle\/} $P\ra X$ is a proper, surjective, \'etale morphism
of $\K$-schemes $\pi:P\ra X$ together with a free involution
$\si:P\ra P$, such that the orbits of $\Z/2\Z=\{1,\si\}$ are the
fibres of $\pi$. The {\it trivial\/ $\Z/2\Z$-bundle\/} is
$\pi_X:X\t\Z/2\Z\ra X$. We will use the ideas of {\it isomorphism\/}
of principal bundles $\io:P\ra Q$, {\it section\/} $s:X\ra P$, {\it
tensor product\/} $P\ot_{\Z/2\Z}Q$, and {\it pullback\/} $f^*(P)\ra Y$
under a 1-morphism of stacks $f:Y\ra X$, all of which are defined in
the obvious ways.

Write $(\Z/2\Z)(X)$ for the abelian group of isomorphism classes $[P]$
of principal $\Z/2\Z$-bundles $P\ra X$, with multiplication
$[P]\cdot[Q]=[P\ot_{\Z/2\Z}Q]$ and identity $[X\t\Z/2\Z]$. Since
$P\ot_{\Z/2\Z}P\cong X\t\Z/2\Z$ for each $P\ra X$, each element of
$(\Z/2\Z)(X)$ is self-inverse, and has order 1 or 2.

If $\pi:P\ra X$ is a principal $\Z/2\Z$-bundle over $X$, define a motive
\begin{equation*}
\smash{\Up(P)=\bL^{-1/2}\od\bigl([X,\id,\hat\io]-[P,\pi,\hat r]\bigr)\in
\cM_X^{\hat\mu},}
\end{equation*}
where $\hat r$ is the $\hat\mu$-action on $P$ induced by the
$\mu_2$-action on $P$ from the principal $\Z/2\Z$-bundle structure, as
$\mu_2\cong\Z/2\Z$.  If $P=X\t\Z/2\Z$ is trivial then
\begin{align*}
\Up(X\t\Z/2\Z)&=\bL^{-1/2}\od\bigl([X,\id,\hat\io]
-[X\t\Z/2\Z,\pi,\hat r]\bigr)\\
&=\bL^{-1/2}\od\bL^{1/2}\od[X,\id,\hat\io]=[X,\id,\hat\io],
\end{align*}
using \eq{sa5eq12}. Note that $[X,\id,\hat\io]$ is the identity in the
ring~$\cM_X^{\hat\mu}$.

As $\Up(P)$ only depends on $P$ up to isomorphism, $\Up$ factors via
$(\Z/2\Z)(X)$, and we may consider $\Up$ as a map~$(\Z/2\Z)(X)\ra
\cM_X^{\hat\mu}$.

For our applications, we want $\Up:(\Z/2\Z)(X)\ra \cM_X^{\hat\mu}$ to be
a group morphism with respect to the multiplication $\od$ on
$\cM_X^{\hat\mu}$, but we cannot prove that it is. Our solution is
to pass to a quotient ring $\oM_X^{\hat\mu}$ of $\cM_X^{\hat\mu}$
such that the induced map $\Up:(\Z/2\Z)(X)\ra\oM_X^{\hat\mu}$ is a group
morphism. If we simply defined $\oM_X^{\hat\mu}$ to be the quotient
ring of $\cM_X^{\hat\mu}$ by the relations
$\Up(P\ot_{\Z/2\Z}Q)-\Up(P)\od\Up(Q)=0$ for all $[P],[Q]$ in $(\Z/2\Z)(X)$ then pushforwards $\phi_*:\oM_X^{\hat\mu}\ra\oM_Y^{\hat\mu}$ would
not be defined for general $\phi:X\ra Y$. So we impose a
more complicated relation.

For each $\K$-scheme $Y$, define $I_Y^{\hat\mu}$ to be the ideal in
the commutative ring $\bigl(\cM_Y^{\hat\mu},\od\bigr)$ generated by
elements $\phi_*\bigl(\Up(P\ot_{\Z/2\Z}Q)-\Up(P)\od\Up(Q)\bigr)$ for
all $\K$-scheme morphisms $\phi:X\ra Y$ and principal $\Z/2\Z$-bundles
$P,Q\ra X$, and define $\oM_Y^{\hat\mu}=\cM_Y^{\hat\mu}/
I_Y^{\hat\mu}$ to be the quotient, as a commutative ring with
multiplication `$\od$', with projection $\Pi_Y^{\hat\mu}:
\cM_Y^{\hat\mu}\ra\oM_Y^{\hat\mu}$. Kontsevich and Soibelman
\cite[\S 4.5]{KoSo1} introduce a relation in their motivic rings
which has a similar effect.

Note that in $\oM_Y^{\hat\mu}$ we do not have the second
multiplication `$\,\cdot\,$', since we do not require
$I_Y^{\hat\mu}$ to be an ideal in $\bigl(\cM_Y^{\hat\mu},
\cdot\bigr)$. Also $\boxt$ and $\Pi_Y:\cM_Y^{\hat\mu}\ra\cM_Y$ on
$\cM_Y^{\hat\mu}$ do not descend to $\oM_Y^{\hat\mu}$. Apart from
this, all the structures on $\cM_Y^{\hat\mu}$ above descend to
$\oM_Y^{\hat\mu}$: operations $\od,\bd$, pushforwards $\phi_*$ and
pullbacks $\phi^*$, and elements $\bL,\bL^{1/2},\ab\Up(P)$. By
definition, $\oM_X^{\hat\mu}$ has the property that
\begin{equation*}
\smash{\Up(P\ot_{\Z/2\Z}Q)=\Up(P)\od\Up(Q)\quad\text{in $\oM_X^{\hat\mu}$}}
\end{equation*}
for all principal $\Z/2\Z$-bundles~$P,Q\ra X$.
\label{sa5def5}
\end{dfn}

\subsection{Motivic vanishing cycles, and d-critical loci}
\label{sa52}

Following Denef and Loeser \cite{DeLo}, we define motivic nearby
cycles, motivic Milnor fibres, and motivic vanishing cycles:

\begin{dfn} Let $U$ be a smooth $\K$-scheme and $f:U\ra\bA^1$ a
regular function, and set $U_0=f^{-1}(0)\subseteq U$. Then Denef and
Loeser \cite[\S 3.5]{DeLo} and Looijenga \cite[\S 5]{Looi} define the
{\it motivic nearby cycle} of $f$, an element $MF_{U,f}^{\rm mot}$
of $\cM_{U_0}^{\hat\mu}$ or $\oM_{U_0}^{\hat\mu}$. It has an
intrinsic definition using arc spaces and the motivic zeta function,
which we will not explain, but we will give a formula \cite[\S
3.3]{DeLo}, \cite[\S 5]{Looi} for $MF_{U,f}^{\rm mot}$ involving
choosing a resolution of $f$.

If $f=0$ then $MF_{U,f}^{\rm mot}=0$, so suppose $f$ is not
constant. By Hironaka's Theorem \cite{Hiro} we can choose a {\it
resolution\/} $(\ti U,\pi)$ of $f$. That is, $\ti U$ is a smooth
$\K$-scheme and $\pi:\ti U\ra U$ a proper morphism, such that
$\pi\vert_{\ti U\sm \pi^{-1}(U_0)}:\ti U\sm \pi^{-1}(U_0)\ra U \sm
U_0$ is an isomorphism, and $\pi^{-1}(U_0)^\red$ has only normal
crossings as a $\K$-subscheme of~$\ti U$.

Write $E_i$, $i\in J$ for the irreducible components of
$\pi^{-1}(U_0)$. For each $i\in J$, denote by $N_i$ the multiplicity
of $E_i$ in the divisor of $f\ci\pi$ on $\ti U$, and by $\nu_i - 1$ the
multiplicity of $E_i$ in the divisor of $\pi^*(\d x)$, where $\d x$
is a local non vanishing volume form at any point of $\pi(E_i)$. For
$I \subset J$, we consider the smooth $\K$-scheme~$E^\ci_I=
\bigl(\bigcap_{i \in I}E_i\bigr) \sm \bigl(\bigcup_{j \in J \sm I}
E_j\bigr)$.

Let $m_I=\gcd(N_i)_{i \in I}$. We introduce an unramified Galois
cover $\ti E^\ci_I$ of $E^\ci_I$, with Galois group $\mu_{m_I}$, as
follows. Let $\ti U'$ be an affine Zariski open subset of $\ti U$, such
that, on $\ti U'$, $f \ci \pi = uv^{m_I}$, with $u:\ti U'\ra\bA^1\sm\{0\}$
and $v:\ti U'\ra\bA^1$. Then the restriction of $\ti E^\ci_I$ above
$E^\ci_I \cap \ti U'$, denoted by $\ti E_I^\ci \cap \ti U'$, is defined as
\begin{equation*}
\smash{\ti E_I^\ci \cap \ti U'=\bigl\{(z,w)\in\bA^1\t(E^\ci_I \cap \ti U'):z^{m_I}=u(w)^{-1}\bigr\}.}
\end{equation*}
Gluing together the covers $\ti E^\ci_I \cap \ti U'$ in the
obvious way, we obtain the cover $\ti E^\ci_I$ of $E^\ci_I$ which
has a natural $\mu_{m_I}$-action $\rho_I$, obtained by multiplying
the $z$-coordinate by elements of $\mu_{m_I}$. This
$\mu_{m_I}$-action on $\ti E^\ci_I$ induces a $\hat\mu$-action
$\hat\rho_I$ on $\ti E^\ci_I$. Then
\e
\smash{MF_{U,f}^{\rm mot}=
\ts\sum_{\es\ne I\subseteq J}(1-\bL)^{|I|-1}
\bigl[\ti E^\ci_I,\pi_{U_0},\hat\rho_I\bigr]
\quad\text{in $\cM_{U_0}^{\hat\mu}$.}}
\label{sa5eq14}
\e
It is independent of the choice of resolution $(\ti U,\pi)$. The
fibre $MF_{U,f}^{\rm mot}\vert_x$ at each $x\in U_0$ is called the
{\it motivic Milnor fibre\/} of $f$ at~$x$.

Now let $X=\Crit(f)\subseteq U$, as a closed $\K$-subscheme of $U$.
Since $f$ is constant on the reduced scheme $X^\red$, $f(X)$ is
finite, and we may write $X=\coprod_{c\in f(X)}X_c$, where
$X_c\subseteq X$ is the open and closed $\K$-subscheme
with~$X_c^\red=f\vert_{X^\red}^{-1}(c)$.

Consider the restriction $MF_{U,f}^{\rm mot}\vert_{U_0\sm X_0}$ in
$\cM_{U_0\sm X_0}^{\hat\mu}$ or $\oM_{U_0\sm X_0}^{\hat\mu}$. We can
choose $(\ti U,\pi)$ above with $\pi\vert_{\ti
U\sm\pi^{-1}(X_0)}:\ti U\sm\pi^{-1}(X_0)\ra U\sm X_0$ an
isomorphism. Write $D_1,\ldots,D_k$ for the irreducible components
of $\pi^{-1}(U_0\sm X_0)\cong U_0\sm X_0$. They are disjoint as
$\pi^{-1}(U_0\sm X_0)$ is nonsingular. The closures
$\,\ov{\!D}_1,\ldots,\,\ov{\!D}_k$ (which need not be disjoint) are
among the divisors $E_i$, so we write $\,\ov{\!D}_a=E_{i_a}$ for
$a=1,\ldots,k$, with $\{i_1,\ldots,i_k\}\subseteq I$. Clearly
$N_{i_a}=\nu_{i_a}=1$ for $a=1,\ldots,k$.

Then in \eq{sa5eq14} the only nonzero contributions to $MF_{U,f}^{\rm
mot}\vert_{U_0\sm X_0}$ are from $I=\{i_a\}$ for $a=1,\ldots,k$,
with $\ti E^\ci_{\{i_a\}}\cong E^\ci_{\{i_a\}}\cong D_a$, and the
$\hat\mu$-action on $\ti E^\ci_{\{i_a\}}$ is trivial as it factors
through the action of $\mu_1=\{1\}$. Hence
\begin{align*}
MF_{U,f}^{\rm mot}\vert_{U_0\sm X_0}&\!=\!\!\ts\sum\limits_{a=1}^k
\bigl[\ti E^\ci_{\{i_a\}},\pi_{U_0\sm X_0},\hat\io\bigr]
\!=\!\!\ts\sum\limits_{a=1}^k\bigl[D_a,\pi_{U_0\sm X_0},\hat\io\bigr]
\!=\!\bigl[U_0\sm X_0,\id_{U_0\sm X_0},\hat\io\bigr].
\end{align*}
Therefore $[U_0,\id_{U_0},\hat\io]-MF_{U,f}^{\rm mot}$ is supported on
$X_0\subseteq U_0$, and by restricting to $X_0$ we regard it as an
element of $\cM_{X_0}^{\hat\mu}$ or~$\oM_{X_0}^{\hat\mu}$.

Define the {\it motivic vanishing cycle\/} $MF_{U,f}^{\rm
mot,\phi}$ of $f$ in $\cM_X^{\hat\mu}$ or $\oM_X^{\hat\mu}$ by
\e
\smash{MF_{U,f}^{\rm mot,\phi}\big\vert_{X_c}=\bL^{-\dim U/2}\od
\bigl([U_c,\id_{U_c},\hat\io]-MF_{U,f-c}^{\rm mot}\bigr)\big\vert_{X_c}}
\label{sa5eq15}
\e
for each $c\in f(X)$, where $\od$ and $\bL^{-\dim U/2}$ are as in
Definitions \ref{sa5def3} and~\ref{sa5def4}.
\label{sa5def6}
\end{dfn}

Here is \cite[Th.~5.10]{BJM}, which we will generalize to stacks in
Theorem~\ref{sa5thm2}.

\begin{thm} Let\/ $(X,s)$ be an algebraic d-critical locus with
orientation $K_{X,s}^{1/2},$ for $X$ of finite type. Then there
exists a unique motive $MF_{X,s}\in\smash{\oM^{\hat\mu}_X}$ with the
property that if\/ $(R,U,f,i)$ is a critical chart on $(X,s),$ then
\e
\smash{MF_{X,s}\vert_R=i^*\bigl(MF_{U,f}^{\rm mot,\phi}\bigr)\od\Up
(Q_{R,U,f,i})\quad\text{in\/ $\oM^{\hat\mu}_R,$}}
\label{sa5eq16}
\e
where $Q_{R,U,f,i}\ra R$ is the principal\/ $\Z/2\Z$-bundle
parametrizing local isomorphisms $\al:K_{X,s}^{1/2}
\vert_{R^\red}\!\ra\! i^*(K_U)\vert_{R^\red}$ with\/ $\al\!\ot\!\al\!=\!
\io_{R,U,f,i},$ for $\io_{R,U,f,i}$ as in\/~\eq{sa3eq4}.

\label{sa5thm1}
\end{thm}

We prove a result on smooth pullbacks and pushforwards of the
motives $MF_{X,s}$ of Theorem \ref{sa5thm1}, a motivic analogue of
Proposition~\ref{sa4prop3}(a).

\begin{prop} Let\/ $\phi:(X,s)\ra (Y,t)$ be a morphism of (finite
type) algebraic d-critical loci in the sense of\/ {\rm\S\ref{sa31},}
and suppose $\phi:X\ra Y$ is smooth of relative dimension $n$. Let\/
$K_{Y,t}^{1/2}$ be an orientation for $(Y,t),$ so that Corollary\/
{\rm\ref{sa3cor1}} defines an induced orientation\/ $\smash{K_{X,s}^{1/2}}$ for $(X,s)$. Theorem\/ {\rm\ref{sa5thm1}} now defines motives
$MF_{X,s},MF_{Y,t}$ on $X,Y$. These are related by
\ea
\phi^*\bigl(MF_{Y,t}\bigr)&=\bL^{n/2}\od MF_{X,s}\in\oM^{\hat\mu}_X,
\label{sa5eq17}\\
\phi_*\bigl(MF_{X,s}\bigr)&=\bL^{-n/2}\od MF_{Y,t} \od
[X,\phi,\hat\io] \in \oM^{\hat\mu}_Y.
\label{sa5eq18}
\ea

\label{sa5prop1}
\end{prop}

\begin{proof} If $x\in X$ with $\phi(x)=y\in Y$ then the proof of
Proposition \ref{sa3prop1} above in \cite{Joyc2} shows we may choose
critical charts $(R,U,f,i), (S,V,g,j)$ on $(X,s),(Y,t)$ with $x\in
R$, $y\in\phi(R)\subseteq S$ of minimal dimensions $\dim U=\dim
T_xX$, $\dim V=\dim T_yY$, and $\Phi:U\ra V$ smooth of relative
dimension $n$ with $f=g\ci\Phi$ and~$\Phi\ci i=j\ci\phi$.

Let $\pi:\ti V \ra V$ be an embedded resolution of singularities of
$g$. Then $\ti U:=U\t_{\Phi,V,\pi}\ti V$ is an embedded resolution
of singularities of $f$, since $\Phi$ is smooth and $f=g\ci\Phi$. As
in Definition \ref{sa5def6}, let $F_i$ for $i\in J$ be the
irreducible components of $\pi^{-1}(V_0)$, so that
$\pi^{-1}(V_0)=\bigcup_{i\in J}F_i$, with multiplicities $N_i$ in
the divisor of $g\ci\pi$ on $\ti V$, and $\nu_i-1$ in the divisor of
$\pi^*(\d x)$, and define $F^\ci_I=\bigl(\bigcap_{i \in I}F_i\bigr)
\sm \bigl(\bigcup_{j \in J\sm I}F_j\bigr)$ and covers $\ti
F^\ci_I\ra F^\ci_I$ for all~$I\subseteq J$.

Define $E_i=U\t_{\Phi,V,\pi\vert_{F_i}}F_i\subset\pi^{-1}(U_0)
\subset\ti U$. Then $\pi^{-1}(U_0)=\bigcup_{i\in J}E_i$. The $E_i$
need not be irreducible, or nonempty, but this is not important.
Neglecting this, we can treat the $E_i$, $i\in J$ as the components
for $(\ti U,\pi)$ in Definition \ref{sa5def6}, and then they have
the same multiplicities $N_i,\nu_i$ as the $F_i$ for $(\ti V,\pi)$,
and the $E^\ci_I,\ti E^\ci_I$ for $I\subseteq J$ defined in
Definition \ref{sa5def6} satisfy $E_I^\ci\cong U\t_VF_I^\ci$ and $\ti
E_I^\ci\cong U\t_V\ti F_I^\ci$. Thus we have
\begin{align*}
MF_{U,f}^{\rm mot}&=\ts\sum_{\es\ne I\subseteq J}(1-\bL)^{|I|-1}
\bigl[\ti E^\ci_I,\pi_{U_0},\hat\rho_I\bigr]\\
&=\ts\sum_{\es\ne I\subseteq J}(1-\bL)^{|I|-1}
\bigl[\ti F^\ci_I\t_{\pi_{V_0},V_0,\Phi\vert_{U_0}}U_0,
\pi_{U_0},\hat\rho_I\bigr]\\
&=\ts\Phi\vert_{U_0}^*\raisebox{-1pt}{$\displaystyle\Bigl[$}
\sum_{\es\ne I\subseteq J}(1-\bL)^{|I|-1}
\bigl[\ti F^\ci_I,\pi_{V_0},\hat\rho_I\bigr]
\raisebox{-1pt}{$\displaystyle\Bigr]$}
=\Phi\vert_{U_0}^*\bigl(MF_{V,g}^{\rm mot}\bigr).
\end{align*}
So from \eq{sa5eq15} we deduce that
\e
\smash{\Phi\vert_{\Crit(f)}^*(MF_{V,g}^{\rm mot,\phi}) =\bL^{n/2}\od
MF_{U,f}^{\rm mot,\phi},}
\label{sa5eq19}
\e
using $\Phi\vert_{U_c}^*\bigl([V_c,\id_{V_c},
\hat\io]\bigr)=[U_c,\id_{U_c},\hat\io]$, where the factor
$\bL^{n/2}$ is to convert the factor $\bL^{-\dim U/2}$ in
$MF_{U,f}^{\rm mot,\phi}$ to the factor $\bL^{-\dim V/2}$ in
$MF_{V,g}^{\rm mot,\phi}$.

Combining \eq{sa5eq19} with \eq{sa5eq16} for $(X,s),(R,U,f,i)$ and
the pullback of \eq{sa5eq16} for $(Y,t),(S,V,g,j)$ by
$\phi\vert_R:R\ra S$, and noting that $\phi^*\ci
j^*=i^*\ci\Phi\vert_{\Crit(f)}^*$ since $j\ci\phi=\Phi\ci i$, we
deduce the restriction of \eq{sa5eq17} to $R\subseteq X$. As we can
cover $X$ by such open $R$, this proves \eq{sa5eq17}. Equation
\eq{sa5eq18} follows by applying $\phi_*$ and noting that
$\phi_*\ci\phi^*(M)=M\od [X,\phi,\hat\io]$ for all $\phi:X\ra Y$
and~$M\in\oM^{\hat\mu}_Y$.
\end{proof}

\subsection{Rings of motives over Artin stacks}
\label{sa53}

We now generalize the material of \S\ref{sa51} to Artin stacks. Our
definitions are new, but very similar to work by Joyce \cite{Joyc1}
on `stack functions', and Kontsevich and Soibelman \cite[\S 4.1--\S
4.2]{KoSo2}. As in \cite{Joyc1}, we restrict our attention to Artin
$\K$-stacks $X$ (always assumed of finite type) with {\it affine
geometric stabilizers}. In \S\ref{sa54}--\S\ref{sa55} we will
restrict further, to stacks which are {\it locally a global
quotient}.

\begin{dfn} An Artin $\K$-stack $X$ has {\it affine geometric
stabilizers\/} if the stabilizer group $\Iso_X(x)$ is an affine
algebraic group for all points~$x\in X$.

An Artin $\K$-stack $X$ is {\it locally a global quotient\/} if we
may cover $X$ by Zariski open $\K$-substacks $Y\subseteq X$
equivalent to global quotients $[S/\GL(n,\K)]$, where $S$ is a
$\K$-scheme with a $\GL(n,\K)$-action.

If $X$ is locally a global quotient then it has affine geometric
stabilizers, since the stabilizer groups of $[S/\GL(n,\K)]$ are
closed $\K$-subgroups of $\GL(n,\K)$, and so are affine. The authors
do not know any example of an Artin $\K$-stack with affine geometric
stabilizers which is not locally a global quotient.
\label{sa5def7}
\end{dfn}

Deligne--Mumford stacks have affine geometric stabilizers, and are
locally a global quotient if their stabilizers are generically
trivial. If $\cM$ is a moduli stack of coherent sheaves $F$ on a
projective scheme $Y$, then using Quot-schemes one can show that
$\cM$ is locally a global quotient. If $\cM$ is a moduli stack of
complexes $F^\bu$ in $D^b\coh(Y)$ with $\Ext^{<0}(F^\bu,F^\bu)=0$
then $\cM$ has affine geometric stabilizers, since $\Iso_\cM(F^\bu)$
is the invertible elements in the finite-dimensional algebra
$\Hom(F^\bu,F^\bu)$, and so is affine. We require affine geometric
stabilizers to use a result of Kresch~\cite[Prop.~3.5.9]{Kres}:

\begin{prop}[Kresch] Let\/ $X$ be a (finite type) Artin $\K$-stack
with affine geometric stabilizers. Then $X$ admits a stratification
$X=\coprod_{i\in I}X_i,$ for $I$ a finite set and\/ $X_i\subseteq X$
a locally closed\/ $\K$-substack, such that $X_i$ is equivalent to a
global quotient stack\/ $[S_i/\GL(n_i,\K)]$ for each\/ $i\in I,$
where $S_i$ is a (finite type)\/ $\K$-scheme with an action of\/
$\GL(n_i,\K)$. Conversely, any Artin $\K$-stack\/ $X$ admitting such
a stratification has affine geometric stabilizers.
\label{sa5prop2}
\end{prop}

For the rest of this paper, all Artin $\K$-stacks $X$ are assumed to
have affine geometric stabilizers. Here are the analogues of
Definitions \ref{sa5def1} and~\ref{sa5def2}:

\begin{dfn} Let $X$ be an Artin $\K$-stack (always assumed to be of
finite type, with affine geometric stabilizers). Consider pairs
$(R,\rho)$, where $R$ is a $\K$-scheme and $\rho:R\ra X$ a
1-morphism. Call two pairs $(R,\rho)$, $(R',\rho')$ {\it
equivalent\/} if there exists an isomorphism $\io:R\ra R'$ such that
$\rho'\ci\io$ and $\rho$ are 2-isomorphic 1-morphisms $R\ra X$.
Write $[R,\rho]$ for the equivalence class of $(R,\rho)$. Define the
Grothendieck ring $K_0(\Sch_X)$ of the category of $\K$-schemes over
$X$ to be the abelian group generated by equivalence classes
$[R,\rho]$, such that as for \eq{sa5eq1} for each closed
$\K$-subscheme $S$ of $R$ we have
\begin{equation*}
\smash{[R,\rho]=[S,\rho\vert_S]+[R\sm S,\rho\vert_{R\sm S}].}
\end{equation*}
When $X=\Spec\K$ we write $K_0(\Sch_\K)$ instead of~$K_0(\Sch_X)$.

Define a biadditive, commutative, associative product `$\,\cdot\,$'
on $K_0(\Sch_X)$ as in \eq{sa5eq2}. It makes $K_0(\Sch_X)$ into a
commutative ring, in general without identity. If $X$ is a
$\K$-scheme $K_0(\Sch_X)$ is as in Definition \ref{sa5def1}, with
identity~$[X,\id_X]$.

For Artin $\K$-stacks $X,Y$, define a biadditive, commutative,
associative {\it external tensor product\/} $\boxt: K_0(\Sch_X)\t
K_0(\Sch_Y)\ra K_0(\Sch_{X\t Y})$ by \eq{sa5eq4}. Taking $Y=\Spec\K$
we see that $\boxt$ makes $K_0(\Sch_X)$ into a module
over~$K_0(\Sch_\K)$.

Next we will define a stack analogue $\cM_X^\stk$ of the motivic
ring $\cM_X$ of \eq{sa5eq3} for $\K$-schemes $X$. Since we have no
identity in $K_0(\Sch_X)$ if $X$ is not a scheme, and we have not
defined a Tate motive $\bL$ in $K_0(\Sch_X)$, the analogue of
\eq{sa5eq3} does not make sense. Instead, we use the
$K_0(\Sch_\K)$-module structure, and define
\e
\smash{\cM_X^\stk=K_0(\Sch_X)\ot_{K_0(\Sch_\K)}K_0(\Sch_\K)\bigl[\bL^{-1},(\bL^k-1)^{-1},\; k=1,2,\ldots\bigr],}
\label{sa5eq20}
\e
where $\bL\in K_0(\Sch_\K)$ is as in Definition \ref{sa5def1}. The
product `$\,\cdot\,$' descends to $\cM^\stk_X$. When $X=\Spec\K$ we
write $\cM^\stk_\K$ instead of~$\cM^\stk_X$.

Note that for $X$ a $\K$-scheme, $\cM_X^\stk$ is not isomorphic to
$\cM_X$ in \eq{sa5eq3}, since we invert $\bL^k-1$ in $\cM_X^\stk$
but not in $\cM_X$. There is a natural projection $\cM_X\ra \cM_X^\stk$. The reason we invert $\bL^k\!-\!1$ as well as $\bL$ is that the motive of $\GL(n,\K)$ in $\cM_\K$~is
\begin{equation*}
\smash{[\GL(n,\K)]:=[\GL(n,\K),\pi_{\Spec\K}]=\bL^{n(n-1)/2}
\ts\prod_{k=1}^n(\bL^k-1),}
\end{equation*}
so that $[\GL(n,\K)]$ is invertible in $\cM^\stk_\K$.

Let $X$ be an Artin $\K$-stack (as usual of finite type, with affine
geometric stabilizers). Then Proposition \ref{sa5prop2} gives a
finite stratification $X=\coprod_{i\in I}X_i$ with
$X_i\simeq[S_i/\GL(n_i,\K)]$. Write $\pi_i:S_i\ra X$ for the
composition of 1-morphisms
$S_i\ra[S_i/\GL(n_i,\K)]\,{\buildrel\sim\over\longra}\,X_i\hookra
X$. Define elements $1_X,\bL\in\cM_X^\stk$ by
\e
\begin{split}
1_X&=\ts\sum_{i\in I}[\GL(n_i,\K)]^{-1}\boxt[S_i,\pi_i],\\
\bL&=\ts\sum_{i\in I}[\GL(n_i,\K)]^{-1}\boxt[\bA^1\t S_i,\pi_i\ci\pi_{S_i}],
\end{split}
\label{sa5eq21}
\e
where $[\GL(n_i,\K)]^{-1}\in\cM^\stk_\K$ exists as above. It is easy to show that these $1_X,\bL$ are independent of the choice of
$I,X_i,S_i,n_i$, and $1_X$ is the identity in~$(\cM_X^\stk,\cdot)$.

Let $\phi:X\ra Y$ be a 1-morphism of Artin $\K$-stacks. Define the
{\it pushforwards\/} $\phi_*: K_0(\Sch_X)\ra K_0(\Sch_Y)$ and
$\phi_*:\cM_X^\stk\ra \cM^\stk_Y$ by \eq{sa5eq5}. If $\phi$ is
representable in $\K$-schemes we may also define {\it pullbacks\/}
$\phi^*:K_0(\Sch_Y)\ra K_0(\Sch_X)$ and $\phi^*:\cM^\stk_Y\ra\cM^\stk_X$ by \eq{sa5eq6}. (Here $\phi$ is {\it representable in\/ $\K$-schemes} if $X\t_{\phi,Y,u}U$ is a $\K$-scheme for all $u:U\ra Y$ with $U$ a $\K$-scheme.) But if $\phi$ is not representable in $\K$-schemes then $R_i\t_{\rho_i,Y,\phi}X$ in \eq{sa5eq6} may not be a $\K$-scheme, so \eq{sa5eq6} does not make sense.

However, for general 1-morphisms $\phi:X\ra Y$ we can still define a
pullback morphism $\phi^*:\cM^\stk_Y\ra\cM^\stk_X$ as follows.
Proposition \ref{sa5prop2} gives a finite stratification
$X=\coprod_{i\in I}X_i$ with $X_i\simeq[S_i/\GL(n_i,\K)]$. Let
$\pi_i:S_i\ra X$ be as above, and define a group morphism
$\phi^*:\cM^\stk_Y\ra\cM^\stk_X$ by
\e
\smash{\phi^*:\ts\sum\limits_{j=1}^nc_j[R_j,\rho_j]\longmapsto
\sum\limits_{j=1}^nc_j
\sum\limits_{i\in I}[\GL(n_i,\K)]^{-1}\boxt
[R_j\t_{\rho_j,Y,\phi\ci\pi_i}S_i,\pi_X].}
\label{sa5eq22}
\e
If $\phi$ is representable in $\K$-schemes, this is the result of multiplying \eq{sa5eq6} by equation \eq{sa5eq21} for $1_X$, and so the two definitions of $\phi^*$ agree. As for $1_X,\bL$ one can
show that $\phi^*$ is independent of the choice of $I,X_i,S_i,n_i$,
and that pullbacks $\phi^*$ have the usual functoriality properties.
As in \cite[Th.~3.5]{Joyc1}, the analogue of \eq{sa5eq7} holds for
2-Cartesian squares in Artin $\K$-stacks.
\label{sa5def8}
\end{dfn}

\begin{dfn} Let $X$ be an Artin $\K$-stack. Consider triples
$(R,\rho,\hat r)$, where $R$ is a $\K$-scheme, $\rho:R\ra X$ a
1-morphism, and $\hat r:\hat\mu\t R\ra R$ a good $\hat\mu$-action on
$R$, in the sense of Definition \ref{sa5def2}. Call two such triples
$(R,\rho,\hat r),(R',\rho',\hat r')$ {\it equivalent\/} if there
exists a $\hat\mu$-equivariant isomorphism $\io:R\ra R'$ and a
2-isomorphism $\rho\cong\rho'\ci\io$. Write $[R,\rho,\hat r]$ for
the equivalence class of $(R,\rho,\hat r)$.

The {\it monodromic Grothendieck group} $K^{\hat\mu}_0(\Sch_X)$ is
the abelian group generated by such equivalence classes
$[R,\rho,\hat r]$, with relations (i),(ii) as in Definition
\ref{sa5def2}, except that we require a 2-isomorphism
$\rho_2\cong\rho_1\ci\pi$ rather than equality $\rho_2=\rho_1\ci\pi$
in (ii). Define a biadditive, commutative, associative product
`$\,\cdot\,$' on $K^{\hat\mu}_0(\Sch_X)$ as in \eq{sa5eq8}. As for
$K_0(\Sch_X)$ in Definition \ref{sa5def8}, this makes
$K_0^{\hat\mu}(\Sch_X)$ into a commutative ring, in general without
identity. If $X$ is a $\K$-scheme $K_0^{\hat\mu}(\Sch_X)$ is as in
Definition \ref{sa5def2}, with identity~$[X,\id_X,\hat\io]$.

For Artin $\K$-stacks $X,Y$, define a biadditive, commutative,
associative {\it external tensor product\/} $\boxt:
K^{\hat\mu}_0(\Sch_X) \!\t\! K^{\hat\mu}_0(\Sch_Y) \!\ra\!
K^{\hat\mu}_0(\Sch_{X\t Y})$ by \eq{sa5eq9}. Taking $Y=\Spec\K$,
this makes $K_0^{\hat\mu}(\Sch_X)$ into a module
over~$K_0^{\hat\mu}(\Sch_\K)$.

As for \eq{sa5eq20}, using the $K_0^{\hat\mu}(\Sch_\K)$-module
structure on $K_0^{\hat\mu}(\Sch_X)$ define
\begin{equation*}
\smash{\cM_X^\stm=K_0^{\hat\mu}(\Sch_X)\ot_{K_0^{\hat\mu}(\Sch_\K)}
K_0^{\hat\mu}(\Sch_\K)\bigl[\bL^{-1},(\bL^k -1)^{-1},\; k=1,2,\ldots\bigr].}
\end{equation*}
The product `$\,\cdot\,$' descends to $\cM^\stm_X$. When $X=\Spec\K$
we write $\cM^\stm_\K$ instead of $\cM^\stm_X$. Using the data
$X_i,S_i,n_i$ of Proposition \ref{sa5prop2}, as in \eq{sa5eq21}
define elements $1_X,\bL\in\cM_X^\stm$ by
\e
\begin{split}
1_X&=\ts\sum_{i\in I}[\GL(n_i,\K)]^{-1}\boxt[S_i,\pi_i,\hat\io],\\
\bL&=\ts\sum_{i\in I}[\GL(n_i,\K)]^{-1}\boxt[\bA^1\t S_i,
\pi_i\ci\pi_{S_i},\hat\io].
\end{split}
\label{sa5eq23}
\e
These are independent of choices, and $1_X$ is the identity
in~$\cM^\stm_X$.

Let $\phi:X\ra Y$ be a 1-morphism of Artin $\K$-stacks. Define the
{\it pushforwards\/} $\phi_*: K_0^{\hat\mu}(\Sch_X)\ra
K_0^{\hat\mu}(\Sch_Y)$ and $\phi_*:\cM_X^\stm\ra \cM^\stm_Y$ by the
analogue of \eq{sa5eq5}. If $\phi$ is representable in $\K$-schemes we may also define {\it pullbacks\/} $\phi^*:K_0^{\hat\mu}(\Sch_Y)\ra
K_0^{\hat\mu}(\Sch_X)$ and $\phi^*:\cM^\stm_Y\ra\cM^\stm_X$ by the
analogue of \eq{sa5eq6}. If $\phi$ is not representable in $\K$-schemes, we can still define $\phi^*:\cM^\stm_Y\ra\cM^\stm_X$ by the analogue of \eq{sa5eq22}. Pushforwards and pullbacks have the usual
functoriality properties, and the analogue of \eq{sa5eq7} holds for
2-Cartesian squares in~$\Art_\K$.

As for \eq{sa5eq10}, there are natural morphisms of commutative rings
\begin{align*}
i_X:K_0(\Sch_X)&\longra K_0^{\hat\mu}(\Sch_X),& i_X:\cM^\stk_X
&\longra \cM^\stm_X, \\
\Pi_X:K_0^{\hat\mu}(\Sch_X)&\longra K_0(\Sch_X),& \Pi_X:
\cM^\stm_X &\longra \cM^\stk_X,
\end{align*}
given by $i_X:[R,\rho]\mapsto[R,\rho,\hat\io]$ and
$\Pi_X:[R,\rho,\hat r]\mapsto[R,\rho]$. If $X$ is a $\K$-scheme, there is a natural projection $\cM_X^{\hat\mu}\ra\cM_X^\stm$.
\label{sa5def9}
\end{dfn}

The analogue of Definition \ref{sa5def3}, defining another
associative, commutative product `$\od$' on
$K_0^{\hat\mu}(\Sch_X),\cM_X^\stm$ and an external version `$\bd$',
works essentially without change. For the analogue of Definition
\ref{sa5def4}, following \eq{sa5eq13} we define $\bL^{1/2}$ in
$\cM_X^\stm$ only by $\bL^{1/2}=1_X \bd \bL^{1/2}_{\K}\in\cM_X^\stm$,  where $1_X$ is as in \eq{sa5eq23}, and
$\bL^{1/2}_{\K}\in\cM_\K^\stm$ as in \eq{sa5eq12}. Then
$\bL^{1/2}\od \bL^{1/2}=\bL$ in $\cM_X^\stm$, and we can define
$\bL^{n/2}$ in $\cM^\stm_X$ for all $n\in \Z$ in the obvious way.

Here is the stack analogue of Definition~\ref{sa5def5}:

\begin{dfn} For each Artin $\K$-stack $Y$, define $I_Y^\stm$ to be
the ideal in the commutative ring $\bigl(\cM_Y^\stm,\od\bigr)$
generated by elements
$\phi_*\bigl(\Up^\stk(P\ot_{\Z/2\Z}Q)-\Up(P)^\stk\od\Up^\stk(Q)\bigr)$
for all 1-morphisms $\phi:X\ra Y$ with $X$ a $\K$-scheme and
principal $\Z/2\Z$-bundles $P,Q\ra X$, where
$\Up^\stk(P),\Up^\stk(Q),\Up^\stk(P\ot_{\Z/2\Z}Q)$ are the images in
$\cM_X^\stm$ of the elements $\Up(P),\Up(Q),\Up(P\ot_{\Z/2\Z}Q)$ in
$\cM_X^{\hat\mu}$ from Definition \ref{sa5def5}. Define
$\oM_Y^\stm=\cM_Y^\stm/I_Y^\stm$ to be the quotient, as a
commutative ring with multiplication `$\od$', with
projection~$\Pi_Y^{\hat\mu}:\cM_Y^{\hat\mu}\ra\oM_Y^{\hat\mu}$.

The second multiplication `$\,\cdot\,$', external product $\boxt$,
and projection $\Pi_Y:\cM_Y^\stm\ra\cM^\stk_Y$ on $\cM_Y^\stm$ do
not descend to $\oM_Y^\stm$. The other structures $\od,\ab\bd,\ab
1_Y,\ab\bL,\ab\phi_*,\phi^*,i_Y,\bL^{1/2}$ do descend to
$\oM_Y^\stm$. If $X$ is a $\K$-scheme, we have a natural projection
$\oM^{\hat\mu}_X\ra\oM^\stm_X$. So in particular, the motives
$MF_{X,s}\in\oM^{\hat\mu}_X$ in Theorem \ref{sa5thm1} also make
sense in $\oM^\stm_X$. We will use this in Theorem~\ref{sa5thm2}.
\label{sa5def10}
\end{dfn}

\subsection{The main result}
\label{sa54}

Here is the main result of this section, the analogue of Theorem
\ref{sa5thm1} from \cite{BJM}. The proof uses our previous results from \cite{BJM,Joyc2}, the theory of rings of motives $\oM^\stm_X$ on Artin stacks $X$ from \S\ref{sa53}, and two new ingredients: Proposition \ref{sa5prop1}, which says that the motives $MF_{X,s}$ from \cite[Th.~5.10]{BJM} pull back as one would expect under smooth morphisms of d-critical loci, and Proposition \ref{sa5prop3}, which is a cunning trick to get round the fact that motives {\it do not have descent in the smooth topology}, that is, we do not have a motivic analogue of Theorem~\ref{sa4thm1}.

\begin{thm} Let\/ $(X,s)$ be an oriented d-critical stack, with
orientation $K_{X,s}^{1/2},$ where $X$ is assumed of finite type and
locally a global quotient. Then there exists a unique motive
$MF_{X,s}\in\oM^\stm_X$ such that if\/ $T$ is a finite type
$\K$-scheme and\/ $t:T\ra X$ is smooth of relative dimension $n,$ so
that\/ $(T,s(T,t))$ is an algebraic d-critical locus over $\K$ with
natural orientation $K_{T,s(T,t)}^{1/2}$ as in Lemma\/
{\rm\ref{sa3lem1},} then
\e
t^*\bigl(MF_{X,s}\bigr)=\bL^{n/2}\od MF_{T,s(T,t)}\quad\text{in\/ $\oM_T^\stm,$}
\label{sa5eq24}
\e
where $MF_{T,s(T,t)}\in\oM^\stm_T$ is as in Theorem\/
{\rm\ref{sa5thm1},} projected from $\oM^{\hat\mu}_T$ in
{\rm\S\ref{sa51}} to $\oM^\stm_T$ in {\rm\S\ref{sa53},} and\/
$t^*:\oM^\stm_X\ra\oM^\stm_T$ is the pullback.
\label{sa5thm2}
\end{thm}

We discuss how to relax the assumptions in Theorem \ref{sa5thm2}
that $X$ is of {\it finite type}, and {\it locally a global
quotient}.

\begin{rem}{\bf(a)} Let $X$ be an Artin $\K$-stack locally of
finite type (but not necessarily of finite type), with affine
geometric stabilizers. Then one can define motivic rings
$K_0(\Sch_X),\ab\cM^\stk_X,\ab
K_0^{\hat\mu}(\Sch_X),\ab \cM^\stm_X,\oM^\stm_X$ generalizing those
in \S\ref{sa53}, using the idea of `local stack functions'
$\mathop{\rm LSF}(X)$ from Joyce~\cite[Def.~3.9]{Joyc1}.

Elements of $K_0(\Sch_X)$ are $\sim$-equivalence classes of sums
$\sum_{i\in I}c_i[R_i,\rho_i]$ for $I$ a possibly infinite indexing
set, $R_i$ a $\K$-scheme locally of finite type, $\rho_i:R_i\ra X$ a
finite type 1-morphism, and $c_i\in\Z$ for $i\in I$, such that for
any finite type $\K$-substack $Y\subseteq X$, we have
$R_i\t_XY\ne\es$ for only finitely many $i\in I$. We set $\sum_{i\in
I}c_i[R_i,\rho_i]\sim\sum_{j\in J}d_j[S_j,\si_j]$ if for all finite
type $Y\subseteq X$, we have $\sum_{i\in I}c_i[R_i\t_{\rho_i,X,{\rm
inc}}Y,\pi_Y]=\sum_{j\in J}d_j[S_j\t_{\si_j,X,{\rm inc}}Y,\pi_Y]$ in
$K_0(\Sch_Y)$, where $K_0(\Sch_Y)$ is as in \S\ref{sa53} as $Y$ is
of finite type.

Then pushforwards $\phi_*$ on $K_0(\Sch_X),\cM^\stk_X,\ldots$ can be
defined only if $\phi:X\ra Y$ is a finite type 1-morphism, but
pullbacks $\phi^*$ can be defined for arbitrary $\phi$ (requiring
$\phi$ representable in $\K$-schemes for~$K_0(\Sch_X),K_0^{\hat\mu}(\Sch_X)$).

As discussed in \cite[Rem.~5.11]{BJM} for $\K$-schemes, it is now easy to generalize Theorem \ref{sa5thm2} to d-critical stacks $(X,s)$ which
are locally of finite type rather than of finite type, giving a
unique $MF_{X,s}\in\oM^\stm_X$ satisfying \eq{sa5eq24}, where it is enough to consider only finite type $\K$-schemes $T$. Note that we
cannot push $MF_{X,s}$ forward to $\oM^\stm_\K$ if $X$ is not of
finite type, since $\pi:X\ra\Spec\K$ is not a finite type
1-morphism, and $\pi_*:\oM^\stm_X\ra\oM^\stm_\K$ is not defined.
\smallskip

\noindent{\bf(b)} The assumption in Theorem \ref{sa5thm2} that $X$
is locally a global quotient is used to prove Proposition
\ref{sa5prop3} in \S\ref{sa55}. We would have preferred to make the
weaker assumption that $X$ has affine geometric stabilizers.

The issue is this: we want to characterize $MF_{X,s}\in\oM^\stm_X$
by prescribing $t^*(MF_{X,s})\in\smash{\oM^\stm_T}$ whenever $T$ is
a $\K$-scheme and $t:T\ra X$ is a smooth 1-morphism. However, if $X$
is not locally a global quotient, it seems conceivable this may not
determine $MF_{X,s}$ uniquely, as there might exist $0\ne
M\in\oM^\stm_X$ with $t^*(M)=0$ for all such~$t:T\ra X$.

One way to fix this might be to expand our whole set-up to include a
suitable class of formal schemes, and then prescribe $t^*(MF_{X,s})$
when $T$ is a formal scheme and $t:T\ra X$ a smooth 1-morphism. If
$X$ has affine geometric stabilizers, there should be enough such
$t:T\ra X$ to determine $MF_{X,s}$ uniquely.

\label{sa5rem1}
\end{rem}

Combining Theorems \ref{sa2thm6}, \ref{sa3thm6}, \ref{sa5thm2} and
Corollary \ref{sa3cor2}, and noting as in \S\ref{sa51} that moduli
stacks of coherent sheaves are locally global quotients, yields:

\begin{cor} Let\/ $(\bX,\om)$ be a $-1$-shifted symplectic derived
Artin $\K$-stack in the sense of Pantev et al.\ {\rm\cite{PTVV},}
and\/ $X=t_0(\bX)$ the associated classical Artin\/ $\K$-stack,
assumed of finite type and locally a global quotient. Suppose we are
given a square root\/ $\det(\bL_\bX)\vert_X^{1/2}$ for
$\det(\bL_\bX) \vert_X$. Then we may define a natural motive
$MF_{\bX,\om}\in\oM^\stm_X,$ which is characterized
by the fact that given a diagram
\begin{equation*}
\smash{\xymatrix@C=60pt{ \bU=\bs\Crit(f:U\ra\bA^1) & \bV \ar[l]_(0.3){\bs i} \ar[r]^{\bs\vp} & \bX }}
\end{equation*}
such that\/ $U$ is a smooth\/ $\K$-scheme, $\bs\vp$ is smooth of
dimension $n,$ $\bL_{\bV/\bU} \simeq \bT_{\bV/\bX}[2],$
$\bs\vp^*(\om_\bX)\sim \bs i^*(\om_\bU)$ for $\om_\bU$ the natural\/
$-1$-shifted symplectic structure on $\bU=\bs\Crit(f:U\ra\bA^1),$
and\/ $\vp^*(\det(\bL_\bX)\vert_X^{1/2})\cong
i^*(K_U)\ot\La^n\bT_{\bV/\bX},$ then~$\vp^*(MF_{\bX,\om})=\bL^{n/2}\od i^*(MF^{{\rm mot},\phi}_{U,f})$ in $\oM^\stm_V$.
\label{sa5cor1}
\end{cor}

\begin{cor} Let\/ $Y$ be a Calabi--Yau\/ $3$-fold over\/ $\K,$
and\/ $\cM$ a finite type classical moduli\/ $\K$-stack of coherent
sheaves in $\coh(Y),$ with natural obstruction theory\/
$\phi:\cE^\bu\ra\bL_\cM$. Suppose we are given a square root\/
$\det(\cE^\bu)^{1/2}$ for $\det(\cE^\bu)$. Then we may define a
natural motive $MF_\cM\in\oM^\stm_\cM$.
\label{sa5cor2}
\end{cor}

Corollary \ref{sa5cor2} is relevant to Kontsevich and Soibelman's
theory of {\it motivic Donaldson--Thomas invariants\/} \cite{KoSo1}.
Our square root $\det(\cE^\bu)^{1/2}$ roughly coincides with their
{\it orientation data\/} \cite[\S 5]{KoSo1}. In \cite[\S
6.2]{KoSo1}, given a finite type moduli stack $\cM$ of coherent
sheaves on a Calabi--Yau 3-fold $Y$ with orientation data, they
define a motive $\int_\cM 1$ in a ring $D^\mu$ isomorphic to our
$\oM^\stm_\K$. We expect this should agree with $\pi_*(MF_\cM)$ in
our notation, with $\pi:\cM\ra\Spec\K$ the projection. This
$\int_\cM 1$ is roughly the motivic Donaldson--Thomas invariant of
$\cM$. Their construction involves expressing $\cM$ near each point
in terms of the critical locus of a formal power series. Kontsevich
and Soibelman's constructions were partly conjectural, and our
results may fill some gaps in their theory.

\begin{ex} As in \cite[Def.~2.1]{Joyc1}, an algebraic $\K$-group
$G$ is called {\it special\/} if every \'etale locally trivial
principal $G$-bundle over a $\K$-scheme is Zariski locally trivial.
Any special $\K$-group can be embedded as a closed $\K$-subgroup
$G\subseteq\GL(n,\K)$, and then $\GL(n,\K)\ra \GL(n,\K)/G$ is a
Zariski locally trivial principal $G$-bundle, so taking motives in
$\cM^\stk_\K$ gives $[\GL(n,\K)]=[G]\cdot[\GL(n,\K)/G]$. Hence $[G]$
is invertible in $\cM^\stk_\K$, with $[G]^{-1}=[\GL(n,\K)/G]\cdot
[\GL(n,\K)]^{-1}$.

Some examples of special $\K$-groups are
$\bG_m,\GL(n,\K),\mathop{\rm SL}(n,\K), \mathop{\rm Sp}(2n,\K)$, and
the group of invertible elements $A^\t$ of any finite-dimensional
$\K$-algebra $A$. Products of special groups are special. Special
$\K$-groups are always affine and connected, so nontrivial finite
groups are not special.

Suppose a special $\K$-group $G$ of dimension $n$ acts on a finite
type, oriented algebraic d-critical locus $(T,s')$ over $\K$
preserving $s'\in H^0(\cSz_T)$ and the orientation $K_{T,s'}^{1/2}$.
Write $X=[T/G]$ for the quotient stack and $t:T\ra X$ for the
projection. Then $s'$ descends to a unique d-critical structure $s$
on $X$ with $s'=s(T,t)$ as in Example \ref{sa3ex}, and using Theorem
\ref{sa3thm5} we also find that the orientation $K_{T,s'}^{1/2}$
descends to a unique orientation $K_{X,s}^{1/2}$ on the d-critical
stack $(X,s)$ with~$K_{X,s}^{1/2}(T^\red,t^\red)\cong
K_{T,s'}^{1/2}\ot \bigl(\La^{\rm
top}T^*_{T/X}\bigr)\big\vert{}_{T^\red}^{\ot^{-1}}$.

Theorem \ref{sa5thm2} gives $MF_{X,s}\in\oM^\stm_X$ with
$t^*\bigl(MF_{X,s}\bigr)=\bL^{n/2}\od MF_{T,s'}$ in $\oM^\stm_T$.
Applying $t_*$ and using $t_*\ci t^*(M)=[T,t,\hat\io]\od M$ for
$M\in \oM^\stm_X$ gives
\e
\smash{MF_{X,s}\od [T,t,\hat\io]=\bL^{n/2}\od t_*(MF_{T,s'}).}
\label{sa5eq25}
\e
Now $t:T\ra X$ is a principal $G$-bundle, and so Zariski locally
trivial as $G$ is special. Therefore $[T,t,\hat\io]=[G,\hat\io]\bd
1_X$, where $[G,\hat\io]=i_\K([G])\in\cM^\stm_\K$. As $[G]$ is
invertible, so is $[G,\hat\io]$. Thus multiplying \eq{sa5eq25} by
$[G,\hat\io]^{-1}$ gives
\begin{equation*}
\smash{MF_{X,s}=[G,\hat\io]^{-1}\bd\bigl(\bL^{n/2}\od t_*(MF_{T,s'})\bigr).}
\end{equation*}\vskip -10pt

\label{sa5ex}
\end{ex}

\subsection{Proof of Theorem \ref{sa5thm2}}
\label{sa55}

We begin with the following result, related to
Proposition~\ref{sa5prop2}.

\begin{prop} Let\/ $X$ be a finite type Artin $\K$-stack which is
locally a global quotient. Then we can find a stratification
$X=\coprod_{j\in J}X_j,$ for $J$ a finite set and\/ $X_j\subseteq X$
a locally closed\/ $\K$-substack, and\/ $1$-morphisms $\phi_j:S_j\ra
X$ smooth of relative dimension $n_j$ with\/ $S_j$ a $\K$-scheme
such that\/ $[S_j\t_X X_j,\pi_{X_j}]$ is an invertible element of\/
$\cM^\stk_{X_j}$ for all\/~$j\in J$.
\label{sa5prop3}
\end{prop}

\begin{proof} As $X$ is finite type and locally a global quotient,
there exist Zariski open $\K$-substacks $Y_j\subseteq X$ and
equivalences $Y_j\simeq[S_j/\GL(n_j,\K)]$ for $j=1,\ldots,m$, where
$S_j$ is a $\K$-scheme with a $\GL(n_j,\K)$-action, such that
$X=Y_1\cup\cdots\cup Y_m$. Define $\phi_j:S_j\ra X$ to be the
composition $S_j\ra[S_j/\GL(n_j,\K)]\,{\buildrel\sim\over\longra}\,
Y_j\hookra X$. For $j=1,\ldots,m$, define a locally closed
$\K$-substack $X_j\subseteq X$ by $X_j=Y_j\sm(Y_1\cup\cdots\cup
Y_{j-1})$. Set $J=\{1,\ldots,m\}$. Then $X=\coprod_{j\in J}X_j$
as~$X=Y_1\cup\cdots\cup Y_m$.

Since $X_j\subseteq Y_j$ and $\phi_j:S_j\ra Y_j$ is a principal
$\GL(n_j,\K)$-bundle, we see that $\pi_{X_j}:S_j\t_X X_j\ra X_j$ is
a principal $\GL(n_j,\K)$-bundle, which is automatically Zariski
locally trivial. Hence $[S_j\t_X X_j,\pi_{X_j}]=[\GL(n_i,\K)]\bd
1_{X_j}$, which is invertible in $\cM^\stk_{X_j}$ with inverse
$[\GL(n_i,\K)]^{-1}\bd 1_{X_j}$.
\end{proof}

We now prove Theorem \ref{sa5thm2}. Suppose first that there exists
$MF_{X,s}\in\oM^\stm_X$ such that \eq{sa5eq24} holds for all $t:T\ra
X$ smooth of dimension $n$ with $T$ a $\K$-scheme. Let
$J,X_j,S_j,\phi_j,n_j$ be as in Proposition \ref{sa5prop3}, and
write $\io_j:X_j\hookra X$ for the inclusion. Then we have
\ea
&MF_{X,s}=\ts\sum_{j\in J}(\io_j)_*\bigl(\io_j^*(MF_{X,s})\bigr)
\nonumber\\
&\;=\ts\sum_{j\in J}(\io_j)_*\bigl([S_j\t_X
X_j,\pi_{X_j},\hat\io]^{-1}\od [S_j\t_X X_j,\pi_{X_j},\hat\io]\od
\io_j^*(MF_{X,s})\bigr)
\nonumber\\
&\;=\ts\sum_{j\in J}(\io_j)_*\bigl([S_j\t_X
X_j,\pi_{X_j},\hat\io]^{-1}\od\io_j^*\bigl([S_j,\phi_j,\hat\io]\od
MF_{X,s})\bigr)\bigr)
\label{sa5eq26}\\
&\;=\ts\sum_{j\in J}(\io_j)_*\bigl([S_j\t_X
X_j,\pi_{X_j},\hat\io]^{-1}\bigr)\od\bigl((\phi_j)_*\ci\phi_j^*(MF_{X,s})\bigr)
\nonumber\\
&\;=\ts\sum_{j\in J}(\io_j)_*\bigl([S_j\t_X
X_j,\pi_{X_j},\hat\io]^{-1}\bigr)\od\bigl((\phi_j)_*\bigl(\bL^{n_j/2}\od
MF_{S_j,s(S_j,\phi_j)}\bigr)\bigr),
\nonumber
\ea
using $X=\coprod_{j\in J}X_j$ in the first step, $[S_j\t_X
X_j,\pi_{X_j}]$ invertible in $\cM^\stk_{X_j}$ so that
$[S_j\t_X X_j,\pi_{X_j},\hat\io]=i_{X_j}([S_j\t_X X_j,\pi_{X_j}])$
is invertible in $\oM^\stm_{X_j}$ in the second, $[S_j\t_X
X_j,\pi_{X_j},\hat\io]\!=\!\io_j^*([S_j,\phi_j,\hat\io])$ and $\io_j^*$
multiplicative for $\od$ in the third,
$[S_j,\phi_j,\hat\io]\od\!=\!(\phi_j)_*\!\ci\!\phi_j^*$ and $(\io_j)_*(M\!\od\!\io_j^*(N))\!=\!((\io_j)_*\!\ci\!\io_j^*(M))\od N$ in the fourth, and
\eq{sa5eq24} with $S_j,\phi_j,n_j$ in place of $T,t,n$ in the fifth.
Equation \eq{sa5eq26} proves $MF_{X,s}$ in Theorem \ref{sa5thm2} is
unique if it exists, and gives a formula for it.
\enlargethispage{5pt}

Now define $MF_{X,s}$ to be the bottom line of \eq{sa5eq26}. Suppose
$t:T\ra X$ is smooth of dimension $n$, with $T$ a $\K$-scheme. Define $T_j=X_j\t_{\io_j,X,t}T\subseteq T$ and $U_j=S_j\t_{\phi_j,X,t}T$ for each $j\in J$. Then $T_j,U_j$ are $\K$-schemes as $X_j\hookra X$ and $S_j\ra X$ are representable in $\K$-schemes, and we have 2-Cartesian
squares
\e
\begin{gathered}
\xymatrix@C=70pt@R=11pt{ *+[r]{T_j} \ar[r]_(0.35){\pi_T} \ar[d]^{\pi_{X_j}}
\drtwocell_{}\omit^{}\omit{^{}} & *+[l]{T} \ar[d]_t \\
*+[r]{X_j} \ar[r]^(0.65){\io_j} & *+[l]{X,\!\!} }\qquad
\xymatrix@C=70pt@R=11pt{ *+[r]{U_j} \ar[r]_(0.35){\Pi_T} \ar[d]^{\Pi_{S_j}}
\drtwocell_{}\omit^{}\omit{^{}} & *+[l]{T} \ar[d]_t \\
*+[r]{S_j} \ar[r]^(0.65){\phi_j} & *+[l]{X.\!\!} }
\end{gathered}
\label{sa5eq27}
\e
Then\begin{small}\begin{align} &t^*\bigl(MF_{X,s}\bigr)=\ts\sum_{j\in
J}\! t^*\!\ci \!(\io_j)_*\bigl([S_j\!\t_X\!
X_j,\pi_{X_j},\hat\io]^{-1}\bigr)\!\od
t^*\!\ci\!(\phi_j)_*\bigl(\bL^{n_j/2}\!\od\!
MF_{S_j,s(S_j,\phi_j)}\bigr)
\nonumber\\
&=\ts\sum_{j\in J} (\pi_T)_*\ci\pi_{X_j}^* \bigl([S_j\t_X
X_j,\pi_{X_j},\hat\io]^{-1}\bigr)\od
(\Pi_T)_*\ci\Pi_{S_j}^*\bigl(\bL^{n_j/2}\od
MF_{S_j,s(S_j,\phi_j)}\bigr)
\nonumber\\
&=\ts\sum_{j\in J} (\pi_T)_*\bigl(\bigl(\pi_{X_j}^*([S_j\!\t_X\!
X_j,\pi_{X_j},\hat\io])\bigr){}^{-1}\bigr)\!\od\!
(\Pi_T)_*\bigl(\bL^{(n+n_j)/2}\!\od\!
MF_{U_j,s(U_j,\phi_j\ci\Pi_{S_j})}\bigr)
\nonumber\\
&=\ts\sum_{j\in J}(\pi_T)_*\bigl([S_j\t_X
X_j\t_{X_j}T_j,\pi_{T_j},\hat\io]^{-1} \bigr)\od (\Pi_T)_*\ci\Pi_T^*
\bigl(\bL^{n/2}\od MF_{T,s(T,t)}\bigr)
\nonumber\\
&=\ts\sum_{j\in J}
(\pi_T)_*\bigl([U_j\t_TT_j,\pi_{T_j},\hat\io]^{-1}\bigr)\od
[U_j,\Pi_T,\hat\io]\od\bL^{n/2}\od MF_{T,s(T,t)}
\nonumber\\
&=\ts\sum_{j\in J}(\pi_T)_*
\bigl(\pi_T^*([U_j,\Pi_T,\hat\io])^{-1}\od\pi_T^*
([U_j,\Pi_T,\hat\io])\bigr)\od  \bL^{n/2}\od MF_{T,s(T,t)}
\nonumber\\
&=\ts\sum_{j\in J}(\pi_T)_*\bigl(1_{T_j}\bigr)\od \bL^{n/2}\od
MF_{T,s(T,t)} =\bigl(\sum_{j\in
J}[T_j,\pi_T,\hat\io]\bigr)\od \bL^{n/2}\od
MF_{T,s(T,t)}
\nonumber\\
&=[T,\id_T,\hat\io]\od \bL^{n/2}\od MF_{T,s(T,t)}=\bL^{n/2}\od
MF_{T,s(T,t)},
\label{sa5eq28}
\end{align}\end{small}using \eq{sa5eq26} and $t^*$ multiplicative
for $\od$ in the first step, the analogue of \eq{sa5eq7} for the
2-Cartesian squares \eq{sa5eq27} in the second, that $\pi_{X_j}^*$
is a ring morphism for $\od$ and \eq{sa5eq17} for the morphism
$\Pi_{S_j}:(U_j,s(U_j,\phi_j\ci\Pi_{S_j}))\ra (S_j,s(S_j,\phi_j))$
of oriented d-critical loci which is smooth of dimension $n$ in the
third, the definition of $\pi_{X_j}^*$ and \eq{sa5eq17} for
$\Pi_T:(U_j,s(U_j,\phi_j\ci\Pi_{S_j}))\ra (T,s(T,t))$ smooth of
dimension $n_j$ in the fourth, $S_j\t_X X_j\t_{X_j}T_j\cong S_j\t_X
T_j=U_j\t_TT_j$ and $(\Pi_T)_*\ci\Pi_T^*=[U_j,\Pi_T,\hat\io]\od$ in the
fifth, $(\pi_T)_*(M)\od N=(\pi_T)_*(M\od\pi_T^*(N))$ in the sixth,
and $T=\coprod_jT_j$ in the ninth.

Equation \eq{sa5eq28} proves \eq{sa5eq24} for all $t:T\ra X$ smooth
of dimension $n$ with $T$ a $\K$-scheme, as we want, for $MF_{X,s}$
the bottom line of \eq{sa5eq26}. The argument of \eq{sa5eq26} shows
$MF_{X,s}$ is unique, and is in particular independent of the choice
of $J,X_j,S_j,\phi_j,n_j$ in Proposition \ref{sa5prop3}. This
completes the proof.

\medskip

\noindent Address for Christopher Brav:

\noindent Institute for Advanced Study, Einstein Drive, Princeton, NJ 08540, U.S.A.

\noindent E-mail: {\tt brav@ias.edu}.
\smallskip

\noindent Address for Oren Ben-Bassat:

\noindent Department of Mathematics, University of Haifa, Haifa, Israel.

\noindent E-mail: {\tt ben-bassat@math.haifa.ac.il}.
\smallskip
\noindent Address for Vittoria Bussi and Dominic Joyce:

\noindent The Mathematical Institute, Radcliffe Observatory Quarter, Woodstock Road, Oxford, OX2 6GG, U.K.

\noindent E-mails: {\tt bussi@maths.ox.ac.uk,

\noindent joyce@maths.ox.ac.uk.}


\begin{thebibliography}{99}
\addcontentsline{toc}{section}{References}

\bibitem{BBD} A.A. Beilinson, J. Bernstein, and P. Deligne, {\it
Faisceaux pervers}, Ast\'{e}risque 100, 1982.

\bibitem{BoGr} E. Bouaziz and I. Grojnowski, {\it A $d$-shifted
Darboux theorem}, \hfil\break arXiv:1309.2197, 2013.

\bibitem{BBDJS} C. Brav, V. Bussi, D. Dupont, D. Joyce, and B. Szendr\H oi, {\it Symmetries and stabilization for sheaves of vanishing cycles}, Journal of Singularities 11 (2015), 85--151. arXiv:1211.3259.

\bibitem{BBJ} C. Brav, V. Bussi and D. Joyce, {\it A Darboux
theorem for derived schemes with shifted symplectic structure},
arXiv:1305.6302, 2013.

\bibitem{BJM} V. Bussi, D. Joyce and S. Meinhardt, {\it On motivic
vanishing cycles of critical loci}, arXiv:1305.6428, 2013.

\bibitem{DeLo} J. Denef and F. Loeser, {\it Geometry on arc spaces
of algebraic varieties}, European Congress of Mathematics, Vol. I
(Barcelona, 2000), 327--348, Progr. Math. 201, Birkh\"auser, Basel,
2001. math.AG/0006050.

\bibitem{Dimc} A. Dimca, {\it Sheaves in Topology}, Universitext,
Springer-Verlag, Berlin, 2004.

\bibitem{Eked} T. Ekedahl, {\it On the adic formalism}, pages
197--218 in {\it The Grothendieck Festschrift, vol. II}, Progr.
Math. 87, Birkh\"auser, Boston, 1990.

\bibitem{FrKi} E. Freitag and R. Kiehl, {\it Etale cohomology and
the Weil Conjecture}, Ergeb. der Math. und ihrer Grenzgebiete 13,
Springer-Verlag, 1988.

\bibitem{GaRo} D. Gaitsgory and N. Rozenblyum, {\it Crystals and
$\cD$-modules}, \hfil\break  arXiv:1111.2087, 2011.

\bibitem{GeMa} S.I. Gelfand and Y.I. Manin, {\it Methods of
Homological Algebra}, second edition, Springer-Verlag, Berlin, 2003.

\bibitem{Hart} R. Hartshorne, {\it Algebraic Geometry}, Graduate
Texts in Math. 52, Springer, New York, 1977.

\bibitem{Hiro} H. Hironaka, {\it Resolution of singularities of an
algebraic variety over a field of characteristic zero I, II}, Ann.
Math. 79 (1964), 109--203 and 205--326.

\bibitem{Joyc1} D. Joyce, {\it Motivic invariants of Artin stacks
and `stack functions'}, Quart. J. Math. 58 (2007), 345--392.
math.AG/0509722.

\bibitem{Joyc2} D. Joyce, {\it A classical model for derived
critical loci}, to appear in Journal of Differential Geometry, 2014. arXiv:1304.4508.

\bibitem{JoSo} D. Joyce, Y. Song, {\it A theory of generalized
Donaldson--Thomas invariants}, Mem. Amer. Math. Soc. 217 (2012), no.
1020. arXiv:0810.5645.

\bibitem{KiWe} R. Kiehl and R. Weissauer, {\it Weil Conjectures,
perverse sheaves and l'adic Fourier transform}, Springer-Verlag,
2001.

\bibitem{KoSo1} M. Kontsevich and Y. Soibelman, {\it Stability
structures, motivic Donaldson--Thomas invariants and cluster
transformations}, arXiv:0811.2435, 2008.

\bibitem{KoSo2} M. Kontsevich and Y. Soibelman, {\it Cohomological
Hall algebra, exponential Hodge structures and motivic
Donaldson--Thomas invariants}, Commun. Number Theory Phys. 5 (2011),
231--352. arXiv:1006.2706.

\bibitem{Kres} A. Kresch, {\it Cycle groups for Artin stacks},
Invent. Math. 138 (1999), 495--536. math.AG/9810166.

\bibitem{LaMo} G. Laumon and L. Moret-Bailly, {\it Champs
alg\'ebriques}, Ergeb. der Math. und ihrer Grenzgebiete 39,
Springer-Verlag, Berlin, 2000.

\bibitem{LaOl1} Y. Laszlo and M. Olsson, {\it The six operations
for sheaves on Artin stacks. I\/} Publ. Math. I.H.E.S. 107 (2008),
109--168. math.AG/0512097.

\bibitem{LaOl2} Y. Laszlo and M. Olsson, {\it The six operations
for sheaves on Artin stacks. II}, Publ. Math. I.H.E.S. 107 (2008),
169--210. math.AG/0603680.

\bibitem{LaOl3} Y. Laszlo and M. Olsson, {\it Perverse
$t$-structure on Artin stacks}, Math. Z. 261 (2009), 737--748.
math.AG/0606175.

\bibitem{LiZh1} Y. Liu and W. Zheng, {\it Enhanced six operations
and base change theorem for sheaves on Artin stacks},
arXiv:1211.5948, 2012.

\bibitem{LiZh2} Y. Liu and W. Zheng, {\it Enhanced adic formalism,
biduality, and perverse $t$-structures for higher Artin stacks},
preprint, 2012.

\bibitem{Looi} E. Looijenga, {\it Motivic measures}, S\'eminaire
Bourbaki, Vol. 1999/2000. Ast\'erisque 276 (2002), 267--297.
math.AG/0006220.

\bibitem{Mass3} D. Massey, {\it Natural commuting of vanishing
cycles and the Verdier dual}, arXiv:0908.2799, 2009.

\bibitem{Olss} M. Olsson, {\it Sheaves on Artin stacks}, J. Reine
Angew. Math. 603 (2007), 55--112.

\bibitem{PTVV} T. Pantev, B. To\"en, M. Vaqui\'e and G. Vezzosi,
{\it Shifted symplectic structures}, Publ. Math. I.H.E.S. 117
(2013), 271--328. arXiv:1111.3209.

\bibitem{Paul} A.G.M. Paulin, {\it The Riemann--Hilbert
correspondence for algebraic stacks}, arXiv:1308.5890, 2013.

\bibitem{Toen} B. To\"en, {\it Higher and derived stacks: a global
overview}, pages 435--487 in {\it Algebraic Geometry --- Seattle
2005}, Proc. Symp. Pure Math. 80, Part 1, A.M.S., Providence, RI,
2009. math.AG/0604504.

\bibitem{ToVe} B. To\"en and G. Vezzosi, {\it Homotopical Algebraic
Geometry II: Geometric Stacks and Applications}, Mem. Amer. Math.
Soc. 193 (2008), no. 902. math.AG/0404373.

\end{thebibliography}
\end{document}